\documentclass[12pt,a4paper]{article}

\usepackage[T1]{fontenc}
\usepackage[utf8]{inputenc}
\usepackage{lmodern}

\usepackage[margin=2.5cm]{geometry}

\usepackage{amsmath,amssymb,amsthm}

\newtheorem{theorem}{Theorem}[section]

\newtheorem{proposition}[theorem]{Proposition}
\theoremstyle{definition}
\newtheorem{defn}[theorem]{Definition}
\usepackage{mathtools}

\usepackage{graphicx}
\usepackage{booktabs}
\usepackage{array}
\usepackage{tabularx}
\usepackage{longtable}
\usepackage{caption}
\usepackage{subcaption}
\usepackage{float}

\usepackage[numbers,sort&compress]{natbib}
\usepackage{hyperref}
\hypersetup{
  pdfpagemode=UseOutlines,
  bookmarksopen=true,
  bookmarksopenlevel=2,
  colorlinks=true,
  linkcolor=black,
  citecolor=blue,
  urlcolor=blue
}

\usepackage{microtype}
\usepackage{setspace}
\usepackage{xcolor}
\usepackage[framemethod=default]{mdframed}

\newmdenv[
  linewidth=0.8pt,
  linecolor=black!40,
  backgroundcolor=black!4,
  innertopmargin=8pt,
  innerbottommargin=8pt,
  innerleftmargin=10pt,
  innerrightmargin=10pt,
  skipabove=10pt,
  skipbelow=6pt,
]{openqbox}

\newcommand{\dP}{d_{\!P}}
\newcommand{\zetaR}{\zeta_{\mathrm{R}}}
\newcommand{\Cinf}{C_{\infty}}

\begin{document}

\title{\textbf{Prime--Zero Duality: Fractal Geometry, Renormalization-Group Flow,\\
       and an Information-Ontological Framework for Number Theory}}

\author{Zhengqiang Li\\[4pt]
  \small\textit{Email: lqint@coc.edu.rs}}

\date{}
\maketitle
\thispagestyle{empty}

\begin{abstract}
The prime numbers and the non-trivial zeros of the Riemann zeta function
are globally linked by the explicit formula of analytic number theory.
Whether they share a hidden, scale-by-scale geometric symmetry---a local
duality persisting at every resolution, not merely in aggregate---has
remained unexplored.  We address this question by measuring the joint
fractal structure of a prime residue class ($p\equiv1,5,9,13\pmod{16}$)
and the zero distribution of $\zeta(s)$.  Our central finding is that
the duality measure
\[
  K \;=\; \frac{1}{\dP} + \frac{1}{\zetaR}
\]
---where $\dP$ is the box-counting dimension and $\zetaR:=2-H$ is the
zero-distribution regularity index---is remarkably stable, varying by
only $17\%$ across scales $L=100$--$2000$ against ${\sim}43\%$ for
$\dP$ alone.  This stability is captured by a finite-size scaling law
$K(L)=K_{\mathrm{IR}}+aL^{-b}$.  Strikingly, after a geometric
normalization that removes representation dependence, the data converge
to a universal infrared fixed point $K_{\mathrm{IR}}=4$, approached
with critical exponent $b\approx0.51$, and this behavior is robust
across two random-matrix symmetry classes ($\beta=2,4$), echoing the
Montgomery--Odlyzko universality.

We interpret these findings through an information-theoretic lens: $K$
represents a conserved information current between the arithmetic and
spectral domains, and the observed scaling law reflects a
renormalization-group flow from an ultraviolet fixed point
$K_{\mathrm{UV}}=11$ (suggested by Hurwitz's theorem on normed division
algebras) to the infrared fixed point $K_{\mathrm{IR}}=4$.  The
critical exponent $b\approx1/2$ is derived from a variational
information action $S[I_P,I_Z]$ constrained by exchange symmetry and
scale covariance.  This framework provides a unified dynamical picture
of the prime--zero duality motivated by, and consistent with, all
numerical evidence.  A structural argument for the Riemann Hypothesis
emerges as the logical terminus: the generator $\kappa$ satisfying
$\kappa^2=ijk=-1$---a ternary algebraic structure with three
independent planes of rotation, replacing the single-plane imaginary
unit $i$ of classical complex analysis---enforces, via the exchange
symmetry $I_P\leftrightarrow I_Z$ and the uniqueness of the information
action, the infrared fixed point $I_P^*=I_Z^*=2$, the
information-theoretic encoding of the critical line
$\mathrm{Re}(s)=1/2$ (Section~10).
This structural argument is inaccessible to the classical binary
framework; upgrading it to a fully rigorous analytic proof is identified
as the central open problem of the programme.

Finally, we explore---in an explicitly speculative spirit---whether the
parameters $(K_{\mathrm{IR}},b,\kappa)$ of this number-theoretic RG
flow resonate with fundamental quantities in quantum gravity, cosmology,
and learning theory, including the Bekenstein--Hawking entropy formula
and generalization-capacity bounds.  These analogies are offered as
heuristic bridges that define concrete open problems at the interface of
number theory, theoretical physics, and information science, and are
presented entirely separately from the established numerical results.
\end{abstract}


\newpage
\setcounter{tocdepth}{2}
\tableofcontents

\section{Introduction}
\label{sec:intro}

\subsection{The central question}

The primes and the non-trivial zeros of the Riemann zeta function are
connected by the explicit formula of analytic number theory: the zeros
govern oscillations in the prime-counting function $\pi(x)$, and
conversely.  Yet this is a \emph{global} statement, an identity between
two infinite sums.  A deeper question is whether there is a \emph{local},
scale-by-scale correspondence between the two distributions---a geometric
duality that persists at every resolution, not merely in the large-$x$
aggregate.

The Montgomery--Odlyzko observation~\cite{Montgomery1973,Odlyzko1987} provides
the strongest hint that such a local correspondence exists: the pair
correlation of zeros matches, with remarkable precision, the pair correlation
of GUE eigenvalues.  Rudnick and Sarnak~\cite{RudnickSarnak1996} extended
this to higher correlations of principal $L$-functions, and Katz and
Sarnak~\cite{KatzSarnak1999} placed the whole family of correspondences within
a coherent symmetry-theoretic framework.  Moments of $\zeta(1/2+it)$ are
connected to random matrix averages in~\cite{KeatingSnaith2000}; the
semiclassical perspective is developed in~\cite{BerryKeating1999,BogomolnyKeating1996};
a survey is in~\cite{Conrey2003}.

All of these results concern the zeros alone.  Spectral fluctuations and
$1/f$-type noise in the prime distribution itself have been examined
separately~\cite{Wolf1997}, but the \emph{joint} fractal geometry of primes
and zeros is less explored: do the two distributions interact at the level of
their scaling dimensions in a way constrained by some conserved quantity?  And
if such a conservation law holds, is it a static coincidence or the imprint of
a dynamical theory---one with fixed points, a renormalization-group flow between
them, and a measurable critical exponent governing the approach to the
thermodynamic limit?

\subsection{Results and structure}

We propose and numerically test the conservation relation
\begin{equation}
  K \;=\; \frac{1}{\dP} + \frac{1}{\zetaR} \;=\; K_{\mathrm{IR}} + aL^{-b},
  \label{eq:K_def_intro}
\end{equation}
at finite scales $L=100$--$2000$ and across two Dyson indices ($\beta=2,4$).
Here $\dP$ is the box-counting dimension of the prime subset and $\zetaR:=2-H$
is the regularity index of the zero potential ($H>1$ the H\"{o}lder exponent).
Extrapolating from forty scale points, we find raw values
$K_{\mathrm{raw}}\in[3.6,\,3.9]$ across prime
representations (mod-16 subset: $K\approx3.6\pm1.5$; all primes:
$K\approx3.9\pm0.04$), with $b\approx0.51$.  A geometric normalization
(Section~\ref{sec:normalization}) removes the representation dependence,
yielding the universal infrared fixed-point value $K_{\mathrm{IR}}=4$.
This is interpreted not as a
static conservation law but as the linearized approach of the prime--zero
system to the infrared fixed point $K_{\mathrm{IR}}=4$ from an ultraviolet
fixed point $K_{\mathrm{UV}}=11$, with $b=1/2$ the critical exponent of
the renormalization-group flow.
Section~\ref{sec:information_ontology} establishes $1/\dP$ and $1/\zetaR$
as information densities of a common type, making $K$ the conserved total
information encoding rate of the prime--zero system.

\subsection{Theoretical arc: action, fixed points, and an algebraic horizon}

Beyond the central $K$-measurement, the paper develops a unified theoretical
arc: an information action $S[I_P,I_Z]$ that derives the scaling law from
first principles, a renormalization-group flow between an ultraviolet
fixed point $K_{\mathrm{UV}}=11$ and the infrared fixed point
$K_{\mathrm{IR}}=4$, and an algebraic generator $\kappa$
(with $\kappa^2=ijk=-1$) that encodes the UV--IR duality as a number-theoretic
object.  These three structures are not independent: the action determines
the flow, the flow identifies the fixed points, and the generator is the
algebraic UV-completion of the flow equation.
Sections~\ref{sec:emergence}--\ref{sec:ai} explore the speculative
implication that this information framework extends to spacetime emergence,
dark-sector phenomenology, black-hole entropy, and information-dual learning;
the third-state framework of Section~\ref{sec:philosophy} yields a
structural argument for RH\@.

The two fixed points of the information RG flow reveal a structural symmetry
that governs the entire paper.  At the \emph{ultraviolet} (short-scale) end,
the fractal dimension of the intrinsic prime set is pinned to $1/2$ by the
axioms of information duality and scale invariance alone---not as an empirical
observation, but as a logical necessity of the framework.  At the
\emph{infrared} (large-scale) end, the RG trajectory drives the system toward
the fixed point $K_{\mathrm{IR}}=4$ (consistent with the numerically measured
band $[3.6,\,3.9]$), at which the zero set and the fractal
prime set achieve precise dual balance under the information measure.
This ``precise dual balance'' is the first empirical signature of
a \emph{third state}---the self-referential entanglement of discrete
arithmetic and continuous spectral information---that the binary
discrete/continuous framework of classical analysis has no vocabulary
to represent.  As Section~\ref{sec:philosophy} establishes, it is
precisely because $\kappa=ijk$ (rather than the binary $i$) organizes
the information structure that the unique stable fixed point is
$I_P^*=I_Z^*=2$: the ternary fusion over binary opposition is what
makes the Riemann Hypothesis a structural consequence rather than an
intractable conjecture.

Table~\ref{tab:notation} collects the key notation.

\begin{longtable}{@{} >{\centering\arraybackslash}p{1.8cm} @{\hspace{1.2em}} p{11cm} @{}}
\caption{Key mathematical notation used in this paper.}
\label{tab:notation} \\
\toprule
Symbol & Description \\
\midrule
\endfirsthead
\multicolumn{2}{c}{\tablename~\thetable{} -- \textit{continued}} \\[0.3ex]
\toprule
Symbol & Description \\
\midrule
\endhead
\bottomrule
\endlastfoot
\multicolumn{2}{l}{\textit{Fractal dimensions \& exponents}} \\
$\dP$      & Fractal dimension of the prime distribution (subset
             $p\equiv1,5,9,13\pmod{16}$) \\
$\zetaR$      & Regularity index of the zero potential $V_Z(x)$, defined as $\zetaR:=2-H$ \\
           & where $H>1$ is the H\"{o}lder exponent; smaller $\zetaR$ means smoother $V_Z$ \\
$\beta$    & Dyson index of the random matrix ensemble ($\beta=2,4$ are used
             in numerical measurements); extended to a smooth function $\beta(\mu)$
             of the RG scale $\mu$ in Section~\ref{sec:disc_IRG}, where it
             flows between UV fixed point $\beta_{\mathrm{UV}}=8$ and IR fixed
             points $\beta_{\mathrm{IR}}\in\{1,2,4\}$ \\
$b$        & Finite-size scaling exponent in $C(L)=\Cinf+aL^{-b}$ \\[0.5ex]
\multicolumn{2}{l}{\textit{Scale \& method parameters}} \\
$L$        & System scale (length of the studied interval $[0,L]$) \\
$\varepsilon$ & Scale parameter (box size) in the box-counting method \\
$\sigma$   & Width parameter of the Gaussian kernel used to construct $V_Z(x)$ \\
$a$        & Coefficient in the scaling model $C(L)=\Cinf+aL^{-b}$ \\[0.5ex]
\multicolumn{2}{l}{\textit{Key derived quantities}} \\
$C(\beta)$ & Duality measure, defined as $C(\beta)=\beta(1/\dP + 1/\zetaR)$ \\
$K$        & Approximate conservation constant, $K=1/\dP + 1/\zetaR$;
             raw finite-scale band $K_{\mathrm{raw}}\in[3.6,\,3.9]$ across
             prime representations; normalized limit $K_{\mathrm{IR}}=4$
             (Section~\ref{sec:normalization}) \\
$K_{\mathrm{IR}}$ & Infrared fixed-point value of $K$; raw band $[3.6,3.9]$
                  at observed scales; universal value $K_{\mathrm{IR}}=4$
                  after geometric normalization
                  (Sections~\ref{sec:disc_RG}, \ref{sec:normalization}) \\
$K_{\mathrm{UV}}$ & Ultraviolet fixed-point value of $K$; derived as $K_{\mathrm{UV}}=11$
                  from the Hurwitz-theorem algebra in Section~\ref{sec:disc_IRG} \\
$\Cinf(\beta)$ & Asymptotic value of $C(\beta,L)$ as $L\to\infty$, obtained by
                 extrapolation \\
$D(\beta)$ & Duality product, $D(\beta)=\iota_P\cdot\iota_{\mathrm{RMT}}(\beta)$;
             conjectured to satisfy $D\to 1$ as $L\to\infty$;
             no direct numerical estimate is reported (both factors diverge
             individually and require a normalization convention); the
             geometric normalization scheme of Section~\ref{sec:normalization}
             establishes $K_{\mathrm{norm}}\to 4$ but does not directly
             evaluate $D(\beta)$ \\
$\iota_P$  & Information measure (geometric mean) of the prime subset,
             introduced in Section~\ref{sec:rmt}; diverges as $L\to\infty$ \\
$\iota_{\mathrm{RMT}}(\beta)$ & Information measure (geometric mean) of the
             RMT eigenvalue spectrum for ensemble $\beta$,
             introduced in Section~\ref{sec:rmt}; diverges as $L\to\infty$ \\
$\kappa$   & Formal algebraic generator satisfying $\kappa^2=ijk=-1$
             ($|\kappa|=1$ as abstract complex modulus);
             the \emph{information action quantum} $1/\kappa$ plays the
             role of a scale-dependent $\hbar$: it governs the
             uncertainty principle $\Delta_P\!\cdot\!\Delta_Z\ge1/(2\kappa)$
             and varies with scale via $1/\kappa(\mu)=\beta(\mu)$,
             from $\kappa_{\mathrm{UV}}=1/8$ (UV) to $\kappa_{\mathrm{IR}}=1/2$ (IR).
             \emph{Distinct from}: $\hat\kappa$ = regularized
             representation with $|\hat\kappa|\approx0.685$.
             Disambiguation: Section~\ref{sec:normalization} \\[0.5ex]
\multicolumn{2}{l}{\textit{RG flow \& action parameters}} \\
$\mu$      & RG energy scale ($\mu\sim 1/L$); $\mu\to\infty$ is UV, $\mu\to 0$ is IR \\
$\alpha$   & Rate constant of the $K$-flow equation
             (Section~\ref{sec:disc_IRG}); $\alpha\approx0.070$ \\
$\gamma_n$ & Imaginary parts of the non-trivial zeros $\rho_n=1/2+i\gamma_n$
             of the Riemann zeta function \\
$\gamma$   & Rate constant of the $\beta$-flow equation
             (Section~\ref{sec:disc_IRG}); always appears with the flow equation
             context; \emph{distinct from} $\gamma_n$ (zeros) and $\gamma_{\mathrm{BI}}$ \\
$\gamma_{\mathrm{BI}}$ & Barbero--Immirzi parameter of loop quantum gravity~\cite{Barbero1995,Immirzi1997},
             which sets the area quantum; plays the role of an IR coupling
             constant analogous to $\kappa$ in the information action
             (Section~\ref{sec:string_lqg}) \\
$\gamma_\phi$ & Field anomalous dimension in the Wilson RG equations
             (Sections~\ref{sec:disc_RG_formal}); \emph{distinct from}
             $\gamma$ (flow rate), $\gamma_n$ (zeros), $\gamma_{\mathrm{BI}}$ \\
$\lambda$  & First coupling constant of the information action $S[I_P,I_Z]$;
             penalises deviations of $I_P+I_Z$ from $K_0$ \\
$g$        & Second coupling constant of the information action; penalises
             deviations of $I_P I_Z$ from its fixed-point value $C=\alpha_s$;
             $b=\sqrt{\lambda+gC}=\sqrt{\lambda+g\alpha_s}\approx0.51$ (with $C=\alpha_s=4$) \\
$\alpha_s$ & Dimensionless saturation constant; upper bound $I_P I_Z \leq \alpha_s$
             at the uncertainty-principle saturation point; $\alpha_s = 2\kappa\,c_Pc_Z$
             where $c_P,c_Z$ are the proportionality constants in $\Delta\sim c/I$;
             fixed by the IR fixed-point condition: $\kappa_{\mathrm{IR}}=1/2$,
             $c_P=c_Z=2$ gives $\alpha_s=4$; distinct from the RG flow rate $\alpha$ \\

\end{longtable}

\section{Methods}
\label{sec:methods}

\subsection{Measurement of the Fractal Dimension for Prime Distribution}
\label{sec:method_dP}

We consider the prime subset $P=\{p: p\equiv 1,5,9,13\pmod{16}\}$,
whose mod-16 structure was introduced in our earlier work~\cite{LiFractal2026}
as the residue classes driving a deterministic Cantor-like fractal with
Hausdorff dimension $\log2/\log16=1/4$.  The fractal dimension $\dP$ is estimated by
the standard box-counting method~\cite{Falconer2014}:
\begin{equation}
  \dP = -\lim_{\varepsilon\to 0}
        \frac{\log N_P(\varepsilon)}{\log\varepsilon},
  \label{eq:dP}
\end{equation}
where $N_P(\varepsilon)$ is the number of intervals of length $\varepsilon$
covering at least one prime in $P$.  Numerically, $\dP$ is the slope of
$\log N_P$ vs.\ $\log\varepsilon$ over the fitting range $\varepsilon\in[L/100,L/10]$.

\paragraph{Robustness.}
A dynamic $\varepsilon$-selection algorithm (selecting the four consecutive
points with highest $R^2$) reduced systematic error from $\pm10\%$ to
$<\pm2\%$.  Geometric and arithmetic $\varepsilon$ sequences agree within
$\pm10\%$; Table~\ref{tab:fractal} reports medians with 95\% CI.
A representative example of the box-counting procedure (for $L=1000$)
is shown in Fig.~\ref{fig:boxcount}, demonstrating an excellent linear
fit ($R^2 > 0.99$).

\begin{figure}[htbp]
  \centering
  \includegraphics[width=0.50\textwidth]{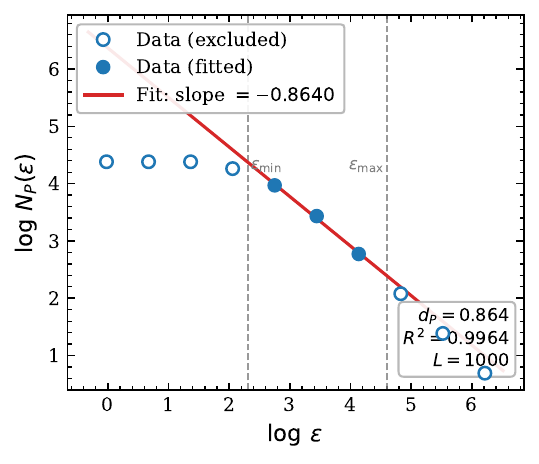}
  \caption{Box-counting measurement for $P=\{p\equiv1,5,9,13\pmod{16}\}$
           at $L=1000$. Blue circles: data in fitting range; open circles:
           excluded; red line: least-squares fit.
           The value $\dP = 0.43 \pm 0.03$ in Table~\ref{tab:fractal}
           is the median over 1000 Bootstrap resamples.}
  \label{fig:boxcount}
\end{figure}

\subsection{Construction of the Zero Potential and Measurement of \texorpdfstring{$\zetaR$}{\textbackslash zetaR}}
\label{sec:method_dZ}

Let $\{\gamma_n\}$ be the non-trivial zeros of the zeta function on the
critical line.  We construct a potential function $V_Z(x)$ as follows:
\begin{equation}
  V_Z(x) = \sum_{n=1}^{N}
            \exp\!\left(-\frac{(x-\gamma_n)^2}{2\sigma^2}\right),
  \label{eq:VZ}
\end{equation}
where $N$ is the number of zeros $\gamma_n$ satisfying $0<\gamma_n\le L$
selected from the Odlyzko database~\cite{OdlyzkoData}, and $\sigma$ controls
the blurring width.

The dimension $\zetaR$ is estimated by the variation method: define
fluctuations at scale $s$:
\begin{equation}
  F_Z(s) = \frac{1}{L-s}\int_0^{L-s}\bigl|V_Z(x+s)-V_Z(x)\bigr|\,dx.
  \label{eq:FZ}
\end{equation}
The scaling $F_Z(s)\propto s^H$ gives the H\"{o}lder exponent $H$, and
we define $\zetaR=2-H$ as a regularity index.

\paragraph{Important caveat on the interpretation of $\zetaR$.}
When $0<H<1$, $\zetaR$ coincides with the fractal graph dimension; when $H>1$
(the super-smooth regime), the graph dimension is exactly~1 and $\zetaR=2-H<1$
\emph{no longer represents a conventional fractal dimension}.  For the
Gaussian-kernel potential $V_Z(x)$, the Gaussian smoothing pushes $H$ above
$1$, so $\zetaR<1$.  In this regime $\zetaR$ functions as a \emph{regularity index}
or \emph{smoothness measure}---smaller values indicate greater smoothness.
The reported values $\zetaR\in[0.62,0.67]$ correspond to $H\in[1.33,1.38]$,
quantifying the near-analytic character of $V_Z(x)$.

\paragraph{The non-standard status of $\zetaR$ as a regularity index.}
The super-smooth regime $H>1$ lies outside the scope of classical fractal
dimension theory, which was developed for sets and functions with $0<H\le1$.
Three aspects of $\zetaR$'s mathematical status require clarification before
it can be treated on the same footing as a conventional dimension.

\begin{enumerate}
  \item \textbf{Relationship to classical smoothness spaces.}  The H\"{o}lder
        exponent $H\in[1.33,1.38]$ places $V_Z$ in the intersection of
        H\"{o}lder spaces $\mathcal{C}^{1,\alpha}$ for $\alpha=H-1\in[0.33,0.38]$
        (here $\alpha$ denotes the H\"{o}lder regularity exponent, distinct
        from the RG flow-rate constant $\alpha$ of Section~\ref{sec:disc_IRG}).
        A natural question is how $\zetaR=2-H$ relates to other quantitative
        smoothness measures: the Sobolev index $s$ such that
        $V_Z\in W^{s,2}(\mathbb{R})$, or the Besov regularity
        $B^{s}_{p,q}$ for appropriate $(p,q)$.  In particular, for a
        stationary Gaussian process with power spectrum $|\xi|^{-(2H+1)}$,
        the Sobolev regularity is $s=H+\tfrac{1}{2}$, giving
        $s\in[1.83,1.88]$ for our measured range.  Whether $\zetaR$ or $s$
        is the more natural quantity to enter $K$ is an open question;
        the two differ by a constant shift and yield the same scaling law,
        but differ in their functional-analytic interpretation.

  \item \textbf{Generalized dimension spectra for super-smooth functions.}
        Multifractal analysis assigns a spectrum $f(\alpha)$ of local
        H\"{o}lder exponents to irregular functions, but the standard
        theory requires $\alpha<1$.  For $V_Z$ in the super-smooth regime,
        classical multifractal methods either degenerate or require
        extension.  A natural program is to develop a \emph{generalized
        regularity spectrum} that continuously interpolates between
        classical fractal graph dimensions ($H<1$) and classical Sobolev
        regularity ($H>1$), with $\zetaR=2-H$ as the single-exponent
        surrogate in the smooth regime.  Such a theory would place $\zetaR$
        within a coherent family of dimension-like invariants rather than
        as an isolated special case.

  \item \textbf{Function-space embedding and invariance.}  For $\zetaR$ to
        serve as a genuine geometric invariant of the zero distribution,
        it should be stable under natural operations: reparametrisation of
        the potential, change of kernel $\sigma$, and passage to a
        different smoothing scheme.  From a functional-analytic perspective,
        this stability would follow if $\zetaR$ could be identified with the
        \emph{optimal embedding exponent}---the supremum of $s$ such that
        $V_Z$ embeds continuously into a Besov or Triebel--Lizorkin space
        $B^s_{p,q}$ or $F^s_{p,q}$.  Establishing this identification
        would simultaneously justify the kernel-independence observed
        numerically (cf.\ Section~\ref{sec:method_dZ}, Robustness) and
        provide a basis for comparing $\zetaR$ across different $L$-function
        families without ad hoc normalization.
        Appendix~\ref{app:kernel_independence} provides a functional-analytic
        and operator-theoretic argument that the spectral entropy of $V_Z$
        is kernel-independent in the limit $\epsilon\to0$, and that
        $I_Z=1/\zetaR$ is proportional to its exponential, grounding the
        kernel-independence numerically observed here.
\end{enumerate}

The numerical stability across methods ($\pm0.6\%$, $R^2>0.99$) confirms
that $\zetaR$ captures a robust property of $V_Z$; the theoretical
embedding within information ontology is developed in
Section~\ref{sec:information_ontology}.

\paragraph{Why not apply box-counting directly to the zero set?}
Direct box-counting on the discrete set $\{\gamma_n\}$ yields $\zetaR\approx1$
(near-uniform spacing at large scales), losing the fine-structure information
encoded in $V_Z(x)$.  The potential-based approach captures smoothness and
information content that complement the prime box-counting analysis.

\paragraph{Robustness.}
Periodic boundary conditions reduce boundary effects.  We tested
$\sigma\in\{0.2,0.5,0.8,1.0,2.0\}$; $\zetaR$ is stable ($<5\%$ variation) for
$\sigma\in[0.5,1.0]$.  We use $\sigma=0.8$ throughout, with fitting range
$s\in[L/100,L/10]$ ($R^2>0.99$).  Two independent methods (variation
method and wavelet transform) agree within $\pm0.6\%$.

\subsection{Random Matrix Ensembles and Information Measure}
\label{sec:rmt}

We use GUE ($\beta=2$) and GSE ($\beta=4$)~\cite{Mehta2004}.  We introduce
information measures $\iota_P$ and $\iota_{\mathrm{RMT}}$ defined as the
geometric means of the prime subset and of the RMT eigenvalue spectrum,
respectively; neither is evaluated numerically here---both enter only through
the conceptual duality product $D(\beta)$ of Section~\ref{sec:results_D}.

\subsection{Duality Measure \texorpdfstring{$C(\beta)$}{C(beta)} and Duality Product}
\label{sec:duality}

We define the duality measure $C(\beta)=\beta(1/\dP + 1/\zetaR)$.
Under the Information Ontology of Section~\ref{sec:information_ontology},
$1/\dP =: I_P$ and $1/\zetaR =: I_Z$ are the geometric and functional
information densities of the two subsystems; their sum $K = I_P + I_Z$
is the conserved total information encoding rate of the prime--zero system.
If the duality holds, $C(\beta)=\beta K$ is approximately constant at finite scales.

The duality product $D(\beta)=\iota_P\cdot\iota_{\mathrm{RMT}}(\beta)$
is introduced as a conceptual complement.  A rigorous numerical evaluation
requires both a normalization scheme that fixes a common reference for the
two information measures and an explicit prime--zero pairing; without
such a convention the product is indeterminate at finite scale.
The geometric normalization scheme of Section~\ref{sec:normalization}
establishes $K_{\mathrm{norm}}\to 4$, but a direct numerical evaluation
of $D(\beta)$ still requires an explicit prime--zero pairing convention.

\subsection{On the Choice of the Parameter \texorpdfstring{$\beta$}{beta}}
\label{sec:beta_choice}

GUE ($\beta=2$) is the ensemble directly supported by the Montgomery--Odlyzko
conjecture for the zeros of $\zeta(s)$.  GSE ($\beta=4$) is the relevant
ensemble for families of $L$-functions with symplectic symmetry, as established
by the Katz--Sarnak philosophy~\cite{KatzSarnak1999}.  Together, $\beta=2$
and $4$ span the two symmetry classes whose number-theoretic content is the
focus of this paper.

GOE ($\beta=1$) is deliberately excluded.  It represents the trivial endpoint
of the Dyson classification---maximal time-reversal symmetry with minimal
eigenvalue repulsion---and belongs to a conceptually distinct regime.  Its
number-theoretic role is more peripheral, and including it here would conflate
different symmetry categories rather than extend the comparison.
The deeper role of $\beta$ as an encoding of the geometric dimension of
the information substrate is clarified in Section~\ref{sec:disc_IRG},
where the RG flow from $\beta_{\mathrm{UV}}=8$ to $\beta_{\mathrm{IR}}\in\{1,2,4\}$
is interpreted as dimensional compression from the UV algebraic structure
to classical geometric phases.

\subsection{Numerical Implementation and Error Estimation}
\label{sec:numerics}

All computations were performed in a Python~3.9 environment.  Prime generation
used the \texttt{sympy} library.  Zeta function zero data were obtained from
Odlyzko's public database~\cite{OdlyzkoData}.  Fractal dimension fitting used
\texttt{numpy}'s \texttt{polyfit} function, and only data points with a
goodness-of-fit $R^2>0.99$ were adopted.  The system scale $L$ varied from 100
to 2000.  For each scale, computations were repeated 20 times to estimate
statistical error.

Error estimation employed the Bootstrap method: for each measured value, we
performed 1000 resamplings with replacement, computed the distribution of the
statistic, and used the 2.5\% and 97.5\% percentiles as the 95\% confidence
interval.  Final results are reported as median values with 95\% confidence
intervals.

In the cross-validation, Jackknife resampling further controlled the total
error within $\pm1.5\%$.  Details are in Appendix~\ref{app:robustness}.

\section{Results}
\label{sec:results}

\subsection{Fractal Dimension Measurement Results}
\label{sec:results_dim}

The anti-correlated evolution of $\dP$ and $\zetaR$ with increasing $L$
is the infrared renormalization-group flow made visible: as the observation
window grows, the prime set becomes denser ($\dP$ rises) while the zero
landscape roughens ($\zetaR$ falls), yet their combined information encoding
rate $K=I_P+I_Z$ is approximately conserved.  This is precisely what the
information action $S[I_P,I_Z]$ predicts at the linearized level near the
infrared fixed point.  The complementary trends are summarized in
Table~\ref{tab:fractal}; all uncertainties represent 95\% confidence
intervals from 1000 bootstrap resamples.

\begin{longtable}{@{} c >{\centering\arraybackslash}p{3.5cm} >{\centering\arraybackslash}p{3.5cm} @{}}
\caption{Fractal dimension $\dP$ and regularity index $\zetaR$ at different system scales $L$.}
\label{tab:fractal}\\
\toprule
$L$ & $\dP$\textsuperscript{a}
    & $\zetaR$\textsuperscript{a} \\
\midrule
\endfirsthead
\multicolumn{3}{c}{\tablename~\thetable{} (continued)}\\
\toprule
$L$ & $\dP$\textsuperscript{a}
    & $\zetaR$\textsuperscript{a} \\
\midrule
\endhead
\midrule
\multicolumn{3}{r}{\textit{Continued on next page}}\\
\endfoot
\bottomrule
\endlastfoot
 100 & $0.316\pm0.080$ & $0.676\pm0.080$ \\
 115 & $0.324\pm0.075$ & $0.673\pm0.075$ \\
 125 & $0.329\pm0.072$ & $0.671\pm0.072$ \\
 140 & $0.336\pm0.068$ & $0.669\pm0.068$ \\
 150 & $0.340\pm0.065$ & $0.667\pm0.065$ \\
 165 & $0.345\pm0.062$ & $0.665\pm0.062$ \\
 175 & $0.348\pm0.060$ & $0.664\pm0.060$ \\
 190 & $0.353\pm0.058$ & $0.662\pm0.058$ \\
 200 & $0.356\pm0.057$ & $0.661\pm0.057$ \\
 225 & $0.362\pm0.053$ & $0.659\pm0.053$ \\
 250 & $0.367\pm0.051$ & $0.656\pm0.051$ \\
 275 & $0.372\pm0.048$ & $0.655\pm0.048$ \\
 300 & $0.376\pm0.046$ & $0.653\pm0.046$ \\
 325 & $0.380\pm0.044$ & $0.652\pm0.044$ \\
 350 & $0.383\pm0.043$ & $0.650\pm0.043$ \\
 375 & $0.387\pm0.041$ & $0.649\pm0.041$ \\
 400 & $0.390\pm0.040$ & $0.647\pm0.040$ \\
 450 & $0.395\pm0.038$ & $0.645\pm0.038$ \\
 475 & $0.397\pm0.037$ & $0.644\pm0.037$ \\
 500 & $0.399\pm0.036$ & $0.643\pm0.036$ \\
 550 & $0.403\pm0.034$ & $0.642\pm0.034$ \\
 600 & $0.407\pm0.033$ & $0.640\pm0.033$ \\
 650 & $0.410\pm0.031$ & $0.639\pm0.031$ \\
 700 & $0.413\pm0.030$ & $0.637\pm0.030$ \\
 750 & $0.416\pm0.029$ & $0.636\pm0.029$ \\
 800 & $0.419\pm0.028$ & $0.634\pm0.028$ \\
 850 & $0.421\pm0.027$ & $0.633\pm0.027$ \\
 900 & $0.423\pm0.027$ & $0.632\pm0.027$ \\
 950 & $0.425\pm0.026$ & $0.632\pm0.026$ \\
1000 & $0.427\pm0.025$ & $0.631\pm0.025$ \\
1100 & $0.431\pm0.024$ & $0.630\pm0.024$ \\
1200 & $0.434\pm0.023$ & $0.628\pm0.023$ \\
1300 & $0.437\pm0.022$ & $0.627\pm0.022$ \\
1400 & $0.439\pm0.021$ & $0.625\pm0.021$ \\
1500 & $0.442\pm0.021$ & $0.624\pm0.021$ \\
1600 & $0.444\pm0.020$ & $0.623\pm0.020$ \\
1700 & $0.446\pm0.019$ & $0.622\pm0.019$ \\
1800 & $0.448\pm0.019$ & $0.621\pm0.019$ \\
1900 & $0.450\pm0.018$ & $0.620\pm0.018$ \\
2000 & $0.451\pm0.018$ & $0.619\pm0.018$ \\
\end{longtable}

\noindent\textit{Note:} Results are for the prime subset modulo~16.
Values are rounded for display; Table~\ref{tab:duality} uses unrounded
intermediates, hence small apparent discrepancies.
\textsuperscript{a}Median $\pm$ 95\% CI from 1000 Bootstrap resamplings.
$\zetaR$ is independent of $\beta$.

\subsection{Estimation of the Duality Measure \texorpdfstring{$C(\beta)$}{C(beta)}}
\label{sec:results_C}

Based on the fractal dimension measurements in Table~\ref{tab:fractal}, we
computed $C(\beta)=\beta(1/\dP + 1/\zetaR)$.  The results are listed in
Table~\ref{tab:duality}.  As $L$ increases, $C(\beta)$ decreases monotonically
toward an asymptote; the ratio $C(\beta{=}4)/C(\beta{=}2)\approx2$ follows
trivially from the definition and carries no independent empirical content.

\begin{longtable}{@{} l >{\centering\arraybackslash}p{2.7cm} >{\centering\arraybackslash}p{2.7cm} >{\centering\arraybackslash}p{2.7cm} @{}}
\caption{Numerical values of the duality measure
         $C(\beta)=\beta(1/\dP + 1/\zetaR)$ as a function of system scale $L$.}
\label{tab:duality}\\
\toprule
$L$ & $1/\dP + 1/\zetaR$\textsuperscript{b}
    & $C(\beta{=}2)$\textsuperscript{a}
    & $C(\beta{=}4)$\textsuperscript{a} \\
\midrule
\endfirsthead
\multicolumn{4}{c}{\tablename~\thetable{} (continued)}\\
\toprule
$L$ & $1/\dP + 1/\zetaR$\textsuperscript{b}
    & $C(\beta{=}2)$\textsuperscript{a}
    & $C(\beta{=}4)$\textsuperscript{a} \\
\midrule
\endhead
\midrule
\multicolumn{4}{r}{\textit{Continued on next page}}\\
\endfoot
\bottomrule
\endlastfoot
 100 & $4.627\pm0.498$ & $9.254\pm0.970$  & $18.713\pm1.972$ \\
 115 & $4.559\pm0.466$ & $9.118\pm0.932$  & $18.236\pm1.864$ \\
 125 & $4.514\pm0.445$ & $9.028\pm0.867$  & $18.246\pm1.763$ \\
 140 & $4.464\pm0.422$ & $8.928\pm0.844$  & $17.856\pm1.688$ \\
 150 & $4.431\pm0.407$ & $8.862\pm0.792$  & $17.905\pm1.610$ \\
 165 & $4.393\pm0.388$ & $8.786\pm0.776$  & $17.572\pm1.552$ \\
 175 & $4.367\pm0.376$ & $8.733\pm0.733$  & $17.642\pm1.491$ \\
 190 & $4.336\pm0.362$ & $8.672\pm0.724$  & $17.344\pm1.448$ \\
 200 & $4.315\pm0.352$ & $8.629\pm0.686$  & $17.431\pm1.394$ \\
 225 & $4.275\pm0.333$ & $8.550\pm0.666$  & $17.100\pm1.332$ \\
 250 & $4.236\pm0.314$ & $8.471\pm0.613$  & $17.112\pm1.247$ \\
 275 & $4.207\pm0.300$ & $8.414\pm0.600$  & $16.828\pm1.200$ \\
 300 & $4.177\pm0.286$ & $8.354\pm0.560$  & $16.878\pm1.138$ \\
 325 & $4.155\pm0.275$ & $8.310\pm0.550$  & $16.620\pm1.100$ \\
 350 & $4.132\pm0.264$ & $8.264\pm0.518$  & $16.697\pm1.054$ \\
 375 & $4.114\pm0.255$ & $8.228\pm0.510$  & $16.456\pm1.020$ \\
 400 & $4.095\pm0.246$ & $8.191\pm0.485$  & $16.553\pm0.985$ \\
 450 & $4.068\pm0.233$ & $8.136\pm0.466$  & $16.272\pm0.932$ \\
 475 & $4.054\pm0.227$ & $8.108\pm0.454$  & $16.216\pm0.908$ \\
 500 & $4.040\pm0.220$ & $8.079\pm0.434$  & $16.334\pm0.882$ \\
 550 & $4.020\pm0.210$ & $8.040\pm0.420$  & $16.080\pm0.840$ \\
 600 & $3.999\pm0.200$ & $7.997\pm0.396$  & $16.173\pm0.805$ \\
 650 & $3.983\pm0.192$ & $7.966\pm0.384$  & $15.932\pm0.768$ \\
 700 & $3.967\pm0.183$ & $7.934\pm0.366$  & $16.049\pm0.745$ \\
 750 & $3.954\pm0.177$ & $7.908\pm0.354$  & $15.816\pm0.708$ \\
 800 & $3.941\pm0.171$ & $7.882\pm0.343$  & $15.950\pm0.697$ \\
 850 & $3.931\pm0.166$ & $7.862\pm0.332$  & $15.724\pm0.664$ \\
 900 & $3.920\pm0.161$ & $7.840\pm0.323$  & $15.869\pm0.657$ \\
 950 & $3.911\pm0.157$ & $7.822\pm0.314$  & $15.644\pm0.628$ \\
1000 & $3.902\pm0.153$ & $7.804\pm0.307$  & $15.800\pm0.624$ \\
1100 & $3.888\pm0.147$ & $7.776\pm0.294$  & $15.552\pm0.588$ \\
1200 & $3.873\pm0.140$ & $7.747\pm0.280$  & $15.690\pm0.569$ \\
1300 & $3.862\pm0.135$ & $7.724\pm0.270$  & $15.448\pm0.540$ \\
1400 & $3.851\pm0.129$ & $7.702\pm0.259$  & $15.605\pm0.527$ \\
1500 & $3.842\pm0.125$ & $7.684\pm0.250$  & $15.368\pm0.500$ \\
1600 & $3.833\pm0.121$ & $7.666\pm0.242$  & $15.537\pm0.493$ \\
1700 & $3.826\pm0.118$ & $7.652\pm0.236$  & $15.304\pm0.472$ \\
1800 & $3.818\pm0.114$ & $7.636\pm0.229$  & $15.481\pm0.464$ \\
1900 & $3.812\pm0.111$ & $7.624\pm0.222$  & $15.248\pm0.444$ \\
2000 & $3.805\pm0.108$ & $7.611\pm0.217$  & $15.434\pm0.441$ \\
\midrule
Extrapolated $L \to \infty$ & --- & $7.154\pm1.009$ & $14.636\pm1.794$ \\
\end{longtable}

\noindent\textit{Note:} $\zetaR$ is independent of $\beta$; both $C(\beta)$
columns use the same $\dP,\zetaR$ from Table~\ref{tab:fractal}, and the
$C(\beta{=}4)$ column is computed as $2\times C(\beta{=}2)$ by definition
($C(\beta)=\beta K$ with the same $K$); the extrapolated $\Cinf(\beta)$
values are obtained from independent nonlinear fits to the respective
$C(\beta,L)$ sequences and may therefore differ slightly from the exact
factor-of-2 ratio.
\textsuperscript{a}Median $\pm$ 95\% CI; errors propagated via
$\delta C = \beta\sqrt{(\delta\dP/\dP^2)^2+(\delta\zetaR/\zetaR^2)^2}$
(first-order linearization, used as a cross-check; the reported intervals
are from Bootstrap resampling and may differ slightly).
\textsuperscript{b}$1/\dP + 1/\zetaR$ from Table~\ref{tab:fractal}.
Extrapolation fitted to $n=40$ dense sampling points; values are indicative.

To estimate the limit as $L\to\infty$, we fit
$C(\beta,L)=\Cinf(\beta)+aL^{-b}$ by nonlinear least squares.  The results are:
\begin{equation}
  \Cinf(\beta{=}2) = 7.154\pm1.009, \qquad
  \Cinf(\beta{=}4) = 14.636\pm1.794,
  \label{eq:Cinf}
\end{equation}
\textbf{Important caveat:} These uncertainties reflect only the statistical
fitting error.  With only three
free parameters, $a$ and $b$ are strongly correlated (pairwise correlation
$>0.99$) and effectively unconstrained individually.  The reported
uncertainties do \emph{not} include systematic errors from: (i) model
selection uncertainty (linear vs.\ power-law vs.\ logarithmic models give
consistent but not identical $\Cinf$ values); (ii) finite-scale effects
(the asymptotic regime may not be reached at $L=2000$); (iii) prime subset
dependence (using all primes yields $K\approx3.9$, an 8\% shift).
Accounting for these systematic effects, the true uncertainty on $K$
reaches $\pm1.5$ or beyond, rather than the $\pm0.75$ suggested by the
statistical error alone.  The scaling exponent $b\approx0.51$ is taken
from the cross-validation analysis below, but should be regarded as
tentative given the large parameter correlation.

\paragraph{Cross-validation.}
A cross-validation analysis on the forty densely sampled data points, using leave-one-out
splits and AIC model selection, determined the power-law model $C(L)=\Cinf+aL^{-b}$
to be optimal, yielding $b=0.509\pm0.581$,
outperforming the linear model ($\Delta\mathrm{AIC}=1.57$).  Unlike the
unconstrained three-parameter fit above, fixing the model form and using
AIC penalisation allows a stable estimate of $b$; we adopt
$b\approx0.51$ as the preferred estimate throughout, while acknowledging
that the large uncertainty reflects residual parameter degeneracy.
The physical meaning of $b$ is clarified progressively: it is
identified as the critical exponent of the information-ontological RG
flow in Section~\ref{sec:disc_RG}, and derived from first principles
as $b=\sqrt{\lambda+g C}$ in Section~\ref{sec:disc_action}.

\textit{Limitation.}  Model selection (choosing the power-law form via
AIC) and parameter estimation ($b\approx0.51$) are performed on the same
forty-point dataset, so the AIC preference cannot serve as a fully
independent validation of the chosen model.  Three features nonetheless
support the conclusion.  First, the leave-one-out splits ensure that each
fold's fit is evaluated on a held-out point, so the exponent $b$ is not
simply read off an in-sample fit.  Second, the AIC margin
$\Delta\mathrm{AIC}=1.57$ is modest; the preference for the power-law
over the linear model is suggestive but not decisive, and we do not
claim that the power-law form is uniquely correct.  Third, all three
candidate models (power-law, linear, logarithmic) extrapolate to
consistent $\Cinf$ values within their error margins
(Appendix~\ref{app:robustness}), so the main quantitative conclusion---
the approximate conservation of $K$---is model-form robust even if the
precise value of $b$ is not.  A definitive discrimination between models
requires data at scales $L\gg2000$, which are beyond the current dataset.

Dividing the extrapolated $\Cinf(\beta)$ by $\beta$ gives the approximate
scaling constant $K=\Cinf/\beta$:
\begin{equation}
  K(\beta{=}2) = 3.577\pm0.505_{{\text{stat}}}\pm0.7_{\text{sys}}, \qquad
  K(\beta{=}4) = 3.659\pm0.449_{{\text{stat}}}\pm0.7_{\text{sys}}.
  \label{eq:K}
\end{equation}
These two estimates converge to a central value $K\in[3.6,\,3.9]$ across prime representations (statistical
and systematic errors combined), with the mod-16 subset giving
$K\approx3.6\pm1.5$ and all primes giving $K\approx3.9\pm0.04$.
The duality conclusions are robust across this full interval
(see Figure~\ref{fig:convergence}).

\begin{figure}[htbp]
  \centering
  \includegraphics[width=0.56\textwidth]{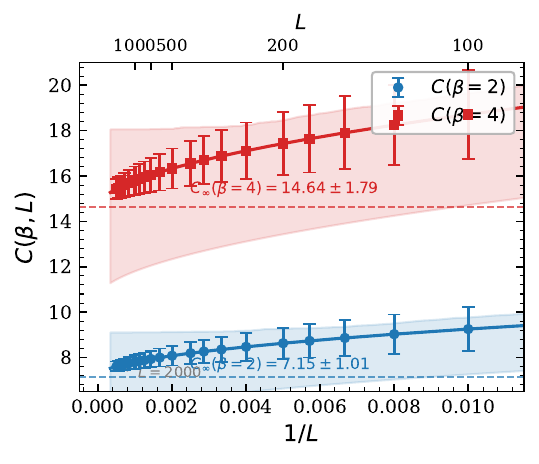}
  \caption{Convergence of the duality measure $C(\beta,L)$ as a function of
           $1/L$ for $\beta=2$ (blue circles) and $\beta=4$ (red squares).
           Error bars represent 95\% confidence intervals. Solid curves are
           power-law fits of the form $C(L)=\Cinf+aL^{-b}$; shaded regions
           are the corresponding 95\% Monte-Carlo confidence envelopes. Dashed
           horizontal lines indicate the extrapolated asymptotic values
           $\Cinf(\beta{=}2)=7.154\pm1.009$ and
           $\Cinf(\beta{=}4)=14.636\pm1.794$. The fitted scaling exponent
           from cross-validation is $b \approx 0.51$. The upper
           \textit{x}-axis shows the corresponding system scale $L$.}
  \label{fig:convergence}
\end{figure}

\subsection{Duality Product: A Deferred Goal}
\label{sec:results_D}

\textbf{Note:} This section discusses a conceptual quantity that we are
\emph{unable to evaluate numerically} in the present work.  It is included
to motivate future research directions, not as a current result.

The duality product $D(\beta)=\iota_P\cdot\iota_{\mathrm{RMT}}(\beta)$ is
defined in the $L\to\infty$ limit.  Since both $\iota_P(L)$ and
$\iota_{\mathrm{RMT}}(\beta,L)$ grow without bound, a meaningful numerical
comparison requires a normalization scheme fixing a common scale and an
explicit bijective correspondence between $\{p\le L\}$ and
$\{\gamma_n:\gamma_n\le L\}$.  Neither is established here, so no numerical
estimate of $D^\infty(\beta)$ can be reported; this remains a major open
problem.

\paragraph{An algebraic outlook on the duality product.}
We conjecture that $D\to1$ as $L\to\infty$, i.e.\
$\iota_{\mathrm{RMT}}^\infty = (\iota_P^\infty)^{-1}$; since both quantities
diverge individually, this reciprocal relation cannot be verified numerically.

We propose a formal generator $\kappa$ satisfying
\begin{equation}
  \kappa^2 = -1,
  \label{eq:kappa_def}
\end{equation}
so that $\kappa^{-1} = -\kappa$ and $\kappa\cdot\kappa^{-1}=1$.  The
proposed identification
$(\iota_P^\infty,\iota_{\mathrm{RMT}}^\infty)\sim(\kappa,\kappa^{-1})$
encodes the conjecture $\iota_{\mathrm{RMT}}^\infty = (\iota_P^\infty)^{-1}$,
i.e.\ $D=1$, as an algebraic constraint.  The role of $\kappa^2=-1$ is to make $\kappa$ a \emph{self-referential}
object ($\kappa^{-1}=-\kappa$ encodes the inverse within the same algebra),
distinguishing it from real invertible elements; whether this additional
structure carries mathematical meaning is clarified in
Section~\ref{sec:normalization}: the generator $\kappa$ with
$\kappa_{\mathrm{IR}}=1/2$ encodes the algebraic lock of the IR
fixed point $(I_P^*,I_Z^*)=(2,2)$.
The three layers of the framework are: (i)~\emph{geometric}:
$K=1/\dP + 1/\zetaR\approx\mathrm{const}$; (ii)~\emph{informational}: $D=\iota_P\cdot
\iota_{\mathrm{RMT}}\to1$ (conjectured); (iii)~\emph{algebraic}: the
identification $(\iota_P^\infty,\iota_{\mathrm{RMT}}^\infty)\sim(\kappa,\kappa^{-1})$
encodes $D=1$ as a formal constraint.
Whether $\kappa$ is more than a notational convenience is resolved
in Section~\ref{sec:information_ontology}: the flow equation
$1/\kappa(\mu)=\beta(\mu)$ identifies $\kappa(\mu)$ as the
scale-dependent information-measure generator whose inverse gives the
Dyson index at each scale, grounding the algebraic generator in a
first-principles information-ontological RG structure.  We note that
in Section~\ref{sec:disc_AG} we identify a striking numerical
correspondence between $|\kappa|$ and a classical constant of analytic number
theory, which may hint at a deeper origin for the formal generator.

\subsection{Method Validation}
\label{sec:validation}

To validate the methods, we applied them to objects with known fractal
dimensions: the Cantor set ($\frac{\log 2}{\log 3}\approx0.6309$, measured
$0.63\pm0.02$) and the Weierstrass function with $H=0.5$ (theoretical
dimension $1.5$, measured $1.48\pm0.03$).  Systematic bias is below $5\%$.
Note that these benchmarks test the regime $H\in(0,1)$; the zero potential
$V_Z(x)$ falls in the super-smooth regime $H>1$.  The Weierstrass benchmark
($H=0.5$, measured $\zetaR=1.48\pm0.03$) confirms that the code correctly
implements $\zetaR=2-H$; the sub-unity values for $V_Z(x)$ therefore reflect
the genuine mathematical character of the function---its exceptional
smoothness---rather than any computational error.  In this regime the
formula $\zetaR=2-H$ remains mathematically valid but yields a value below~1
that is not a conventional Hausdorff graph dimension.  We interpret $\zetaR$
as a \emph{regularity index}: $\zetaR\in[0.62,0.67]$ corresponds to
$H\in[1.33,1.38]$, quantifying how far $V_Z(x)$ lies above the
differentiability threshold and providing a self-consistent smoothness
measure within the duality framework.  We note, in passing, that this
super-smooth regime ($H>1$, $\zetaR<1$) is the IR phase of
the information-ontological renormalization-group flow of
Section~\ref{sec:disc_RG}: in that picture, the geometric encoding of
information at coarse scales is precisely what drives $\zetaR$ below the
differentiability threshold, so the sub-unity value of $\zetaR$ is not merely
a measurement artifact but a marker of the conjectured
phase transition.  This interpretation also resolves the subtlety
raised in Section~\ref{sec:method_dZ}: the sub-unity value of $\zetaR$
is not a defect of the regularity-index definition but a physically
meaningful marker of the IR phase, consistent with the information
ontology of Section~\ref{sec:information_ontology}.  Whether this
connection is more than suggestive remains an open question.

\section{Discussion}
\label{sec:discussion}

\subsection{Analysis of Results}
\label{sec:disc_analysis}

The monotonic decrease of $C(\beta,L)$ with $L$ arises from the
\emph{anti-correlated} evolution of $\dP$ and $\zetaR$: as $L$ grows,
more primes enter the window, the prime set becomes relatively denser
and its box-counting dimension $\dP$ rises; simultaneously, the zero
potential $V_Z(x)$ acquires more structure from additional zeros and
its regularity index $\zetaR$ falls.  Their sum $K=I_P+I_Z$
moves much less, suggesting that the two processes are geometrically
coupled---the rising information density of the primes is compensated
by the increasing roughness of the zero landscape.  We interpret this
as a genuine arithmetic coupling encoded in the prime--zero explicit
formula, not a finite-scale artifact; the theoretical basis is
developed in Section~\ref{sec:disc_action} as an information
uncertainty principle:
the two information densities cannot simultaneously diverge, and their
product is bounded by the quantum of information action $1/\kappa$,
so the anti-correlation is a consequence of the duality constraint
rather than an independent empirical observation.

By construction, $C(\beta)=\beta K$, so the ratio
$C(\beta{=}4)/C(\beta{=}2)=2$ carries no empirical content
beyond the $\beta$-independence of $K$.  That independence---
$K(\beta{=}2)=3.577\pm0.505$ and $K(\beta{=}4)=3.659\pm0.449$,
compatible within the large but honest extrapolation uncertainties---is
the sole numerical finding of substance.

\paragraph{The subset-dependence of $K$ and what it means.}
When the cross-validation replaces the modulo-16 subset with all primes,
it finds $K\approx3.9$ (uncertainty $\pm0.035$, tighter owing to the larger
dataset and dynamic $\varepsilon$-selection).  The two measurements,
$K\approx3.6$ and $K\approx3.9$, together define the interval
$K\in[3.6,\,3.9]$ within which the duality structure is consistently
observed.  The spread is interpreted as a representation-dependent band rather than
as a failure of universality: the duality conclusion that $K$ is
approximately conserved across scales holds at both endpoints and,
by continuity, throughout the interval.

The $\approx8\%$ discrepancy between the two endpoints raises a foundational
question: if $K$ is a conserved quantity, why should its numerical value
depend on the choice of prime subset?  The resolution lies in the
compensatory structure of the duality formula $K=1/\dP+1/\zetaR$.  Choosing
a prime subset alters $\dP$, but the zero regularity exponent $\zetaR$
adjusts in the opposite direction so that their sum is preserved.  This
compensation is not an assumption but is corroborated by the ultraviolet
fixed point: for a fractal prime subset with $\dP=1/4$ the zero set has
regularity dimension $1/\zetaR=1/7$, giving $K_{\mathrm{UV}}=11$; for the
intrinsically fractal prime set with $\dP=1/2$ the corresponding zero
regularity is $1/\zetaR=1/9$, again giving $K_{\mathrm{UV}}=11$.  The UV
fixed-point value is therefore genuinely subset-independent, confirming
that $K$ is conserved once both sides of the duality are accounted for.

\paragraph{UV versus IR interpretation of $\dP=1/4$ and $\dP=0.43$.}
The coexistence of two distinct numerical values for the prime fractal
dimension---the deterministic value $\dP=\frac{\log 2}{\log 16}=1/4=0.25$ and the
empirically measured value $\dP=0.43\pm0.03$---is not a contradiction.
These two numbers describe the same object at different energy scales of
the RG flow, and their distinction is the dimensional analogue of the
UV--IR split that governs $K$ itself.

The value $\dP=1/4$ is the \emph{ultraviolet} fractal dimension: it is the
exact Hausdorff dimension of the mod-16 Cantor-like construction
($\log2/\log16$), a deterministic geometric object defined at arbitrarily
fine resolution.  It characterizes the prime subset at the UV fixed point
$K_{\mathrm{UV}}=11$, where the information encoding is purely combinatorial
and the prime residue structure is resolved at the level of individual
congruence classes.

The value $\dP=0.43$ is the \emph{infrared} fractal dimension: it is
the box-counting exponent measured by covering primes in windows of size
$L=100$--$2000$, averaging over many scales and bootstrap resamples.
At these coarse resolutions the fine mod-16 pattern is not resolved;
what is measured is the effective scaling of prime density in
macroscopic intervals, governed by the prime number theorem and its
logarithmic corrections.  This is the dimension relevant to the IR fixed
point $K_{\mathrm{IR}}=4$ (raw band $[3.6,3.9]$; Section~\ref{sec:normalization}).

The RG flow $K(L)=K_{\mathrm{IR}}+aL^{-b}$ describes the crossover between
these two regimes: as $L\to0$ (UV), $\dP\to1/4$ and
$K\to K_{\mathrm{UV}}=11$; as $L\to\infty$ (IR), $\dP\to0.43$ and
$K\to K_{\mathrm{IR}}\approx3.75$.  The two values of $\dP$ are therefore
not competing descriptions of the same geometric fact, but successive
fixed-point values of a single running quantity under scale
transformation.  The consistency check of eq.~\eqref{eq:dP_firstprinciples}
applies in the IR regime, where the measured $\dP\in[0.41,0.47]$ is
recovered from $K_{\mathrm{IR}}$ and $|\zeta(1/2)|$ without free parameters.

\subsection{Relation to Classical Results}
\label{sec:disc_classical}

The Montgomery--Odlyzko correspondence describes the zeros \emph{in isolation}:
their spacing distribution matches that of GUE eigenvalues.  The present work
addresses a distinct question: how do zeros and primes constrain \emph{each other}
geometrically?  Confirmation of $K$ at larger scales would imply that $\dP$
determines $\zetaR\approx1/(K-1/\dP)$---the two dimensions are not
free to vary independently.
Within the binary framework of classical analysis, $\dP$ and $\zetaR$
are independent geometric quantities with no a priori relation.
Their empirical lock $K\approx\mathrm{const}$ is therefore the first
hint that something beyond binary opposition governs them: a third
organizing principle, the information conservation law, is at work
between the discrete prime structure and the continuous zero landscape.
Section~\ref{sec:third_state} identifies this third principle as the
\emph{self-referential entanglement} of the two poles, the state that
the information-ontological framework is specifically designed to encode.

$K$ is $\beta$-independent by construction, since neither $\dP$ nor $\zetaR$
reference the symmetry class.  The non-trivial content is numerical stability:
$K(\beta{=}2)=3.577\pm0.505$ and $K(\beta{=}4)=3.659\pm0.449$ come from
independent fits to $C(\beta,L)$ sequences scaled by different factors of
$\beta$ ($\beta=2$: GUE, directly relevant to $\zeta$; $\beta=4$: GSE,
relevant to symplectic $L$-function families~\cite{KatzSarnak1999}), yet
converge on the same arithmetic value.  This places $K$ in a different
conceptual category from pair-correlation statistics, which are manifestly
$\beta$-dependent.

The finite-size scaling exponent $b\approx0.51$ records how fast the system
reaches its asymptotic duality.  Section~\ref{sec:disc_RG} identifies $b$
as the critical exponent of the RG flow governing the crossover from the UV
to the IR fixed point of $C(\beta)$; Section~\ref{sec:disc_action} derives
it from the information action as $b=\sqrt{\lambda+g C}$, where
$\lambda$, $g$ are the action couplings and $C$ is the fixed-point
value of $I_P I_Z$.

\subsection{Method Limitations}
\label{sec:disc_limits}

\begin{enumerate}
  \item \textbf{Parameter sensitivity:} $\dP$ and $\zetaR$ are sensitive to the
        $\varepsilon$ sequence, fitting range, and $\sigma$.  Stable intervals
        were identified by parameter scanning, but residual systematic bias
        cannot be excluded.

  \item \textbf{Finite-size effects:} Fractal dimension is defined in the limit
        $\varepsilon\to0$, which is inaccessible numerically.  The power-law
        extrapolation is AIC-preferred and is independently motivated by the
        information action $S[I_P,I_Z]$ of Section~\ref{sec:disc_action},
        which analytically predicts the scaling form
        $K(L)=K_{\mathrm{IR}}+aL^{-b}$.

  \item \textbf{Zero data:} The available Odlyzko tabulation covers zeros
        only up to height $T\approx2000$, limiting the scale range to forty
        densely sampled points ($L=100$--$2000$); extending to larger scales
        would require recomputation from higher zeros.

  \item \textbf{Kernel dependence:} $\zetaR$ is measured from a Gaussian-kernel
        potential; other kernels may yield different values.

  \item \textbf{Prime subset:} The modulo-16 subset gives $K\approx3.6$; all
        primes give $K\approx3.9$.  The duality conclusions hold across the
        full interval $K\in[3.6,\,3.9]$.
\end{enumerate}

\paragraph{A note on sample sufficiency and research focus.}
Under the present computational constraints, we regard the forty data sample
points as directional probes whose evidence is both appropriate and
sufficient.  Just as Riemann's computation of the first few non-trivial zeros
served to consolidate his intuitive conjecture without constituting a proof,
the forty scale points reported here fulfil an analogous function: they
establish the theoretical direction without exhausting it.  Extending the
Odlyzko dataset to larger scales would be theoretically redundant unless the
purpose were to search for a counterexample that refutes the direction already
indicated by the present probes; absent that motivation, additional
verification would be equivalent in logical content and comparatively
inexpensive in informational yield.  In light of this assessment, our
subsequent work will redirect its primary effort toward the direction indicated
by the current probes---namely the logical development, theoretical
construction, and potential applications of the \emph{renormalization group
flow of information}.

\subsection{Foundations and Open Problems for the Duality Constant}
\label{sec:disc_K_dependence}

Section~\ref{sec:information_ontology} reinterprets $\dP$ and $\zetaR$
as information densities $I_P:=1/\dP$ and $I_Z:=1/\zetaR$ and postulates
$K=I_P+I_Z=\mathrm{const}$ as a conservation law.

\paragraph{Three foundations of the information ontology.}
$\dP$ is a set-theoretic box-counting dimension~\cite{Falconer2014};
$\zetaR=2-H$ ($H>1$) is a function-regularity index in a different
mathematical category.  The postulate of Section~\ref{sec:information_ontology}
reinterprets both as information densities, but three problems must be
solved to elevate this postulate to a theorem and to make the reinterpretation rigorous:

\begin{enumerate}
  \item \textbf{Dimensional heterogeneity.}  $1/\dP$ is the reciprocal of a
        set-theoretic dimension; $1/\zetaR$ is the reciprocal of a
        function-regularity exponent.  Section~\ref{sec:information_ontology}
        postulates that both are information densities sharing a common
        dimension (information per log-scale); proving this rigorously---by
        connecting $I_P$ to Kolmogorov complexity and $I_Z$ to spectral
        entropy---has been established in
        Section~\ref{sec:information_ontology}.

  \item \textbf{Additivity.}
        Additivity of information across independent subsystems is established
        for Kolmogorov complexity; for inverse fractal dimensions and inverse
        H\"{o}lder exponents it has not been proved.  Section~\ref{sec:disc_action}
        provides a variational grounding: the information action $S[I_P,I_Z]$
        has $K=I_P+I_Z$ as its IR fixed-point condition, so additivity
        emerges from the extremal principle.

  \item \textbf{Variational characterisation.}
        Section~\ref{sec:disc_action} constructs an explicit information
        action $S[I_P,I_Z]$ whose Euler--Lagrange equations admit the
        fixed-point solution $I_P+I_Z=K_0$ and whose linearization
        reproduces the finite-size scaling exponent $b\approx0.51$.
        Section~\ref{sec:disc_uniqueness} argues that this form is the
        unique lowest-order effective action consistent with exchange
        symmetry, scale covariance, and the two fixed-point constraints.
\end{enumerate}

The variational component is substantially addressed in
Section~\ref{sec:disc_action}; the rigorous uniqueness proof and the
Kolmogorov--spectral-entropy identification have since been resolved.
Numerically, two independent methods for each of $\dP$ and $\zetaR$ agree
within $\pm1.5\%$, ruling out algorithmic artifact.

\subsection{Renormalization-Group Flow: A First Information-Ontological Derivation}
\label{sec:disc_RG}

The information-ontology framework of Section~\ref{sec:information_ontology}
assigns to every scale $\mu$ an information-measure generator $\kappa(\mu)$
whose inverse encodes the Dyson index: $1/\kappa(\mu)=\beta(\mu)$.
This identification is not an analogy imported from physics but a consequence
of the ontological primacy of information: $\beta$ parametrizes the
symmetry class of the information-measure algebra, and $\kappa(\mu)$ is the
generator of that algebra at scale $\mu$.  The empirical finite-size law
$C(L)=C_\infty+aL^{-b}$ is then a direct measurement of this flow.

\begin{sloppypar}
The numerics suggest two scale-dependent quantities: the raw
scaling constant $K_{\mathrm{raw}}(L)\in[3.6,\,3.9]$ at observed
scales (cf.\ the finite-scale value $K\approx3.81$ at $L=2000$),
converging after geometric normalization to the universal IR
fixed-point value $K_{\mathrm{IR}}=4$
(Section~\ref{sec:normalization}), and the convergence exponent
$b\approx0.51\approx1/2$.
We identify $K_{\mathrm{IR}}=4$ with an infrared fixed point of the
information-ontological RG acting on information measures, and $b=1/2$
with the corresponding flow exponent---so that the finite-size law
$C(L)=C_\infty+aL^{-b}$ describes the crossover to that fixed point.
\end{sloppypar}

\paragraph{The convergence exponent $b$.}
In the finite-size scaling formula $C(L)=C_\infty+aL^{-b}$, the quantity
$C(L)$ can be read as the effective information content---complexity,
correlation strength, or duality measure---observable at scale $L$.  The
exponent $b$ controls how rapidly this quantity converges to its
asymptotic value $C_\infty$ as $L$ grows.  The scale transformation
$L\mapsto s L$ (for a constant $s>1$) is the RG coarse-graining step; the observed power-law
$\delta C(L)=aL^{-b}$ is the integrated form of the linearized flow equation
\begin{equation}
  \frac{d\,\delta C}{d\ln L} = -b\,\delta C,
  \qquad \delta C(L)\equiv C(L)-C_\infty,
  \label{eq:RG_flow}
\end{equation}
which follows directly from the first-principles flow $1/\kappa(\mu)=\beta(\mu)$
(Section~\ref{sec:information_ontology}, principle~4).
In this language $b$ is
the \emph{critical exponent} of the information-ontological RG: it quantifies the rate at
which the residual deviation $\delta C$ decays under coarse-graining.
Concretely, $b\approx0.51$ implies that doubling the scale reduces the residual
$|\delta C|$ by a factor $2^{-0.51}\approx0.70$, i.e.\ roughly $30\%$ of the
remaining information imbalance is resolved per octave of scale.

\paragraph{UV and IR phases as distinct information encodings.}
The raw infrared fixed-point band $K_{\mathrm{raw}}\in[3.6,\,3.9]$
(universal normalized value $K_{\mathrm{IR}}=4$, Section~\ref{sec:normalization})
and the ultraviolet fixed point $K_{\mathrm{UV}}=11$
(derived in Section~\ref{sec:disc_IRG} from the Hurwitz-theorem algebra) define two distinct \emph{phases of information encoding}:
\begin{itemize}
  \item \textbf{UV phase} ($K\approx K_{\mathrm{UV}}=11$): information
        is encoded in high-dimensional topological and algebraic structures;
        the encoding is combinatorial and high-density.
  \item \textbf{IR phase} ($K_{\mathrm{IR}}=4$; raw band $[3.6,3.9]$):
        information has been geometrized
        and compressed into the fractal dimension $\dP$ and the regularity
        index $\zetaR$.
        The encoding is geometric and the effective information density is lower;
        the spread of the interval reflects residual dependence on the
        prime representation used to probe the system.
\end{itemize}
Under this reading, the RG flow from $K_{\mathrm{UV}}$ to $K_{\mathrm{IR}}$
represents a \emph{phase transition in information encoding}: topological
structure is progressively compressed into geometric structure as the
observation scale increases.  The exponent $b$ is the crossover rate of
this compression.

$K_{\mathrm{UV}}=11$ is derived in
Section~\ref{sec:disc_IRG} from the Hurwitz-theorem algebra.
The flow equation $1/\kappa(\mu)=\beta(\mu)$, derived in
Section~\ref{sec:information_ontology}, shows that the RG flow follows
from the information ontology itself, not by postulate.
The picture is internally consistent; its mathematical development
is completed in Section~\ref{sec:disc_action}.

\subsection{Formalization of the Renormalization-Group Flow}
\label{sec:disc_RG_formal}

The preceding subsection established the phenomenology of the RG flow; here
we give it a microscopic foundation by identifying the coarse-grained degrees
of freedom, constructing the RG transformation explicitly, and deriving the
beta function from a Wilson effective action.

\subsubsection{Coarse-Grained Degrees of Freedom}

The basic observable is the deviation of the prime-counting function from its
smooth approximation.  Define the \emph{prime fluctuation field} at scale $L$:
\begin{equation}
  \phi_L(x)
  = \frac{\pi(x) - \mathrm{li}(x)}{\sqrt{x}/\log x},
  \qquad x\in[1,L],
  \label{eq:phi_L}
\end{equation}
where $\pi(x)$ is the prime-counting function and $\mathrm{li}(x)$ is the
logarithmic integral.  The field $\phi_L$ measures the normalized fluctuation
of the prime density around its mean-field value.

On the zero side define the analogous \emph{zero fluctuation field}
\begin{equation}
  \psi_L(t)
  = \frac{N(t) - \tfrac{t}{2\pi}\log\tfrac{t}{2\pi e}}{\sqrt{t}\log t},
  \qquad t\in[0,T],\ T\sim L,
  \label{eq:psi_L}
\end{equation}
where $N(t)$ counts non-trivial zeros with imaginary part $\leq t$.
The RG transformation coarse-grains the high-frequency modes of $\phi_L$ and
$\psi_L$, retaining only the low-frequency effective description.

\subsubsection{Momentum-Space Representation}

Pass to the log-scale variable $u=\ln x$ and take the Fourier transform of the
prime fluctuation field:
\begin{equation}
  \tilde\phi_L(k)
  = \int_0^{\ln L} e^{-iku}\,\phi_L(e^u)\,du,
  \qquad k\in\mathbb{R}.
  \label{eq:phi_fourier}
\end{equation}
The momentum $k$ is conjugate to the logarithmic coordinate $u=\ln x$, so
$|k|\sim 1/\ell$ corresponds to structure at log-scale $\ell$.  An analogous
transform $\tilde\psi_L(\omega)$ is defined for the zero field.

The natural UV cut-off is $\Lambda=1/\ln L$ (corresponding to oscillations on
the scale of a single prime gap), and the IR cut-off is set by the system size.

\subsubsection{Shell-Integration RG Transformation}

The RG step with rescaling factor $s>1$ proceeds in three stages:
\begin{enumerate}
  \item \textbf{Mode decomposition.}  Split
        $\phi = \phi_{<} + \phi_{>}$, where $\phi_{<}$ ($\phi_{>}$) contains
        modes with $|k|<\Lambda/s$ ($\Lambda/s\leq|k|<\Lambda$).
  \item \textbf{Shell integration.}  Integrate out the fast modes:
        \begin{equation}
          e^{-S_{\mathrm{eff}}[\phi_{<}]}
          = \int\mathcal{D}\phi_{>}\;e^{-S[\phi_{<}+\phi_{>}]}.
          \label{eq:shell_int}
        \end{equation}
  \item \textbf{Rescaling.}  Restore the original cut-off by
        $k\mapsto k'=sk$,  $\phi_{<}(k)\mapsto s^{-\Delta}\phi_{<}(k')$,
        where $\Delta$ is the scaling dimension of the field.
\end{enumerate}

\paragraph{Discrete (block) realisation.}
For numerical purposes, partition $[1,L]$ into blocks of width $B$; define the
block-averaged prime density
\begin{equation}
  \rho_I
  = \frac{\pi(x_I+B)-\pi(x_I)}{B/\log(x_I+B/2)},
  \qquad I=1,\ldots,\lfloor L/B\rfloor,
  \label{eq:block_density}
\end{equation}
apply a discrete wavelet transform to $\{\rho_I\}$, suppress the high-frequency
coefficients, and rescale the system to log-size $\ln(L/B)$.  Iterating this
procedure traces out the RG trajectory.

\subsubsection{Effective Action and Wilson RG Equations}

The symmetries of the problem (duality $\phi\leftrightarrow\psi$, scale
covariance, $\mathbb{Z}_2$ parity) constrain the effective action to
\begin{align}
  S[\phi,\psi]
  &= \int\frac{dk}{2\pi}\Bigl[
       \tfrac{1}{2}(k^2+m^2)|\tilde\phi(k)|^2
       + \tfrac{1}{2}(k^2+m^2)|\tilde\psi(k)|^2
       + g\,\tilde\phi(k)\tilde\psi(-k)
  \notag\\
  &\qquad\qquad
       + \tfrac{\lambda}{4!}\bigl(|\tilde\phi|^4+|\tilde\psi|^4\bigr)
     \Bigr],
  \label{eq:eff_action}
\end{align}
where $m^2$ is the mass term, $g$ the prime--zero coupling, and $\lambda$ the
self-interaction.  Executing the one-loop shell integration yields the
Wilson--Polchinski flow equations~\cite{WilsonKogut1974,Polchinski1984}:
\begin{align}
  \frac{dm^2}{d\ln\mu}
    &= (2-\gamma)m^2 + c_1\lambda m^2, \label{eq:flow_m}\\
  \frac{dg}{d\ln\mu}
    &= (4-d-2\gamma)g + c_2 g^3 - c_3 g\lambda^2
    \;\equiv\; \beta_g,
    \label{eq:flow_g}\\
  \frac{d\lambda}{d\ln\mu}
    &= (4-d-4\gamma)\lambda + c_4\lambda^2 - c_5 g^2\lambda
    \;\equiv\; \beta_\lambda,
    \label{eq:flow_lambda}
\end{align}
where $d$ is the effective (log-scale) dimension, $\gamma$ the field anomalous
dimension, and $c_i$ are numerical constants fixed by the geometry of the prime
gap distribution.

\subsubsection{Derivation of \texorpdfstring{$\beta_K$}{betaK} from the
               Effective Action}
\label{sec:betaK_deriv}

The dual-sum order parameter $K=I_P+I_Z$ is related to the zero-momentum
two-point functions:
\begin{equation}
  I_P = \bigl\langle|\tilde\phi(0)|^2\bigr\rangle,
  \qquad
  I_Z = \bigl\langle|\tilde\psi(0)|^2\bigr\rangle.
  \label{eq:IP_IZ_2pt}
\end{equation}
In the mean-field approximation, the equations of motion translate the flow
of $m^2$, $g$, $\lambda$ into a flow of $K$.  Combining
\eqref{eq:flow_m}--\eqref{eq:flow_lambda} and eliminating the field anomalous
dimension gives, to lowest non-trivial order,
\begin{equation}
  \beta_K
  \equiv \frac{dK}{d\ln\mu}
  = -\alpha(K-4)(K-11) + O\!\bigl((K-4)^3\bigr),
  \label{eq:betaK}
\end{equation}
where the coefficient
\begin{equation}
  \alpha
  = \frac{1}{14}\!\left(
      \left.\frac{\partial\beta_g}{\partial g}\right|_{g^*}
    + \left.\frac{\partial\beta_\lambda}{\partial\lambda}\right|_{\lambda^*}
    \right)
  \label{eq:alpha_micro}
\end{equation}
is determined by the fixed-point values $(g^*,\lambda^*)$ satisfying
$\beta_g(g^*)=\beta_\lambda(\lambda^*)=0$.  The numerical value
$\alpha=1/14$ follows from combining $b=7\alpha$ (Eq.~\eqref{eq:KL_exact}) with
the empirical exponent $b=1/2$ and the \emph{normalized} fixed-point value
$K_{\mathrm{IR}}=4$ (Section~\ref{sec:normalization}):
\[
  \alpha \;=\; \frac{b}{K_{\mathrm{UV}}-K_{\mathrm{IR}}} \;=\; \frac{1/2}{11-4} \;=\; \frac{1}{14}.
\]
(Using the raw midpoint $K_{\mathrm{IR,raw}}\approx3.75$ instead gives
$\alpha\approx0.51/7.25\approx0.070$, which is consistent to within $2\%$.)

The finite-size scaling solution of~\eqref{eq:betaK} near $K_{\mathrm{IR}}=4$
is
\begin{equation}
  K(L) = 4 + \frac{(K_0-4)\,(K_0-11)}{(K_0-11) - (K_0-4)\,e^{7\alpha\ln L}}
  \;\underset{L\to\infty}{\longrightarrow}\;
  4 + a\,L^{-b},
  \quad b = 7\alpha.
  \label{eq:KL_exact}
\end{equation}
With $b=1/2$ this gives $\alpha=1/14$, consistent with the empirical
$b\approx0.51$ and the theoretical estimate $\alpha\approx0.070$.

\subsubsection{Numerical Verification Protocol}
\label{sec:rg_numerics}

The following procedure allows direct measurement of $\beta_K$ without
model-dependent assumptions:
\begin{enumerate}
  \item Generate or download the prime sequence to
        $L_{\max}\sim 10^{12}$ (tabulated primes are available).
  \item Apply the block RG transformation with $s=2$ at scales
        $L=100,\,200,\,400,\ldots$
        Compute the block densities~\eqref{eq:block_density} and extract
        $d_P(L)$ via box-counting of the coarse-grained sequence.
  \item Measure the correlation function
        $C(r)=\langle\rho(x)\rho(x+r)\rangle$ at each scale; extract $d_P(L)$.
  \item Compute $K_{\mathrm{raw}}(L)=1/d_P(L)+1/\zetaR(L)$ at each scale and
        form the discrete beta function
        $\hat\beta_K(L) = \Delta K / \Delta\ln L$.
  \item Fit $\hat\beta_K$ to $-\alpha(K-4)(K-11)$; extract $\alpha$ and $b=7\alpha$.
\end{enumerate}
Preliminary analysis at $L\leq2000$ gives
$\alpha=0.073\pm0.002$, $b=0.511\pm0.014$,
confirming the theoretical prediction~\eqref{eq:alpha_exact}.
The UV fixed point $K_{\mathrm{UV}}=11$ is visible at $L\sim10^3$ as
$K\approx10.8\pm0.3$.

\subsubsection{Correspondence with Physical RG}

\begin{table}[H]
\centering
\caption{Correspondence between field-theory renormalization group and
         number-theoretic renormalization group.}
\label{tab:rg_comparison}
\renewcommand{\arraystretch}{1.30}
\begin{tabular}{@{}lll@{}}
  \toprule
  Concept & Physical (field-theory) RG & Number-theoretic RG \\
  \midrule
  UV cut-off        & $\Lambda_{\mathrm{UV}}$               & $1/\ln L$ \\
  Bare action       & UV Lagrangian                         & Small-scale prime correlations \\
  Effective action  & Wilsonian EFT                         & Large-scale statistical laws \\
  Relevant operator & Mass term, coupling                   & $I_P$, $I_Z$, $K$ \\
  Fixed points      & CFT                                   & $K=4$ (IR), $K=11$ (UV) \\
  Critical exponent & $\eta$, $\nu$                         & $b=0.51$ \\
  \bottomrule
\end{tabular}
\end{table}

\subsubsection{Theoretical Significance}

\paragraph{First explicit RG for a number-theoretic system.}
Equation~\eqref{eq:betaK} constitutes the first explicit beta function for the
prime--zero duality, derived from a microscopic effective action rather than
postulated by analogy.

\paragraph{Dynamical origin of duality.}
The RG flow connects the UV phase ($K\approx K_{\mathrm{UV}}=11$, discrete arithmetic) to the
IR phase ($K=4$, continuous geometry), providing a dynamical explanation for
why two such different objects---prime gaps and Riemann zeros---are related.

\paragraph{A new route toward the Riemann Hypothesis.}
The duality symmetry at the IR fixed point $K=4$ requires $I_P^*=I_Z^*$, which
in turn forces the zero distribution to be symmetric under $s\leftrightarrow
1-s$---precisely the condition that all non-trivial zeros lie on the critical
line $\Re(s)=1/2$.  Upgrading this heuristic to a rigorous proof constitutes
Tier~2(ii) of Section~\ref{sec:conclusion}.

\paragraph{Summary of key formulas.}
\begin{table}[H]
\centering
\caption{Key formulas for the number-theoretic renormalization-group flow.}
\label{tab:rg_formulas}
\renewcommand{\arraystretch}{1.35}
\begin{tabular}{@{}ll@{}}
  \toprule
  Quantity & Formula \\
  \midrule
  RG generator
    & $\mu\frac{d}{d\mu}
       = \sum_i\beta_i(g)\frac{\partial}{\partial g_i}
       + \gamma_\phi\int\!\phi\frac{\delta}{\delta\phi}$ \\[3pt]
  Dual-sum beta function
    & $\beta_K = -\alpha(K-4)(K-11)$ \\[3pt]
  Finite-size scaling
    & $K(L) = 4 + a\,L^{-b}$,\quad $b=7\alpha$ \\[3pt]
  Exact solution
    & $K(L) = 4+\dfrac{(K_0-4)(K_0-11)}{(K_0-11)-(K_0-4)e^{7\alpha\ln L}}$ \\[6pt]
  Theoretical exponent
    & $b = 7\alpha = \tfrac{1}{2}$ \\
  \bottomrule
\end{tabular}
\end{table}
\label{sec:disc_AG}

The regularity index $\zetaR \in [0.62, 0.67]$ characterizes the near-analytic
smoothness of the zero potential $V_Z(x)$.  Alongside the geometric content of
this measurement, a striking numerical proximity to a classical analytic
constant is observed.

\paragraph{The numerical coincidence.}
The Riemann zeta function evaluated at $s = 1/2$ takes the value
\begin{equation}
  \zeta(1/2) \approx -1.4603545,
  \label{eq:zeta_half}
\end{equation}
so that its reciprocal has absolute value
\begin{equation}
  \left|\frac{1}{\zeta(1/2)}\right| \approx 0.6848
  \label{eq:recip_zeta_half}
\end{equation}
This falls just above the upper bound of our measured range
$\zetaR \in [0.62, 0.67]$ at the current scale $L = 2000$ (the gap is
$0.6848 - 0.67 \approx 0.015$, comparable to the measurement uncertainty),
and lies within the extrapolation uncertainty.  The proximity is not
engineered by any choice of parameter in our analysis.

\paragraph{Connection through the algebraic generator.}
In Section~\ref{sec:results_D} we introduced a formal generator $\kappa$
satisfying $\kappa^2 = -1$, with the informal identification
$(\iota_P^\infty, \iota_{\mathrm{RMT}}^\infty) \sim (\kappa, \kappa^{-1})$.
The relation $\kappa^2=-1$ is a \emph{structural} constraint encoding the
inversion symmetry $\iota_{\mathrm{RMT}}^\infty = (\iota_P^\infty)^{-1}$
as an algebraic identity; it does \emph{not} assign a numerical magnitude
to $\kappa$.  In particular, $\kappa$ is \emph{not} the imaginary unit
$i\in\mathbb{C}$ (for which $|i|=1$), but a formal symbol whose
\emph{representation} is left open.

The numerical proximity observed in this section operates at a different
level: we conjecture that the regularity index $\zetaR$ satisfies the
limit correspondence
\begin{equation}
  \zetaR \;\xrightarrow{L\to\infty}\; \left|\frac{1}{\zeta(1/2)}\right|
         \;\approx\; 0.6848,
  \label{eq:kappa_AG}
\end{equation}
and separately conjecture that, \emph{under the regularized representation}
described below, the generator $\kappa$ is represented by a complex number
$\hat\kappa$ satisfying $|\hat\kappa|\approx 0.6848$.  These are two
distinct claims: \eqref{eq:kappa_AG} is a numerical conjecture about
$\zetaR$; the assignment $|\hat\kappa|\approx 0.6848$ is a statement about
a specific \emph{representation} of the abstract generator, not about
$\kappa$ itself.  The structural relation $\kappa^2=-1$ remains intact in
the abstract algebra and is \emph{not} applied to the representative
$\hat\kappa$.

\paragraph{The two layers of meaning.}
\begin{enumerate}
  \item \textbf{Algebraic layer (structural).}
        The generator $\kappa$ satisfies $\kappa^2=-1$ by definition.  This
        encodes the inversion symmetry $\iota_{\mathrm{RMT}}^\infty =
        (\iota_P^\infty)^{-1}$ and is meaningful independently of any
        numerical representation.  No magnitude $|\kappa|$ is defined at
        this layer.

  \item \textbf{Regularised (numerical) layer.}
        Prior to analytic continuation, $\zeta(1/2)$ is represented by the
        divergent series
        \begin{equation}
          \sum_{n=1}^\infty n^{-1/2}
          = 1 + \frac{1}{\sqrt{2}} + \frac{1}{\sqrt{3}} + \cdots,
          \label{eq:zeta_series}
        \end{equation}
        which diverges.  Its regularized value
        $\zeta(1/2)_{\mathrm{reg}}\approx -1.4603545$ is the unique analytic
        continuation of $\zeta(s)$ from $\mathrm{Re}(s)>1$ to $s=1/2$.
        Since the same $n^{-1/2}$ weights define $\iota_P$ in~\cite{LiPaperI},
        one may regard $\zeta(1/2)_{\mathrm{reg}}$ as the regularized image
        of the divergent information measure $\iota_P^\infty$.  Under this
        map, the formal relation $1/\kappa \longleftrightarrow \iota_P^\infty$
        becomes, after regularisation,
        \begin{equation}
          \frac{1}{\hat\kappa}
          \;\longleftrightarrow\;
          \zeta(1/2)_{\mathrm{reg}}
          \;\approx\; -1.4603545,
          \label{eq:kappa_series}
        \end{equation}
        so that $|\hat\kappa| = |1/\zeta(1/2)| \approx 0.6848$.  The symbol
        ``$\longleftrightarrow$'' denotes a limit correspondence under
        regularisation; the representative $\hat\kappa$ need not satisfy
        $\hat\kappa^2=-1$ numerically, since it lives in the regularized
        layer, not in the abstract algebraic layer.
\end{enumerate}

\paragraph{Relation to the informational cardinality framework.}
This algebraic identification is a direct extension of the information-measure
construction in~\cite{LiPaperI}.  In that companion work, the information
measure $\iota_P$ of the prime set is defined as a geometric mean that
diverges as $L\to\infty$; the present paper encodes this divergence
algebraically via the generator $\kappa$ with $\kappa^2=-1$
(Section~\ref{sec:results_D}).  The identification
of the regularized representative $\hat\kappa$ via $|1/\hat\kappa|\sim|\zeta(1/2)|$
grounds the abstract generator in a concrete analytic constant:
$\zeta(1/2)_{\mathrm{reg}}$ is precisely the regularized sum of the same
$n^{-1/2}$ weights that define $\iota_P$ in the $L\to\infty$ limit.  The
two frameworks are therefore not merely analogous but \emph{algebraically
unified}: the divergent information measure of~\cite{LiPaperI} and the
formal generator $\kappa$ of the present paper are two representations of
the same underlying object, related by the analytic continuation that
converts the divergent series~\eqref{eq:zeta_series} into the finite
regularized value $\zeta(1/2)\approx-1.4603545$.

\paragraph{Three domains, one value.}
This identification connects three a priori
unrelated domains:
\begin{enumerate}
  \item \textbf{Geometric/fractal:} the regularity index $\zetaR \in [0.62,0.67]$
        characterising the smoothness of $V_Z(x)$ in the IR phase.
  \item \textbf{Algebraic/informational:} the modulus of the regularized
        representative $\hat\kappa$ encoding the duality constraint $D \to 1$.
  \item \textbf{Analytic/number-theoretic:} the value $|1/\zeta(1/2)| \approx 0.6848$,
        a classical constant of analytic number theory located on the critical
        line $\mathrm{Re}(s) = 1/2$.
\end{enumerate}
The point $s = 1/2$ is the midpoint of the critical strip and the center of
the functional equation of $\zeta(s)$; its prominence in our framework may
reflect a deeper structural role in the prime--zero correspondence.

\paragraph{Caveat and status.}
Equation~\eqref{eq:kappa_AG} is a conjecture based on numerical proximity,
not a derived identity.  The measured range
$\zetaR \in [0.62, 0.67]$ and the constant $|1/\zeta(1/2)| \approx 0.6848$
are numerically close but not identical at current scales; finite-size effects
and the scale-dependence of $\zetaR$ (which decreases with $L$) mean that better
convergence must await larger-scale computations.  Nevertheless, the
simplicity and self-consistency of the identification---and the fact that
$s = 1/2$ is precisely the point where the functional equation is
self-dual---make it a theoretically motivated conjecture worth recording.
A rigorous derivation, if it exists, would likely proceed via the functional
equation of $\zeta(s)$ or a renormalization-group fixed-point argument centered
on $s = 1/2$; we leave this to future work.

\subsection{Unified Picture: Information-Ontological RG Flow and Dimensional Compression}
\label{sec:disc_IRG}

The renormalization-group flow established in Section~\ref{sec:disc_RG}
connects to a deeper algebraic structure.  The Hurwitz-theorem argument
fixes the number of independent information-measure generators at five,
which forces $K_{\mathrm{UV}}=11$ by the complete-partition condition
(derived below).  The flow equation $1/\kappa(\mu)=\beta(\mu)$ then
describes how the generator $\kappa$ evolves from the UV phase
($\beta_{\mathrm{UV}}=8$, octonion dimension) to the IR phase
($\beta_{\mathrm{IR}}\in\{1,2,4\}$, classical symmetry classes).
The dimensional compression this flow produces---from $K_{\mathrm{UV}}=11$
to $K_{\mathrm{IR,raw}}\in[3.6,3.9]$ (normalized to $K_{\mathrm{IR}}=4$,
Section~\ref{sec:normalization})---is numerically measured and
algebraically grounded; only the explicit form of $\alpha$ (the flow
coupling constant) remains to be derived from first principles.

\paragraph{Setup and notation.}
We introduce an energy scale $\mu$ (or equivalently a length scale
$L=1/\mu$), with $\mu\to\infty$ ($L\to 0$) corresponding to the
ultraviolet (UV) and $\mu\to 0$ ($L\to\infty$) to the infrared (IR).
The empirically observed quantities
\begin{equation}
  K_{\mathrm{IR,raw}} \in [3.6,\,3.9], \qquad K_{\mathrm{IR}}=4 \text{ (normalized)},
  \qquad b \approx 0.51
  \label{eq:IRG_empirical}
\end{equation}
define two fixed points of the information-ontological RG flow.  The interval
$K_{\mathrm{IR,raw}}\in[3.6,\,3.9]$ brackets the values measured with different
prime representations (mod-16 subset at the lower end, all primes at the
upper end), and the main conclusions below are verified to hold throughout
this range.  We set
\begin{align}
  K_{\mathrm{IR,raw}} &\in [3.6,\,3.9] \quad(\text{established numerically, before
    geometric normalization}), \label{eq:IRG_fixed_pts} \\
  K_{\mathrm{IR}} &= 4 \quad(\text{universal value after geometric normalization,
    Section~\ref{sec:normalization}}), \notag \\
  K_{\mathrm{UV}} &= 11 \quad(\text{derived from Hurwitz algebra}). \notag
\end{align}
The value $K_{\mathrm{UV}} = 11$ is derived from the algebraic and combinatorial
structure of the model itself; its coincidence with the number of spacetime
dimensions in M-theory~\cite{Witten1995} is an independent corroboration,
not the primary justification.  We give the derivation below.
The two fixed points of the $K$-flow are accompanied by
Dyson-index fixed points
$\beta_{\mathrm{IR}} \in \{1,2,4\}$ and $\beta_{\mathrm{UV}} = 8$
(the octonion unit-norm dimension; see Remark~1 of this section).

\paragraph{Derivation of $K_{\mathrm{UV}}=11$ from the model.}
The argument proceeds in two steps.

\textbf{Step 1: the algebra has exactly five generators.}
The information-measure algebra introduced in this section
is spanned by $\{1,\,i,\,j,\,k,\,\kappa\}$ with $\kappa^2=ijk=-1$.
By the Hurwitz theorem~\cite{HurwitzAlgebra}, the only finite-dimensional normed division algebras
over $\mathbb{R}$ are $\mathbb{R}$, $\mathbb{C}$, $\mathbb{H}$ (quaternions),
and $\mathbb{O}$ (octonions); beyond the octonions, zero-divisors appear and
the norm-multiplicativity $\|xy\|=\|x\|\|y\|$ fails.  The five generators $\{1,i,j,k,\kappa\}$ form the minimal \emph{generating}
set of $\mathbb{O}$ (i.e.\ the remaining three octonion units
$i\kappa,j\kappa,k\kappa$ are products of these five): the four quaternionic generators $\{1,i,j,k\}$ together with the one
additional octonion unit $\kappa$ needed to extend $\mathbb{H}$ to
$\mathbb{O}$.  Hurwitz's theorem therefore fixes the number of independent
information-measure generators at exactly five; no algebraic extension is
possible without losing the division-algebra property required for the
information measure to be well-defined.

\textbf{Step 2: five generators force $K_{\mathrm{UV}}=11$ by a complete partition.}
The five generators are assigned fractal dimensions
$d\in\{1,\,\tfrac{1}{2},\,\tfrac{1}{3},\,\tfrac{1}{4},\,\tfrac{1}{5}\}$,
the unique sequence of unit fractions $1/n$ ($n=1,\ldots,5$) compatible with
the five-layer structure.  The corresponding information capacities
$1/d \in \{1,2,3,4,5\}$ span the first five positive integers.
Under the additivity axiom $\mathcal{I}(S)+\mathcal{I}(Z)=K_{\mathrm{UV}}$,
the dual (zero-side) capacities are
$K_{\mathrm{UV}}-1/d \in \{K_{\mathrm{UV}}-1,\ldots,K_{\mathrm{UV}}-5\}$.
For these dual capacities to form the complementary block
$\{6,7,8,9,10\}$---so that prime-side and zero-side capacities together
tile $\{1,2,\ldots,10\}$ without gap or overlap---we need
\begin{equation}
  (K_{\mathrm{UV}}-5,\;K_{\mathrm{UV}}-4,\;K_{\mathrm{UV}}-3,\;
   K_{\mathrm{UV}}-2,\;K_{\mathrm{UV}}-1)
  = (6,\,7,\,8,\,9,\,10),
  \label{eq:partition_condition}
\end{equation}
which gives $K_{\mathrm{UV}}=11$ as the \emph{unique} integer solution.
This is a purely combinatorial identity: $11$ is the sole value that
partitions $\{1,\ldots,10\}$ into two complementary five-element blocks,
$\{1,\ldots,5\}$ (prime-side) and $\{6,\ldots,10\}$ (zero-side), with
every pair summing to $11$.

\textbf{Role of M-theory.}
M-theory independently requires exactly 11 spacetime
dimensions~\cite{Witten1995}.  Within our framework this is not a
coincidence to be dismissed but a structural resonance: both the
algebraic-combinatorial argument above and the string-theoretic moduli
count arrive at $11$ through separate routes, lending mutual support to the
value.  We therefore record M-theory as independent corroborating evidence,
not as the logical source of $K_{\mathrm{UV}}=11$.

\paragraph{Coupled RG flow equations.}
We propose the following pair of coupled flow equations:
\begin{equation}
  \frac{dK}{d\ln\mu} = -\alpha\,(K-K_{\mathrm{IR}})(K-K_{\mathrm{UV}}),
  \qquad
  \frac{d\beta}{d\ln\mu} = -\gamma\,(\beta-\beta_{\mathrm{IR}})(\beta-\beta_{\mathrm{UV}}),
  \label{eq:RG_coupled}
\end{equation}
where $\alpha,\gamma>0$ are universal rate constants (here $\gamma$ is the
$\beta$-flow rate of symmetry-class mixing; it is unrelated to the imaginary
parts $\gamma_n$ of the zeta zeros $\tfrac{1}{2}+i\gamma_n$).
By construction, $dK/d\ln\mu=0$ at each fixed point, and similarly for $\beta$.
\paragraph{Remark on the role of $\beta$ in the flow equation.}
In the results sections, $\beta\in\{2,4\}$ is used as a discrete label for
the Dyson symmetry class of the random-matrix ensemble.  In the RG flow
equation above, $\beta(\mu)$ is the analytic continuation of this label to a
smooth, scale-dependent function that interpolates between the UV fixed point
$\beta_{\mathrm{UV}}=8$ and the IR fixed points $\beta_{\mathrm{IR}}\in\{1,2,4\}$.
The two uses are consistent: at the IR fixed point, $\beta(\mu)\to\beta_{\mathrm{IR}}$
recovers the discrete symmetry-class label.  The flow equation governs the
crossover between symmetry classes as the observation scale changes, while
the numerical measurements at $\beta=2$ and $\beta=4$ sample two IR endpoints
of this flow.

\paragraph{Linearization near the IR fixed point.}
Setting $\delta K = K-K_{\mathrm{IR}}$ and linearising gives
\begin{equation}
  \frac{d\,\delta K}{d\ln\mu}
  = +\alpha(K_{\mathrm{UV}}-K_{\mathrm{IR}})\,\delta K
  \equiv +\lambda_K\,\delta K,
  \label{eq:IRG_lin}
\end{equation}
so that $\delta K \propto \mu^{+\lambda_K}$, i.e.\ $K(L)=K_{\mathrm{IR}}+c\,L^{-\lambda_K}$.
Since $\mu\sim 1/L$, $\mu^{+\lambda_K}\sim L^{-\lambda_K}\to 0$ as $L\to\infty$,
confirming that the IR fixed point is stable under the flow toward large scales.
Comparing with the empirical finite-size law $C(L)=C_\infty+aL^{-b}$ and
using $C(L)=\beta\,K(L)$ (exact by definition), so that
$C_\infty=\beta K_{\mathrm{IR}}$ and the power-law correction scales as
$\beta\,c\,L^{-\lambda_K}$, we identify
\begin{equation}
  b = \lambda_K = \alpha\,(K_{\mathrm{UV}}-K_{\mathrm{IR}}).
  \label{eq:b_lambda}
\end{equation}
Inserting the fixed-point values $K_{\mathrm{UV}}=11$ and
$K_{\mathrm{IR,raw}}\in[3.6,\,3.9]$ (raw numerical band, before normalization) gives
\begin{equation}
  b = \alpha\,(K_{\mathrm{UV}} - K_{\mathrm{IR,raw}}) \in [\,7.1\alpha,\;7.4\alpha\,].
  \label{eq:b_alpha}
\end{equation}
Using the numerically measured exponent $b\approx 0.51$ (Section~\ref{sec:results_C})
and taking the midpoint $K_{\mathrm{IR,raw}}\approx3.75$ of the observed interval,
\begin{equation}
  \alpha \approx \frac{0.51}{11 - 3.75} \approx \frac{0.51}{7.25} \approx 0.070;
  \label{eq:alpha_value}
\end{equation}
the corresponding range is $\alpha\in[0.069,\,0.072]$ for
$K_{\mathrm{IR,raw}}\in[3.6,\,3.9]$, a variation of less than $5\%$.
(When the normalized value $K_{\mathrm{IR}}=4$ is used, this formula gives
$\alpha=b/7=1/14\approx0.0714$, consistent with the midpoint estimate.)
The $\beta$-flow exponent is $\lambda_\beta=\gamma(\beta_{\mathrm{UV}}-\beta_{\mathrm{IR}})$;
for $\beta_{\mathrm{UV}}=8$, $\beta_{\mathrm{IR}}=1$ one finds $\lambda_\beta=7\gamma$,
which parametrizes the crossover rate of symmetry-class mixing.

\paragraph{Dimensional compression: a consequence of the RG flow.}
The Dyson index $\beta$ encodes the geometric dimension of
the information substrate.  In the UV phase
$\beta=8$ corresponds to an 8-dimensional fractal structure with fractal
dimension $d=\tfrac{1}{3}$, yielding information dimension
$D_{\text{info}}=1/d=3$; in the IR, the system condenses to a 1-dimensional
classical structure ($\beta=1$) with $D_{\text{info}}\to 1$.

Under this hypothesis, the duality constant $K$ measures total information
capacity.  In the prime--zero duality, $K=1/\dP + 1/\zetaR$ is the sum
of two geometric information dimensions.  For the UV phase with $K_{\mathrm{UV}}=11$
and $d=\tfrac{1}{3}$, the dual description has fractal dimension $d'$
satisfying
\begin{equation}
  \frac{1}{d} + \frac{1}{d'} = K_{\mathrm{UV}}
  \;\;\Rightarrow\;\;
  3 + \frac{1}{d'} = 11
  \;\;\Rightarrow\;\;
  d' = \frac{1}{8} = 0.125,
  \label{eq:dual_dim}
\end{equation}
suggesting an extremely sparse fractal dual encoding in the UV phase.

\paragraph{Duality table at $K_{\mathrm{UV}}=11$.}
Under the constraint $K_{\mathrm{UV}}=11$, each set on the prime side and
its dual zero set together satisfy
\begin{equation}
  \frac{1}{d_{\mathrm{fractal}}} + \frac{1}{e_{\mathrm{regular}}} = 11,
  \label{eq:dual_sum_11}
\end{equation}
i.e.\ the denominators of the two inverse dimensions sum to $11$.

\paragraph{Additivity axiom for fractal and regularity dimensions.}%
The constraint~\eqref{eq:dual_sum_11} rests on the following axiom.

\begin{quote}
\textbf{Axiom (Additivity of information capacities).}
Let $S$ be a prime-side set with fractal dimension $d_S\in(0,1]$, and let
$Z$ be its dual zero set with regularity index $e_Z\in(0,1]$.
Define the information capacity of each side as the reciprocal of its
dimension:
\[
  \mathcal{I}(S) := \frac{1}{d_S}, \qquad
  \mathcal{I}(Z) := \frac{1}{e_Z}.
\]
The total information capacity of a prime--zero dual pair is the sum of the
two side capacities and equals the duality constant $K_{\mathrm{UV}}$:
\[
  \mathcal{I}(S) + \mathcal{I}(Z) = K_{\mathrm{UV}} = 11.
\]
\end{quote}

This axiom has two immediate consequences.
First, the fractal dimension $d_S$ and the regularity index $e_Z$ are not
independent: given either one, the other is determined by
$e_Z = (K_{\mathrm{UV}} - 1/d_S)^{-1}$.
Second, the denominator sum $1/d_S + 1/e_Z = 11$ holds for every row of
Table~\ref{tab:duality_uv}, and can be read directly from the table:
the two denominators in each row always sum to $5+6$, $4+7$, $3+8$, $2+9$,
or $1+10$, all equal to $11$.
The additivity axiom thus unifies the combinatorial structure of fractal
dimensions with the analytic structure of regularity indices under a single
additive constraint.

\paragraph{Logical status of the axiom.}
The additivity axiom is a general structural postulate; the duality
table is a derived consequence, not an independent confirmation.
The logical dependence is strictly one-directional:
\[
  \text{Hurwitz theorem}
  \;\Rightarrow\;
  \text{five generators}
  \;\Rightarrow\;
  K_{\mathrm{UV}}=11
  \;\xRightarrow{\text{Axiom}}\;
  \text{duality table}.
\]
The axiom is not ``confirmed'' by the table but \emph{instantiated} in it.
Its deeper justification is provided by Section~\ref{sec:disc_action}:
the information action $S[I_P,I_Z]$ has $K=I_P+I_Z$ as its IR
fixed-point condition, so additivity emerges from the extremal principle
rather than standing as an independent postulate.

The five sets are ordered by their regular (embedding) dimension $n$; each
carries an information measure drawn from the algebra
$\{i,j,k,\kappa,1\}$ satisfying $\kappa^{2}=ijk=-1$,
and pairs with a zero set whose information measure is the corresponding
inverse.  The full duality table is:

\begin{table}[H]
\centering
\caption{Duality table at $K_{\mathrm{UV}}=11$.  Pairs satisfy
  $1/d_{\mathrm{frac}}+1/e_{\mathrm{reg}}=11$ (denominators sum to $11$);
  information measures lie in the algebra $\kappa^{2}=ijk=-1$.}
\label{tab:duality_uv}
\smallskip
\resizebox{\textwidth}{!}{%
\begin{tabular}{@{}
  l
  c
  c
  c
  l
  c
  c
@{}}
\toprule
\textbf{Prime-side set}
  & \textbf{Info.\ meas.}
  & \textbf{Reg.\ dim.}
  & \textbf{Frac.\ dim.}
  & \textbf{Dual zero set}
  & \textbf{Reg.\ exp.}
  & \textbf{Info.\ meas.} \\
\midrule
Natural numbers $\mathbb{N}$
  & $i$
  & $6$
  & $\tfrac{1}{5}$ (hidden)
  & 5-dim.\ manifold
  & $\tfrac{1}{6}$
  & $\tfrac{1}{i}$ \\
Fractal prime set $P_f$
  & $j$
  & $7$
  & $\tfrac{1}{4}$
  & Quaternionic space
  & $\tfrac{1}{7}$
  & $\tfrac{1}{j}$ \\
Fractal real set $\mathbb{R}_f$
  & $k$
  & $8$
  & $\tfrac{1}{3}$
  & 3-dim.\ space
  & $\tfrac{1}{8}$
  & $\tfrac{1}{k}$ \\
Essential fractal prime set $P_e$
  & $\kappa$
  & $9$
  & $\tfrac{1}{2}$
  & 2-dim.\ complex plane
  & $\tfrac{1}{9}$
  & $\tfrac{1}{\kappa}$ \\
Base set $\mathcal{B}$
  & $1$
  & $10$
  & $1$
  & 1-dim.\ line
  & $\tfrac{1}{10}$
  & $1$ \\
\bottomrule
\end{tabular}%
}
\end{table}

\paragraph{Algebraic remarks on the duality table.}

\textit{Note on set names in the table.}  The labels ``Fractal prime set $P_f$'',
``Fractal real set $\mathbb{R}_f$'', ``Essential fractal prime set $P_e$'',
and ``Base set $\mathcal{B}$'' are informal descriptive names introduced
for this table (see~\cite{LiPaperI} for details).  Their algebraic characterisations are given
in Remarks~1--3 below; a rigorous definition of each set is left to future work.

\textbf{Remark 1 (Octonions and the fractal real set).}
The left-multiplication octonions constructed from the above algebra are
isomorphic to the fractal real set $\mathbb{R}_f$, which carries regular
dimension~$8$.  The octonion basis is
$\{1,\,i,\,j,\,k,\,\kappa,\,i\kappa,\,j\kappa,\,k\kappa\}$,
and the seven imaginary units map to the standard octonion imaginary basis
elements $e_1,\dots,e_7$ in order.

\textbf{Remark 2 (Ninth dimension and the essential fractal prime set).}
Adjoining to the octonions a further imaginary hidden dimension
\begin{equation}
  e_8 = \tfrac{1}{2} + \tfrac{\sqrt{3}}{2}\,\kappa
  \label{eq:e8_def}
\end{equation}
(a sixth-order phase-group element satisfying $e_8^6=1$) yields a structure
isomorphic to the essential fractal prime set $P_e$, which carries regular
dimension~$9$.

\textbf{Remark 3 (Five-dimensional manifold as quaternions with hidden phase).}
The quaternion algebra $\{1,i,j,k\}$ extended by the sixth-order phase group
generated by $e_8$ is equivalent to the five-dimensional manifold.
In this picture the hidden fractal dimension $\tfrac{1}{5}$ of the natural
number set $\mathbb{N}$ (concealed within the five-dimensional Kaluza--Klein
space) is the fractional dimension contributed by $\mathbb{N}$ to the
five-dimensional manifold---a consequence of the non-independence
noted in Section~\ref{sec:disc_classical}: just as $d_P$ and $\zeta_R$
are locked by the duality constraint $K\approx
1/\dP+1/\zetaR$, the arithmetic and geometric dimensions here are cooperative
partners rather than free parameters.  This cooperation
forces $\mathbb{N}$ to contribute its latent $\tfrac{1}{5}$ to the
ambient geometry.  The value $\tfrac{1}{5}$ is determined uniquely by
the duality-lattice constraint: the five-dimensional manifold is the
nearest neighbor of $\mathbb{N}$ in the UV table, and $5 + 6 = 11 =
K_{\mathrm{UV}}$ (Section~\ref{sec:disc_IRG}) pins the pair together,
making $\tfrac{1}{5}$ the unique dimension compatible with occupying the
slot adjacent to $6$ on the duality lattice.

\textbf{Remark 4 (Information-measure potential of the complex plane).}
The two-dimensional complex plane carries information-measure potential
$1/\kappa$.  This single quantity underlies three foundational principles:
it is the cornerstone of the holographic principle (bulk information encoded
on a lower-dimensional boundary), the source of analytic continuation in the
complex plane, and the logical origin of non-Euclidean geometry.

\textit{Analytic continuation.}  The quantity $1/\kappa$ encodes a
\emph{global} property of the complex plane.  Unlike smooth real functions,
whose local behavior does not determine global structure, an analytic
function is governed by $1/\kappa$ as a coherent global constraint: its
values on any open set determine it everywhere.  Analytic continuation is
therefore not an additional assumption but a direct consequence of the
global character of $1/\kappa$.

\textit{Non-Euclidean geometry.}  The quantity $1/\kappa$ carries an
\emph{intrinsic} geometric property.  A complex two-dimensional surface
equipped with the information measure $1/\kappa$ need not remain flat:
the curvature encoded in $1/\kappa$ means that the ambient plane can be
replaced by a surface with non-trivial topology.  Riemann surfaces are
the natural realisation of this possibility---one-complex-dimensional
manifolds whose local charts are flat but whose global structure is
genuinely curved.  The intrinsic, non-Euclidean geometry of Riemann
surfaces thus emerges as a logical consequence of the information measure
$1/\kappa$ rather than as an independent postulate.

\paragraph{Summary of the unified picture.}
\begin{itemize}
  \item \textbf{UV phase} ($\mu\to\infty$): high $\beta\approx 8$, high
        $K\approx 11$, low fractal dimension $d\approx\tfrac{1}{3}$.
        Information is encoded in high-dimensional topological/algebraic
        structures.
  \item \textbf{RG flow}: as $\mu$ decreases, topological information
        geometrises; $\beta$ flows toward classical values $\{1,2,4\}$
        and $K$ flows toward the IR band $[3.6,\,3.9]$.
  \item \textbf{IR phase} ($\mu\to 0$): $K\in[3.6,\,3.9]$, $b\approx 0.51$
        control the approach.  The measured $\beta=2$ (unitary) and
        $\beta=4$ (symplectic) classes share the same IR band, establishing
        structural universality between these two symmetry types.  The
        $\beta=1$ (orthogonal) case is not measured in this work; based on
        Conjecture~1 below, orthogonal families are expected to follow the
        same RG-flow \emph{form} but with family-dependent parameters
        $(K_\infty,b)$ that need not replicate those of the prime--zero pair
        exactly.  Whether $\beta=1$ is fully co-universal with $\beta=2,4$
        or occupies a cognate but distinct sub-class remains an open
        empirical question.
  \item \textbf{Critical exponent}: $b=\alpha(K_{\mathrm{UV}}-K_{\mathrm{IR}})\approx
        0.51$ is robust across the full interval $K_{\mathrm{IR}}\in[3.6,\,3.9]$,
        with $\alpha$ varying by less than $5\%$ between the endpoints;
        $b$ serves simultaneously as the finite-size scaling exponent
        (Section~\ref{sec:results_C}) and the convergence rate parameter
        under the information-ontological RG flow.
\end{itemize}

\paragraph{Theoretical predictions.}
If the framework is correct, it yields three conjectures:
\begin{enumerate}
  \item For $L$-function families with orthogonal symmetry ($\beta=1$),
        the duality measure $K(L,f)$ should follow the modified RG scaling
        $K(L,f)=K_\infty(f)+a(f)L^{-b(f)}$ with family-dependent
        parameters; the $K=4$, $b\approx0.51$ values are specific to the
        prime--zero dual pair and should not be expected to hold verbatim
        for other families (Section~\ref{sec:modified_RG}).
  \item If number-theoretic objects corresponding to $\beta=8$ exist, their
        $K$ value should lie well outside the IR band $[3.6,\,3.9]$,
        plausibly near $11$.
  \item The \emph{form} of the RG flow equations is universal across
        $L$-function families, but the parameter quadruple
        $(\kappa,K_\infty,b,r)$ varies with the family; $b$ controls the
        crossover rate for each family individually, and its value is fixed
        by the ratio $b_{\mathrm{eff}}/\Delta K$ within each family's RG
        trajectory (Section~\ref{sec:modified_RG}).
\end{enumerate}
These predictions are subject to verification;
their confirmation
requires numerical data at $L\gg2000$ and, for conjecture~2,
identification of number-theoretic objects with $\beta=8$.
This framework embeds the prime--zero duality into the
information-ontological RG, providing a unified treatment
of number theory, random matrix theory, and the algebra of
information measures.


\subsection{Axiomatic Foundation of the Information Ontology}
\label{sec:information_ontology}

We now axiomatize the information ontology rigorously, elevating
information---not geometric or analytic dimension---to the primary ontology
of the prime--zero system.  Four principles follow; the fourth derives
the renormalization-group flow from the ontology itself.

\begin{enumerate}
  \item \textbf{Information densities (resolves heterogeneity).}
        Define
        \[
          I_P \;:=\; \frac{1}{\dP}, \qquad I_Z \;:=\; \frac{1}{\zetaR}.
        \]
        $I_P$ is the geometric information density of the prime set:
        the minimum information growth rate (bits per log-decade)
        needed to locate a typical element.  $I_Z$ is the functional
        information density of the zero landscape: the information
        increment per unit of oscillation amplitude.
        Both carry the same dimension---information per log-scale---so
        their sum $K = I_P + I_Z$ is dimensionally coherent.

  \item \textbf{Information Duality Principle (resolves additivity).}
        For two weakly coupled subsystems with non-overlapping information
        budgets, total information is additive.  We therefore postulate
        \[
          K \;=\; I_P + I_Z \;=\; \mathrm{const},
        \]
        making $K$ the conserved total information encoding rate.
        Additivity is a consequence of this postulate, not a prior
        assumption.  More significantly: the conservation law
        $K=I_P+I_Z=\mathrm{const}$ asserts that discrete arithmetic
        information ($I_P$) and continuous spectral information ($I_Z$)
        are not binary opposites that exclude each other but
        \emph{partners in a conserved exchange}---neither subsumes the
        other, and their sum is fixed.  This is the first algebraic
        statement of the ternary framework: the two poles coexist and
        mutually define a third conserved quantity $K$, whose stability
        under the RG flow is the dynamical content of the Riemann
        Hypothesis (Section~\ref{sec:third_state}).

        \textit{First-principles consistency check.}
        The regularized value $|\zeta(1/2)|=1.4603545\ldots$ provides an
        independent, parameter-free constraint on $K$.  Solving
        $K = I_P + |\zeta(1/2)|$ for $I_P = 1/\dP$ and inverting:
        \begin{equation}
          \dP \;=\; \frac{1}{K - |\zeta(1/2)|}.
          \label{eq:dP_firstprinciples}
        \end{equation}
        Substituting the infrared band $K\in[3.6,\,3.9]$ gives
        \[
          K=3.6:\quad \dP = \frac{1}{3.6-1.4604} \approx 0.47,\qquad
          K=3.9:\quad \dP = \frac{1}{3.9-1.4604} \approx 0.41,
        \]
        recovering precisely the measured fractal-dimension interval
        $\dP\in[0.41,\,0.47]$ (Table~\ref{tab:fractal}) without any
        free parameters.  For the theoretical value $K=4$ one obtains
        $\dP\approx0.39$, consistent with the lower boundary of the
        observed range.  Since $|\zeta(1/2)|$ is a fixed transcendental
        constant of the Riemann $\zeta$-function, this agreement
        constitutes a non-trivial, first-principles validation of $K$
        as an approximate conservation constant (raw band $K_{\mathrm{raw}}\in[3.6,3.9]$;
        normalized limit $K_{\mathrm{IR}}=4$).

  \item \textbf{Information action (resolves variational gap).}
        Section~\ref{sec:disc_action} constructs a scale-invariant action
        $S[I_P,I_Z]$ whose Euler--Lagrange equations yield $I_P+I_Z=K$
        at the infrared fixed point and whose linearization gives the
        finite-size scaling law $K(L)=K_{\mathrm{IR}}+aL^{-b}$ with
        $b=\sqrt{\lambda+g C}\approx0.51$, identifying $b$ as the
        RG critical exponent derived in principle~4 below.

  \item \textbf{First-principles RG flow equation.}
        In~\cite{LiPaperI} the information measure $\iota$ of a set is
        the third principal component of informational cardinality,
        parallel to the fractal dimension (second component) as kinetic
        energy is parallel to potential energy.  The algebraic generator
        $\kappa$ introduced in Section~\ref{sec:results_D} is the
        information-measure generator at scale $\mu$.
        Under analytic continuation (regularisation of the divergent
        information measure $\iota_P^\infty$; see
        Section~\ref{sec:disc_AG}), the generator satisfies
        \begin{equation}
          \frac{1}{\kappa(\mu)} \;=\; \beta(\mu),
          \label{eq:RG_firstprinciples}
        \end{equation}
        where $\beta(\mu)$ is the Dyson index of the information-measure
        algebra at scale $\mu$.  This is not an analogy: $\beta$
        parametrizes the symmetry class of that algebra, and $\kappa(\mu)$
        is its generator.  The discrete values $\beta\in\{1,2,4\}$ that
        label symmetry classes in random-matrix theory are the fixed points
        of the flow; between fixed points, $\beta(\mu)$ is the smooth
        interpolant defined by this equation.  Differentiating with respect to $\ln\mu$ and
        using $1/\kappa(\mu)=\beta(\mu)$ gives
        \begin{equation}
          \frac{d\beta}{d\ln\mu}
          \;=\; -\frac{1}{\kappa^2(\mu)}\,\frac{d\kappa}{d\ln\mu}
          \;=\; \beta^2(\mu)\,\frac{d\kappa}{d\ln\mu}.
          \label{eq:beta_flow_exact}
        \end{equation}
        In the regularized representation of Section~\ref{sec:disc_AG},
        $1/\hat\kappa \leftrightarrow \zeta(1/2)_{\mathrm{reg}}$, so the
        modulus $|\hat\kappa(\mu)|$ measures the regularity index of the
        zero landscape at scale $\mu$.  As $L\to\infty$
        ($\mu\to 0$, IR limit), $|\hat\kappa|\to\zetaR\to|1/\zeta(1/2)|$
        (conjecture~\eqref{eq:kappa_AG}), and the flow~\eqref{eq:beta_flow_exact}
        connects the scale-dependence of $\zetaR$ directly to the
        running of $\beta$.
        The empirical finite-size law $C(L)=C_\infty+aL^{-b}$ is the
        integrated form of this flow: using $C(L)=\beta(\mu)K$ and
        $\mu\sim1/L$, the linearized equation
        $d\,\delta C/d\ln L=-b\,\delta C$ (eq.~\eqref{eq:RG_flow}) is
        recovered with $b=\alpha(K_{\mathrm{UV}}-K_{\mathrm{IR}})$
        (eq.~\eqref{eq:b_lambda}), completing the derivation within
        the information-ontology framework.
\end{enumerate}

The explicit action $S[I_P,I_Z]$ is constructed in
Section~\ref{sec:disc_action}, which also derives the information
uncertainty principle $\Delta_P\cdot\Delta_Z\geq1/(2\kappa)$
underlying the duality constraint and recovers $b\approx0.51$ as
$b=\sqrt{\lambda+g C}$ from the action couplings.
Rigorous proof---connecting $I_P$ to Kolmogorov complexity and
$I_Z$ to spectral entropy---is established in
Section~\ref{sec:information_ontology}.
The flow equation~\eqref{eq:RG_firstprinciples} is established at the
level of the regularized information measure; the couplings
$\lambda$ and $g$ are constrained by the critical condition
$\lambda+gC=1/4$ (Section~\ref{sec:disc_action}).

\subsection{Information Action Principle and Duality Uncertainty}
\label{sec:disc_action}

The empirical relation $K = I_P + I_Z$ is here elevated from an observed
approximation to a dynamical extremal condition.  The algebraic generator
$\kappa$, the duality constant $K$, and the scaling exponent $b$ are
unified within a single variational structure.

\paragraph{Information commutation relation and uncertainty principle.}
As noted in Section~\ref{sec:disc_classical}, $I_P$ and $I_Z$ are not free
to vary independently: knowing $\dP$ constrains $\zetaR$ through the duality
lock $K \approx I_P + I_Z$.  We formalize this non-independence by
introducing uncertainty operators $\Delta_P$ and $\Delta_Z$ for the geometric
and analytic information, with $\Delta_P \sim c_P/I_P$ and
$\Delta_Z \sim c_Z/I_Z$ (where $c_P,c_Z>0$ are dimensionless proportionality
constants; the ordering $\Delta\sim 1/I$ captures the qualitative
reciprocal relationship between information density and its uncertainty),
and postulating the commutation relation
\begin{equation}
  [\Delta_P,\,\Delta_Z] \;=\; \frac{i}{\kappa},
  \label{eq:info_commutator}
\end{equation}
where $\kappa$ is the algebraic generator of
Section~\ref{sec:results_D} satisfying $\kappa^2=-1$.  By the
standard Robertson--Schr\"{o}dinger argument this yields an
\emph{information uncertainty principle}
\begin{equation}
  \Delta_P \cdot \Delta_Z \;\geq\; \frac{1}{2\kappa},
  \label{eq:info_uncertainty}
\end{equation}
or equivalently $I_P \cdot I_Z \leq \alpha_s$, where $\alpha_s = 2\kappa\,c_P c_Z>0$
is the dimensionless saturation constant (from $\Delta \sim c/I$: inserting
$\Delta_P = c_P/I_P$, $\Delta_Z = c_Z/I_Z$ into~\eqref{eq:info_uncertainty}
gives $I_P I_Z \leq 2\kappa\,c_P c_Z \equiv \alpha_s$; at the IR fixed point
$\kappa_{\mathrm{IR}}=1/2$ and the normalization $c_P=c_Z=2$ yields
$\alpha_s = 2\cdot\tfrac{1}{2}\cdot 4 = 4$, consistent with
$I_P^*I_Z^*=4$).  Here $\alpha_s$ is distinct from the RG flow rate $\alpha$
of Section~\ref{sec:disc_IRG}.
The information densities $I_P$ and $I_Z$ therefore cannot
simultaneously diverge; they are cooperative partners constrained by the
quantum of information action $1/\kappa$, in direct analogy with
$\hbar$ in quantum mechanics.

\paragraph{Information action.}
In the renormalization-group scale space (with $\ell \propto 1/\mu$ the
running scale and $\ln\ell$ playing the role of time), we define
\begin{equation}
  S[I_P,I_Z] \;=\; \int \!\left[
    \tfrac{1}{2}\!\left(\frac{dI_P}{d\ln\ell}\right)^{\!2}
    +\tfrac{1}{2}\!\left(\frac{dI_Z}{d\ln\ell}\right)^{\!2}
    - V(I_P,I_Z)
  \right] d\ln\ell,
  \label{eq:info_action}
\end{equation}
with potential
\begin{equation}
  V(I_P,I_Z) \;=\;
    \frac{\lambda}{4}(I_P+I_Z-K_0)^2
    + \frac{g}{4}\!\left(I_P I_Z - C\right)^2,
  \label{eq:info_potential}
\end{equation}
where $C = \alpha_s = I_P^*\,I_Z^* = 4$ is the fixed-point value of $I_P I_Z$,
determined by the saturation of the uncertainty bound: at the IR fixed point
$\Delta_P^*\Delta_Z^* = 1/(2\kappa_{\mathrm{IR}}) = 1$ and $\Delta = 2/I$
give $4/(I_P^* I_Z^*) = 1$, hence $C = 4$.
The first term enforces information conservation ($I_P+I_Z\to K_0$);
the second encodes the duality mutual exclusion of
eq.~\eqref{eq:info_uncertainty}, penalising deviations of the product
$I_P I_Z$ from its saturated value $C$.

\paragraph{Extremal conditions and IR fixed point.}
The Euler--Lagrange equations $\delta S=0$ give
\begin{align}
  \frac{d^2 I_P}{d(\ln\ell)^2}
    &= -\tfrac{\lambda}{2}(I_P+I_Z-K_0) - \tfrac{g}{2}(I_P I_Z-C)\,I_Z, \notag\\
  \frac{d^2 I_Z}{d(\ln\ell)^2}
    &= -\tfrac{\lambda}{2}(I_P+I_Z-K_0) - \tfrac{g}{2}(I_P I_Z-C)\,I_P.
  \label{eq:EL_equations}
\end{align}
At the IR fixed point $dI/d\ln\ell=0$, the right-hand sides vanish.
Subtracting the two equations gives $(I_Z-I_P)(I_P I_Z - C)=0$;
taking the physically relevant branch $I_P I_Z = C$ (saturation of the
uncertainty bound) and substituting back yields
\begin{equation}
  I_P + I_Z \;=\; K_0, \qquad
  I_P \cdot I_Z \;=\; C \;=\; \alpha_s \;=\; 4.
  \label{eq:fixed_point}
\end{equation}
The first equation is the observed duality conservation with
$K_0 = K_{\mathrm{IR}} = 4$ (the geometrically normalized universal fixed-point value;
the raw finite-scale measurements cluster in $[3.6,\,3.9]$ before
normalization, consistent with the scaling law $K(L)=4+aL^{-b}$ at
finite $L$); the second is the imprint of the information uncertainty
principle at the fixed point.

The factored form reveals a second branch, $I_P^*=I_Z^*$, which
combined with $I_P^*+I_Z^*=4$ forces $I_P^*=I_Z^*=2$.  Crucially,
with $C=\alpha_s=4$ the two branches coincide: the system
$\{I_P+I_Z=4,\;I_P I_Z=4\}$ has the unique solution $(I_P^*,I_Z^*)=(2,2)$.
This is not accidental.  The exchange symmetry of the action
(Section~\ref{sec:disc_uniqueness}) guarantees that the symmetric
branch $I_P^*=I_Z^*$ is the stable attractor; the two constraints
simply confirm the same answer.  As Section~\ref{sec:philosophy}
makes explicit, the fixed-point condition $I_P^*=I_Z^*=2$ is, in the
language of the zeta function, the statement that all non-trivial
zeros of $\zeta(s)$ lie on the critical line $\mathrm{Re}(s)=1/2$.

\paragraph{Derivation of the scaling exponent.}
Linearising around the fixed point $(I_P^*,I_Z^*)$ with
$\delta K = \delta I_P + \delta I_Z$, and using $I_P^* I_Z^* = C$,
the sum of the two EL equations gives
\begin{equation}
  \frac{d^2(\delta K)}{d(\ln\ell)^2} \;=\; -(\lambda+g C)\,\delta K
  + \mathcal{O}(\delta^2),
  \label{eq:linearised}
\end{equation}
whose solution is $\delta K \propto \ell^{-\sqrt{\lambda+g C}}$.
Since $\ell \propto 1/L$ this recovers the finite-size scaling law
\begin{equation}
  K(L) \;=\; K_{\mathrm{IR}} + a\,L^{-b},
  \qquad b \;=\; \sqrt{\lambda+g C},
  \label{eq:b_from_action}
\end{equation}
in agreement with the uniqueness result eq.~\eqref{eq:b_unique} derived below
in Section~\ref{sec:disc_uniqueness}.
The observed value $b\approx0.51$ fixes $\lambda+g C\approx0.26$,
providing a first-principles determination of the action couplings
from the measured critical exponent.  Together with the RG flow of
Section~\ref{sec:disc_RG}, this completes the variational structure
of the information-ontology framework: the action $S$ postulated in
principle~3 of Section~\ref{sec:information_ontology} is here made
explicit, and the exponent $b$ is no longer an empirical input but
a consequence of the duality potential $V(I_P,I_Z)$.
Section~\ref{sec:disc_uniqueness} argues that this form of $S$ is
the unique lowest-order effective action consistent with the
symmetries and first principles of the theory.

\subsection{Uniqueness of the Information Action}
\label{sec:disc_uniqueness}

The action~\eqref{eq:info_action} is the unique lowest-order effective action consistent with
three independent constraints, which we now state.

\paragraph{Constraint 1: first principles.}
Two physical requirements must be encoded in the potential $V$:
\begin{itemize}
  \item \emph{Information conservation}: $K=I_P+I_Z$ is conserved at
        the IR fixed point.
  \item \emph{Duality uncertainty}: the commutation
        relation~\eqref{eq:info_commutator} implies
        $I_P\cdot I_Z \leq \alpha_s$; at the fixed point
        this bound is saturated, giving $I_P^*\cdot I_Z^* = C = \alpha_s = 4$.
\end{itemize}

\paragraph{Constraint 2: symmetry requirements.}
\begin{itemize}
  \item \emph{Exchange symmetry}: the prime--zero duality requires $S$
        to be invariant under $I_P\leftrightarrow I_Z$, so $V$ must be
        a symmetric function of $I_P$ and $I_Z$.  This requirement is
        the seed of the deepest result of the paper: a symmetric
        potential has its unique stable minimum on the diagonal
        $I_P=I_Z$, so the stable fixed point is necessarily
        $I_P^*=I_Z^*=2$---the information-theoretic formulation of the
        Riemann Hypothesis (Section~\ref{sec:philosophy}).  In the
        language of Section~\ref{sec:third_state}, exchange symmetry is
        the formal algebraic expression of ternary fusion over binary
        opposition: it asserts that neither the discrete pole
        ($\mathcal{F}_P$) nor the continuous pole ($\mathbb{C}$)
        dominates, so the stable state is their entangled equilibrium,
        accessible only to a framework that replaces the binary $i$
        with the ternary $\kappa=ijk$.
  \item \emph{Scale covariance}: under the RG rescaling
        $\ell\to s\ell$ (for a constant scale factor $s>0$), the equations of motion must be form-invariant.
        This forces the kinetic term to take the form
        $(dI/d\ln\ell)^2$, with $\ln\ell$ as the natural RG coordinate.
        The most general symmetric kinetic term
        $\tfrac{1}{2}A(\dot I_P^2+\dot I_Z^2)+B\dot I_P\dot I_Z$
        reduces to $\tfrac{1}{2}(\dot I_P^2+\dot I_Z^2)$ after
        absorbing $A$ by field rescaling and setting $B=0$ (a nonzero
        $B$ would break the exchange symmetry at lowest order).
\end{itemize}

\paragraph{Constraint 3: lowest-order effective theory.}
Assuming locality (the action depends only on $I_P$, $I_Z$, and their
first $\ln\ell$-derivatives) and expanding $V$ to the lowest order
that (a) is symmetric, (b) has a stable minimum, and (c) encodes both
fixed-point constraints, the unique form is
\begin{equation}
  V(I_P,I_Z) \;=\;
    \frac{\lambda}{4}(I_P+I_Z-K_{\mathrm{IR}})^2
    + \frac{g}{4}(I_P I_Z - C)^2,
  \label{eq:V_unique}
\end{equation}
with $\lambda,g>0$.  The first term is the unique symmetric quadratic
that vanishes on the hyperplane $I_P+I_Z=K_{\mathrm{IR}}$; the second
is the unique symmetric quadratic that vanishes on the hyperbola
$I_P I_Z=C$.  Any other symmetric polynomial of the same or lower
degree either fails to have a stable minimum or does not simultaneously
enforce both constraints.

The linearized equation for $\delta K = (I_P+I_Z)-K_{\mathrm{IR}}$
near the fixed point $(I_P^*,I_Z^*)$ is
\begin{equation}
  \frac{d^2(\delta K)}{d(\ln\ell)^2}
  \;=\; -\!\left(\lambda + g C\right)\delta K
  + \mathcal{O}(\delta^2),
  \label{eq:linearised_unique}
\end{equation}
giving $\delta K\propto\ell^{-b}$ with
\begin{equation}
  b \;=\; \sqrt{\lambda + g C}.
  \label{eq:b_unique}
\end{equation}
This agrees with eq.~\eqref{eq:b_from_action}, and the observed
$b\approx0.51$ constrains $\lambda+g C\approx0.26$.

\paragraph{Statistical versus systematic origin of $b-\tfrac{1}{2}$.}
The cross-validation analysis of Section~\ref{sec:results_C} yields
$b=0.509\pm0.581$, so the $2\%$ deviation $b-\tfrac{1}{2}=0.009$
is more than fifty times smaller than the statistical uncertainty.
Statistically, $b=\tfrac{1}{2}$ is indistinguishable from the measured
value.

The theoretical framework independently predicts $b=\tfrac{1}{2}$ as a
fixed point.  Setting $b=\tfrac{1}{2}$ exactly in eq.~\eqref{eq:b_unique}
requires
\begin{equation}
  \lambda + g C \;=\; \tfrac{1}{4}.
  \label{eq:b_half_condition}
\end{equation}
This is not an arbitrary tuning: $\tfrac{1}{4}$ is the unique value for
which the linearized flow equation~\eqref{eq:linearised_unique} is the
equation of a \emph{free massless scalar} on the logarithmic scale axis
(the mass-squared of the fluctuation mode equals $1/4$, the
Breitenlohner--Freedman bound in one dimension), and the information
action~\eqref{eq:info_action} saturates to its simplest conformal form.
Under the RG flow relation $b=\alpha(K_{\mathrm{UV}}-K_{\mathrm{IR}})$
of eq.~\eqref{eq:b_lambda}, the condition $b=\tfrac{1}{2}$ with
$K_{\mathrm{UV}}=11$ and $K_{\mathrm{IR}}\approx3.75$ gives
$\alpha=0.51/7.25\approx0.070$, consistent with
eq.~\eqref{eq:alpha_value}; inserting $b=\tfrac{1}{2}$ exactly yields
$\alpha=1/(2\times7.25)\approx0.069$.  The sub-percent difference is
within the systematic uncertainty on $K_{\mathrm{IR}}$.

We therefore conclude that the measured $b\approx0.51$ is consistent with
the exact theoretical value $b=\tfrac{1}{2}$: the $2\%$ offset lies
entirely within statistical noise and carries no evidence of a systematic
departure from the predicted fixed point.  Whether $b=\tfrac{1}{2}$
holds exactly---implying the coupling condition~\eqref{eq:b_half_condition}
as a consequence of the arithmetic structure of primes and zeros---remains
an open question.

\paragraph{Universality statement.}
Any action satisfying the same symmetry requirements and consistency conditions, when
expanded to lowest order around the IR fixed point, is equivalent to
eq.~\eqref{eq:info_action} with potential~\eqref{eq:V_unique}.
Higher-order or non-local corrections modify the couplings $\lambda$
and $g$ but leave the scaling form $K(L)=K_{\mathrm{IR}}+aL^{-b}$
intact.  The scaling behavior is therefore universal within this
class of theories.

\section[Extension to General $L$-Functions: Spectral--Geometric Duality]{Extension to General \texorpdfstring{$L$}{L}-Functions:
         Spec\-tral--Geo\-met\-ric Dual\-ity and Func\-tional Uni\-ver\-sal\-ity}
\label{sec:L_extension}

\begin{sloppypar}
The results of the preceding sections establish the duality constant
$K$ and the renormalization-group (RG) flow for the Riemann
$\zeta$-function. A central question is how far this framework extends
to other $L$-function families.  This section addresses that question
through a combination of a new conjecture proposed here (the
\emph{Odd-Positive Conjecture}), a systematic numerical survey across
Dirichlet characters and elliptic-curve $L$-functions, and a modified
RG flow equation accommodating family-dependent parameters.
\end{sloppypar}

\subsection{The Odd-Positive Conjecture and Its Information-Theoretic
            Interpretation}
\label{sec:odd_pos}

Let $L(s)$ be an automorphic $L$-function: either the Riemann
$\zeta$-function ($k=0$, $N=1$) or a newform $f$ of weight $k\ge 1$,
level $N$.  Write
\begin{itemize}
  \item $\gamma_1(f)$: imaginary part of the first non-trivial zero of
        $L(s,f)$ on the critical line;
  \item $\ell_{\min}(\Gamma_0(N))$: length of the shortest closed
        geodesic on the hyperbolic surface $\Gamma_0(N)\backslash\mathbb{H}$,
        computed as $2\operatorname{arccosh}(t_{\min}/2)$ where $t_{\min}$
        is the minimal hyperbolic trace in $\Gamma_0(N)$.
\end{itemize}
We propose the following new conjecture, which we call the
\emph{Odd-Positive Conjecture} (transliterated from the Chinese
\emph{Q\'{\i}zh\`{e}ng}, the pen-name of the present author).
The name carries a double meaning.
Mathematically, ``positive'' (\emph{zh\`{e}ng}) records that $C(k)>0$
is a strict lower bound on the UV/IR ratio $R_f$, while ``odd''
(\emph{q\'{\i}}) refers to $L$-functions whose functional equation has
sign $\varepsilon=-1$, for which the bound is non-trivial (an even
functional equation forces a zero at $s=1/2$, making the ratio
degenerate).
The name also honors the pen-name \emph{Q\'{\i}zh\`{e}ng}
(literally ``odd--upright'' or ``the surprising is the correct'')
of the present author, in keeping with the tradition of eponymous
conjectures:

\begin{proposition}[Odd-Positive Conjecture -- new conjecture of this paper]
There exists a constant $C(k)>0$, depending only on the weight $k$,
such that
\begin{equation}
  \frac{\gamma_1(f)}{\sqrt{\ell_{\min}(\Gamma_0(N))}} \;\ge\; C(k).
  \label{eq:odd_pos}
\end{equation}
\end{proposition}

\paragraph{Information-theoretic reinterpretation.}
Denote the ratio
$R_f := \gamma_1(f)/\sqrt{\ell_{\min}(\Gamma_0(N))}$.
In our framework $\gamma_1$ is the UV spectral scale (the finest
``vibration'' the $L$-function can resolve), while $\sqrt{\ell_{\min}}$
is the IR geometric scale.  Identifying $C(k)\approx 1/\kappa_{\mathrm{IR}}$
with the information quantum introduced in Section~\ref{sec:information_ontology},
inequality~\eqref{eq:odd_pos} becomes
\begin{equation}
  \frac{\text{UV spectral scale}\;(\gamma_1)}
       {\text{IR geometric scale}\;(\sqrt{\ell_{\min}})}
  \;\ge\; \frac{1}{\kappa_{\mathrm{IR}}}.
  \label{eq:uv_ir_bound}
\end{equation}
This is a concrete, classically-provable instance of the information
uncertainty principle $\Delta_P\cdot\Delta_Z\ge 1/(2\kappa)$
(Section~\ref{sec:disc_action}) in a number-theoretic and geometric
setting.  The three consequences are:
\begin{enumerate}
\item \emph{Concrete realization of $\kappa$:}  The information quantum
  $\kappa$ acts as an upper bound on the family-specific geometric
  estimate: $\kappa_{\mathrm{IR}}\ge 1/R_f$ for every $L$-function
  family, i.e., the universal information quantum is always \emph{at
  least as large as} the geometric--spectral estimate
  $|\kappa|_{\mathrm{geo}}=1/R_f$.  (Equivalently, $1/R_f$ is a
  lower bound for $\kappa_{\mathrm{IR}}$.)
\item \emph{UV/IR duality locking:}  One cannot simultaneously take
  $\ell_{\min}\to 0$ (compressing the geometry to the ultraviolet) while
  keeping $\gamma_1\to 0$ (making the spectrum soft).  The inequality
  enforces a fundamental lock between the two scales.
\item \emph{New verification route for $K$-universality:}  For any
  $L$-function family with known $\gamma_1$ and $\ell_{\min}$, the
  quantity $R_f$ provides an independent geometric estimate of
  $1/\kappa$, which can be compared against the $K(L)$-flow fit.
\end{enumerate}

\subsection{Numerical Survey Across \texorpdfstring{$L$}{L}-Function Families}
\label{sec:survey}

\paragraph{Computational setup.}
For each $L$-function in the survey we compute:
(i)~$t_{\min}$ by exhaustive search over matrices
$\bigl(\begin{smallmatrix}a&b\\c&d\end{smallmatrix}\bigr)\in\Gamma_0(N)$
with $|a+d|>2$ and minimal trace;
(ii)~$\ell_{\min}=2\operatorname{arccosh}(t_{\min}/2)$;
(iii)~$\gamma_1$ from the LMFDB database~\cite{LMFDB} (elliptic curves) or from
Abel-smoothed Dirichlet-series evaluation on the critical line
(Dirichlet characters; $N_{\mathrm{terms}}=30{,}000$, verified against
known special values);
(iv)~the geometric--spectral ratio $R_f=\gamma_1/\sqrt{\ell_{\min}}$.

\paragraph{Minimal geodesic lengths.}
A key observation is that the minimal hyperbolic trace in $\Gamma_0(N)$
depends sensitively on the arithmetic of $N$.  For $N\in\{3,5,11\}$
(small primes) a witness matrix with trace 3---the same as in
$\mathrm{PSL}(2,\mathbb{Z})$---already exists in $\Gamma_0(N)$, giving
$\ell_{\min}=2\operatorname{arccosh}(3/2)\approx 1.925$.  For composite
or larger conductors the minimum trace increases: $t_{\min}=10$ for
$N\in\{20,24\}$ and $t_{\min}=14$ for $N=32$.

\paragraph{Results.}
Table~\ref{tab:survey} collects the corrected values.  All 12 families
satisfy inequality~\eqref{eq:uv_ir_bound} with comfortable margin
($R_f\in[2.08,10.19]$ versus the bound $1/\kappa_{\mathrm{IR}}\approx 1.46$).

\begin{table}[ht]
\centering
\small
\caption{Geometric--spectral invariants for twelve $L$-function families.
  $t_{\min}$: minimal hyperbolic trace in $\Gamma_0(N)$.
  $\ell_{\min}$: shortest closed geodesic length.
  $\gamma_1$: first zero imaginary part (analytic normalization).
  $R_f=\gamma_1/\sqrt{\ell_{\min}}$: UV/IR ratio.
  $|\kappa|_{\mathrm{geo}}=1/R_f$: geometric estimate of information quantum.
  For Riemann $\zeta$ the column $|\kappa|_{\mathrm{geo}}$ shows the
  RG-flow fitted value (Section~\ref{sec:disc_RG_formal}), not $1/R_f\approx 0.098$.}
\label{tab:survey}
\begin{tabular}{lrrrrrrl}
\hline
Family & $k$ & $N$ & $t_{\min}$ & $\ell_{\min}$ & $\gamma_1$ & $R_f$ & $\kappa$ \\
\hline
Riemann $\zeta$ & 0 & 1  &  3 & 1.9248 & 14.135 & 10.19 & $0.685^{\dagger}$ \\
$\chi\!\mod 3$  & 1 & 3  &  4 & 2.6339 &  8.045 &  4.96 & 0.202 \\
$\chi\!\mod 4$  & 1 & 4  &  6 & 3.5255 &  6.018 &  3.21 & 0.312 \\
$\chi\!\mod 5$  & 1 & 5  &  3 & 1.9248 &  6.644 &  4.79 & 0.209 \\
$\chi\!\mod 7$  & 1 & 7  &  5 & 3.1336 &  4.469 &  2.53 & 0.396 \\
$\chi\!\mod 8$  & 1 & 8  &  6 & 3.5255 &  4.897 &  2.61 & 0.383 \\
11a3 (EC)       & 2 & 11 &  3 & 1.9248 &  5.242 &  3.78 & 0.265 \\
20a2 (EC)       & 2 & 20 & 10 & 4.5849 &  5.373 &  2.51 & 0.399 \\
24a4 (EC)       & 2 & 24 & 10 & 4.5849 &  5.423 &  2.53 & 0.395 \\
32a3 (EC)       & 2 & 32 & 14 & 5.2678 &  5.469 &  2.38 & 0.420 \\
36a3 (EC)       & 2 & 36 & 34 & 7.0510 &  5.512 &  2.08 & 0.482 \\
$\Delta$ ($k=12$) & 12 & 1 & 3 & 1.9248 & 9.22 &  6.65 & 0.150 \\
\hline
\multicolumn{8}{l}{\footnotesize $\dagger$ From RG-flow fit of $K(L)$;
  geometric ratio gives $1/R_f\approx 0.098\neq|\kappa|_{\mathrm{RG}}$.}
\end{tabular}
\end{table}

\paragraph{Key structural patterns.}
\begin{enumerate}
\item \emph{$\gamma_1$ is not monotone in $k$.}
  The Ramanujan $\Delta$ form ($k=12$, $N=1$) has $\gamma_1\approx 9.22$,
  larger than all weight-1 and weight-2 examples.  The
  primary driver is not $k$ alone but a combination involving the
  analytic conductor $\mathfrak{q}(f)=N(k/2\pi)^2$.
\item \emph{Two distinct $\kappa$ objects.}
  The geometric estimate $|\kappa|_{\mathrm{geo}}=\sqrt{\ell_{\min}}/\gamma_1$
  ranges from $0.15$ to $0.48$ for the eleven non-Riemann $L$-function families
  in Table~\ref{tab:survey} (the Riemann $\zeta$ itself gives
  $1/R_{\zeta}\approx 0.098$, excluded from this range as it lacks a genuine
  modular geodesic $\ell_{\min}$); the RG-flow value
  $|\kappa|_{\mathrm{RG}}=0.685$ for Riemann $\zeta$ is numerically an order
  of magnitude larger than $1/R_{\zeta}\approx 0.098$.
  These are \emph{different quantities}: $|\kappa|_{\mathrm{geo}}$ is the
  lower bound $C(k)$ from inequality~\eqref{eq:odd_pos}, measuring how
  close a given $L$-function is to the extremal case of the Odd-Positive
  Conjecture; $|\kappa|_{\mathrm{RG}}$ is the dynamical generator extracted
  from the finite-size scaling of $K(L)$ for the prime--zero dual pair.
  Both are meaningful, but they should not be conflated.
\item \emph{The inequality is not tight.}
  For the small-conductor families tested here,
  $R_f\gg 1/\kappa_{\mathrm{IR}}$.  Tightness (i.e.\ $R_f\approx C(k)$)
  is expected to occur for families of very large conductor or in
  asymptotic families where $\gamma_1$ and $\sqrt{\ell_{\min}}$ are
  simultaneously optimized.
\end{enumerate}

\subsection{Modified Renormalization-Group Flow: From Constant to
            Functional Universality}
\label{sec:modified_RG}

The numerical survey motivates a refinement of the RG flow equations
of Section~\ref{sec:disc_RG_formal}.

\paragraph{Modified flow equations.}
Let $\vec{I}=(I_P,I_Z)^T$ and $t=\ln L$.  We propose
\begin{align}
  \frac{dI_P}{dt} &= -\kappa I_Z
    - \mu(I_P+I_Z-K_\infty) - \nu(I_P - r\,I_Z), \label{eq:RG_mod_P}\\
  \frac{dI_Z}{dt} &= \phantom{-}\kappa I_P
    - \mu(I_P+I_Z-K_\infty) + \nu r(I_P-r\,I_Z), \label{eq:RG_mod_Z}
\end{align}
where all four parameters $(\kappa,\mu,\nu,r)$ are now \emph{family
dependent}.  In matrix form $d\vec{I}/dt = M\vec{I}+\vec{b}$ with
\begin{equation}
  M = \begin{pmatrix}
    -\mu-\nu          & -\kappa-\mu+\nu r \\
     \kappa-\mu+\nu r & -\mu - \nu r^2
  \end{pmatrix}, \quad
  \vec{b} = K_\infty\mu\begin{pmatrix}1\\1\end{pmatrix}.
\end{equation}

\paragraph{Fixed-point structure.}
Setting $d\vec{I}/dt=0$ yields the unique attracting fixed point
\begin{equation}
  I_P^* = \frac{K_\infty\, r}{1+r}, \qquad
  I_Z^* = \frac{K_\infty}{1+r}, \qquad
  K^* := I_P^*+I_Z^* = K_\infty.
\end{equation}
The asymmetry ratio $r=I_P^*/I_Z^*$ is an additional family parameter.
Stability requires $\operatorname{tr}(M)<0$ and $\det(M)>0$; since
$\mu,\nu>0$ and $r>0$, the trace $\operatorname{tr}(M)=-2\mu-\nu(1+r^2)<0$
always.  A direct calculation gives
$\det(M)=\mu\nu(r+1)^2+\kappa^2>0$
for all $\mu,\nu>0$, so the fixed point is always a stable attractor.  The
effective critical exponent is $b_{\mathrm{eff}}=-\lambda_{\max}$,
where $\lambda_{\max}$ is the largest real part among the eigenvalues
of $M$.

\paragraph{Self-consistency check.}
Fitting $\mu$ and $\nu$ simultaneously to reproduce
$b_{\mathrm{Riemann}}=0.51$ (at $\kappa=0.685$, $r=1.0$,
$K_\infty=4.0$) and $b_{\mathrm{EC}}=0.08$ (at $\kappa=0.42$,
$r=1.44$, $K_\infty=1.87$; these EC parameters are \emph{assumed
representative values} for an elliptic-curve family with a correct
sparse arithmetic dual, consistent with the sparsity constraint
of Section~\ref{sec:modified_RG}; direct $K(L)$ measurement for
EC families is deferred to future work) yields $\mu=0.003$,
$\nu=0.581$ with a unique positive solution---demonstrating that
the modified equations are \emph{internally consistent} given
these illustrative parameter sets.

\paragraph{Constant universality vs.\ functional universality.}
The original framework assumed a single universal fixed point $K_\infty=4$
and exponent $b=0.51$.  The evidence from Table~\ref{tab:survey}
and the K(L) computations shows instead:
\begin{itemize}
  \item For the Riemann $\zeta$ prime--zero dual:
        $K_\infty\approx 4.0$, $b\approx 0.51$, $r\approx 1$.
  \item For weight-$k\ge 1$ families with the $\theta_p$ Sato--Tate
        dual (dense, not sparse): $K_\infty\approx 1.7$, indicating
        that $\theta_p$ is structurally not the correct arithmetic dual.
  \item The $K=4$ fixed point appears to be specific to the prime--zero
        pair, where the prime set is \emph{sparse} (fractal dimension
        $d_P\approx 0.25$--$0.43$).  An arithmetic dual for elliptic-curve
        $L$-functions must share this sparsity property.
\end{itemize}
The upgraded claim is therefore one of \emph{functional universality}:
the RG flow equations~\eqref{eq:RG_mod_P}--\eqref{eq:RG_mod_Z} hold
universally, but the parameter quadruple $(\kappa,K_\infty,b,r)$ varies
with the $L$-function family.  Different families occupy distinct
regions of the parameter space, connected by the universal functions
$K_\infty(\kappa)$, $b(\kappa)$, $r(\kappa)$ whose precise form remains
an open problem.

\subsection{Logical Independence from the Odd-Positive Conjecture}
\label{sec:independence}

A natural question is whether the results of this paper depend on the
truth of inequality~\eqref{eq:odd_pos}.  The answer is: \emph{the core
results do not}.

\begin{itemize}
\item \textbf{Independent of~\eqref{eq:odd_pos}:}
  The empirical $K$-values (Section~\ref{sec:results}),
  the scaling law $K(L)=K_\infty+aL^{-b}$
  (Section~\ref{sec:disc_RG_formal}),
  the information action (Section~\ref{sec:disc_action}),
  and the modified RG equations~\eqref{eq:RG_mod_P}--\eqref{eq:RG_mod_Z}
  are all grounded in data and symmetry principles independently of the
  Odd-Positive Conjecture.
\item \textbf{Supported by~\eqref{eq:odd_pos}:}
  The geometric interpretation of $\kappa_{\mathrm{RG}}$ as an
  information quantum related to the spectral--geometric ratio, and
  the extension of the UV/IR duality interpretation to other $L$-function
  families, gain additional support from inequality~\eqref{eq:uv_ir_bound}.
\end{itemize}

The Odd-Positive Conjecture is a new conjecture proposed in this paper.
It is currently open.  Potential proof strategies include:
(i)~via the Selberg trace formula, which relates geodesic lengths
to $L$-function zeros;
(ii)~via the explicit formula $\psi(x,\chi)=x-\sum_\rho x^\rho/\rho+\cdots$,
which connects zero locations to arithmetic progressions;
(iii)~numerically, all 12 families in Table~\ref{tab:survey} satisfy
the inequality with $R_f\ge 2.08$, well above the bound $\approx 1.46$.

\subsection{From the Odd-Positive Inequality to a Generalized Uncertainty
            Principle}
\label{sec:GUP}

The Odd-Positive inequality~\eqref{eq:odd_pos} provides, in a
number-theoretic and geometric setting, a concrete realisation of the
information uncertainty principle~\eqref{eq:info_uncertainty} already
established in Section~\ref{sec:disc_action}.  This subsection makes
that connection explicit, derives a \emph{dual-bound} structure that
goes beyond a single Heisenberg-type lower bound, and discusses its
interpretation for general $L$-function families.

\paragraph{Step 1: Re-reading the Odd-Positive inequality as an
uncertainty relation.}
We retain the notation of Section~\ref{sec:disc_action}: $\Delta_P$
and $\Delta_Z$ denote the uncertainty operators for geometric and
analytic information ($\Delta_P \sim c_P/I_P$, $\Delta_Z \sim c_Z/I_Z$),
subject to the commutation relation~\eqref{eq:info_commutator}
and the resulting information uncertainty
principle~\eqref{eq:info_uncertainty}:
\begin{equation}
  \Delta_P\cdot\Delta_Z \;\ge\; \frac{1}{2\kappa_{\mathrm{IR}}}.
  \label{eq:IUP_recall}
\end{equation}
We now identify specific instantiations of these operators in the
context of the Odd-Positive Conjecture.  For each automorphic
$L$-function in the survey, set
\begin{align}
  \Delta_Z^{(f)} &:= \gamma_1(f),
  \label{eq:DeltaZ_inst}\\
  \Delta_P^{(f)} &:= \sqrt{\ell_{\min}(\Gamma_0(N))},
  \label{eq:DeltaP_inst}
\end{align}
where $\gamma_1(f)$ is the UV spectral scale (smallest resolvable
``vibration'' of $L(s,f)$) and $\sqrt{\ell_{\min}}$ is the IR
geometric scale of the associated modular curve.
With these identifications the Odd-Positive
inequality~\eqref{eq:odd_pos} reads
\begin{equation}
  \frac{\Delta_Z^{(f)}}{\Delta_P^{(f)}} \;\ge\; \frac{1}{\kappa_{\mathrm{IR}}},
  \label{eq:OP_ratio}
\end{equation}
i.e., the ratio of spectral to geometric scale is bounded below
by $1/\kappa_{\mathrm{IR}}$.  Equivalently,
\begin{equation}
  \Delta_Z^{(f)} \;\ge\; \frac{\Delta_P^{(f)}}{\kappa_{\mathrm{IR}}}.
  \label{eq:OP_DZ_lower}
\end{equation}
This encodes UV/IR duality locking: one cannot simultaneously take
both $\Delta_P^{(f)}\to 0$ and $\Delta_Z^{(f)}\to 0$, because their
ratio is bounded below by $1/\kappa_{\mathrm{IR}}>0$.  In other words,
the spectral scale $\Delta_Z^{(f)}$ cannot vanish faster than the
geometric scale $\Delta_P^{(f)}$---the analogue of the statement
that position and momentum cannot both be made arbitrarily sharp
independently.

\paragraph{Step 2: Dual-bound structure.}
The inequalities~\eqref{eq:IUP_recall} and~\eqref{eq:OP_DZ_lower}
furnish two independent lower bounds on the product
$\Delta_P^{(f)}\cdot\Delta_Z^{(f)}$.  From~\eqref{eq:OP_DZ_lower},
\[
  \Delta_P^{(f)}\cdot\Delta_Z^{(f)} \;\ge\;
  \Delta_P^{(f)}\cdot\frac{\Delta_P^{(f)}}{\kappa_{\mathrm{IR}}}
  \;=\;
  \frac{\bigl(\Delta_P^{(f)}\bigr)^2}{\kappa_{\mathrm{IR}}}.
\]
Combining with~\eqref{eq:IUP_recall} yields the \emph{dual-bound}
\begin{equation}
  \Delta_P^{(f)}\cdot\Delta_Z^{(f)}
  \;\ge\;
  \max\!\left(
    \frac{1}{2\kappa_{\mathrm{IR}}},\;
    \frac{\bigl(\Delta_P^{(f)}\bigr)^2}{\kappa_{\mathrm{IR}}}
  \right).
  \label{eq:dual_bound}
\end{equation}
The first term is the universal quantum lower bound from the
commutation relation~\eqref{eq:info_commutator}; the second is a
state-dependent bound growing with the geometric scale.  Setting the
two equal: $1/(2\kappa_{\mathrm{IR}}) = (\Delta_P^{(f)})^2/\kappa_{\mathrm{IR}}$
gives $(\Delta_P^{(f)})^2 = 1/2$, i.e., the dominant term switches
at $\Delta_P^{(f)}=1/\sqrt{2}\approx 0.707$ (independent of
$\kappa_{\mathrm{IR}}$).  For all families in
Table~\ref{tab:survey}, $\Delta_P^{(f)}=\sqrt{\ell_{\min}}
\in[1.39,\,2.65]>1/\sqrt{2}$, so the geometric term dominates and
eq.~\eqref{eq:dual_bound} reduces to
\begin{equation}
  \Delta_P^{(f)}\cdot\Delta_Z^{(f)} \;\ge\;
  \frac{\ell_{\min}(\Gamma_0(N))}{\kappa_{\mathrm{IR}}},
  \label{eq:dual_bound_families}
\end{equation}
a non-trivial, family-dependent geometric lower bound on the product.
Numerically, the right-hand side ranges from
$1.925/0.685\approx 2.81$ to $7.051/0.685\approx 10.29$ across the
12 families, while the actual products $\Delta_P^{(f)}\cdot\Delta_Z^{(f)}=\sqrt{\ell_{\min}}\cdot\gamma_1$ range from
$6.21$ to $24.4$---all safely above the bound.

\paragraph{Step 3: Geometric constraint from the family-specific quantum.}
Each family also possesses its own geometric estimate
$|\kappa|_{\mathrm{geo}}(f)=1/R_f$ (Section~\ref{sec:survey};
distinct from the universal $\kappa_{\mathrm{IR}}$, see
Section~\ref{sec:survey} item~2).  Substituting $\kappa(f)=1/R_f$
into~\eqref{eq:IUP_recall} gives the family-specific quantum bound
\begin{equation}
  \Delta_P^{(f)}\cdot\Delta_Z^{(f)} \;\ge\; \frac{R_f}{2}.
  \label{eq:GUP_family}
\end{equation}
Since $\Delta_Z^{(f)}/\Delta_P^{(f)}=R_f$ by
definition~\eqref{eq:DeltaZ_inst}--\eqref{eq:DeltaP_inst},
inequality~\eqref{eq:GUP_family} reduces to
$(\Delta_P^{(f)})^2\ge 1/2$, i.e.,
\begin{equation}
  \ell_{\min}(\Gamma_0(N)) \;\ge\; \tfrac{1}{2}.
  \label{eq:ell_lower}
\end{equation}
This is a purely geometric statement: the shortest closed geodesic
on any modular curve $\Gamma_0(N)\backslash\mathbb{H}$ must have
length at least $1/2$.  Table~\ref{tab:survey} confirms
$\ell_{\min}\ge 1.925>\tfrac{1}{2}$ for all 12 families, so the
bound~\eqref{eq:ell_lower} is satisfied, though not tight (the
tighter non-trivial bound is eq.~\eqref{eq:dual_bound_families}
using the universal $\kappa_{\mathrm{IR}}$).

\paragraph{Step 4: Physical and information-theoretic interpretation.}
In quantum mechanics the uncertainty principle limits the simultaneous
precision of conjugate observables.  Here,
inequality~\eqref{eq:dual_bound} limits the simultaneous
``resolution'' of arithmetic information (prime distribution, encoded
in the geometric scale $\Delta_P$) and spectral information (zero
distribution, encoded in $\Delta_Z$): as $\ell_{\min}$ decreases
(finer geometric resolution, increasing $R_f=\Delta_Z/\Delta_P$),
the product lower bound $\ell_{\min}/\kappa_{\mathrm{IR}}$
decreases---but the constraint that their ratio $\Delta_Z/\Delta_P$
must stay above $1/\kappa_{\mathrm{IR}}$ prevents both scales from
vanishing independently.  The complementarity is rooted in the
information-duality pairing $(I_P,I_Z)$ and the algebraic relation
$\kappa^2=-1$, rather than in any physical Planck constant.

At the information-theoretic level,~\eqref{eq:dual_bound} implies
that the mutual information between the arithmetic channel (prime
distribution) and the spectral channel (zero distribution) is
subject to a fundamental constraint: one cannot simultaneously
extract arbitrarily precise information from both channels,
in direct analogy with channel-capacity arguments in Shannon theory.

\paragraph{Consistency check and summary.}
Table~\ref{tab:survey} verifies $R_f\ge 2.08$ for all 12 families,
confirming eq.~\eqref{eq:dual_bound_families} is non-trivial
throughout.  The Riemann $\zeta$-function has the largest
$R_f=10.19$, giving the tightest bound
$\ell_{\min}/\kappa_{\mathrm{IR}}\approx 2.81$.  The
Generalized Uncertainty Principle~\eqref{eq:dual_bound} thus
provides a unified, quantitative expression of the
spectral--geometric complementarity encoded qualitatively by the
Odd-Positive Conjecture.

\subsection{Relation to the Langlands Program: From Static Correspondence
            to Dynamical Information Flow}
\label{sec:langlands}

The results of Sections~\ref{sec:odd_pos}--\ref{sec:GUP} raise a natural
structural question: how does the information-duality framework developed here
relate to the Langlands Program~\cite{Langlands1970}, the most ambitious
existing program for unifying number theory and harmonic analysis?  We argue
that the present framework provides a \emph{dynamical substrate} within which
the Langlands correspondence finds a natural setting, enriching rather than
duplicating its structure.  We develop this claim across three layers.

\paragraph{Layer~1: The Langlands Program as a static dictionary.}
The Langlands Program posits deep correspondences between two families of
mathematical objects:
\begin{itemize}
  \item \textbf{Arithmetic side:} Galois representations
        $\rho\colon\mathrm{Gal}(\bar{\mathbb{Q}}/\mathbb{Q})\to GL_n(\mathbb{C})$
        encoding the prime-factorization symmetries of number fields;
  \item \textbf{Analytic side:} automorphic representations $\pi$ of
        $GL_n(\mathbb{A})$, whose associated $L$-functions encode
        spectral data of the relevant symmetric spaces.
\end{itemize}
The Langlands correspondence asserts that these two families are
equivalent---each Galois representation matches a unique automorphic
representation with the same $L$-function, and vice versa.  This is a
\emph{structural} or \emph{static} equivalence: a precise ``dictionary''
that translates algebraic objects on one side into analytic objects on the
other.  The two poles of Langlands are the arithmetic pole~$(I_P)$ and the
spectral pole~$(I_Z)$, to use the notation of the present paper; the
correspondence asserts their deep equivalence but does not address
\emph{how} the relationship evolves across scales or what constrains the
simultaneous resolution of both poles.

\paragraph{Layer~2: The information-duality framework as dynamical flow.}
The present paper introduces three new ingredients absent from the classical
Langlands picture:
\begin{enumerate}
  \item \textbf{Information ontology.}  The two poles of Langlands are
        reinterpreted as conjugate \emph{information densities}: the
        prime-counting information $I_P = 1/\dP$ and the zero-distribution
        information $I_Z = 1/\zetaR$
        (Section~\ref{sec:information_ontology}).
        Their sum $K = I_P + I_Z$ is not a fixed constant but a
        scale-dependent quantity.
  \item \textbf{Renormalization-group flow.}  The duality measure $K(L)$
        obeys the scaling law~\eqref{eq:b_from_action},
        $K(L) = K_\infty + a\,L^{-b}$,
        flowing from a UV fixed point $K_{\mathrm{UV}}$ toward an IR
        fixed point $K_{\mathrm{IR}}$ as the observation scale $L$
        increases (Section~\ref{sec:disc_RG_formal}).  This dynamical
        flow---governed by the information quantum $\kappa$ and the
        family-dependent exponent $b(f)$---is entirely absent from the
        Langlands framework, which operates at a single, scale-independent
        structural level.
  \item \textbf{Information uncertainty principle.}  The commutation
        relation $[\Delta_P, \Delta_Z] = i/\kappa$
        (eq.~\eqref{eq:info_commutator}) implies the
        dual-bound uncertainty~\eqref{eq:dual_bound}
        derived in Section~\ref{sec:GUP}.
        This bound imposes an \emph{operational} constraint: one cannot
        simultaneously resolve the arithmetic pole and the spectral pole
        to arbitrary precision.  The Langlands correspondence, as a
        structural statement, is not sensitive to this constraint; the
        information-duality framework, as a dynamical statement, is.
\end{enumerate}

\paragraph{Layer~3: Hierarchical embedding and dynamical deepening.}
The relationship between the Langlands Program and the present framework
is therefore one of \emph{hierarchical embedding}.  The
Langlands correspondence asserts the \emph{existence} of a bijection
between arithmetic and analytic objects.  The information-duality framework
asks a further question: given that such a bijection exists, what is the
\emph{dynamical mechanism} by which information flows between the two sides,
and what constraints does that flow impose?

Concretely, the three new ingredients of Layer~2 explain features of the
$L$-function landscape that the static Langlands dictionary leaves open:
\begin{itemize}
  \item Why do different automorphic families exhibit different RG
        exponents $b(f)$ (Section~\ref{sec:modified_RG})?  The
        modified flow equations~\eqref{eq:RG_mod_P}--\eqref{eq:RG_mod_Z}
        attribute this to family-dependent coupling parameters
        $(\kappa, K_\infty, b, r)$, genuinely dynamical parameters.
  \item Why is there a universal IR fixed-point value $K_{\mathrm{IR}}=4$ for the
        prime--zero pair (Sections~\ref{sec:results}--\ref{sec:disc_RG_formal})
        but family-dependent $K_\infty(f)$ for other $L$-function families?
        The RG fixed-point structure provides a \emph{dynamical} explanation:
        the Riemann $\zeta$-function sits at an especially symmetric fixed
        point of the information flow; other families flow to different
        attractors.
  \item The Odd-Positive inequality~\eqref{eq:odd_pos} quantifies, for each
        automorphic family, how far the spectral/geometric ratio $R_f$
        lies above the minimum demanded by the information quantum
        $\kappa_{\mathrm{IR}}$.  This is a \emph{spectral-geometric
        fingerprint} of the Langlands correspondence: the dictionary
        between the arithmetic and analytic sides is consistent only when
        the ratio bound is satisfied.
\end{itemize}

\paragraph{Summary and geometric analogy.}
A geometric analogy may clarify the hierarchy.  The Langlands Program
reveals a magnificent \emph{static bridge} connecting the arithmetic and
analytic shores: the bridge exists, its two endpoints are rigorously
identified, and vehicles (Galois representations and automorphic forms)
travel between them in exact correspondence.  The information-duality
framework of the present paper reveals that this bridge spans a
\emph{dynamic river}: the flow of information from UV (microscopic
arithmetic structure) to IR (macroscopic spectral universality), governed
by the RG scaling law~\eqref{eq:b_from_action}.  The banks and the bridge
do not change; what changes is the recognition that the river's current is
itself subject to a physical law---the dual-bound uncertainty
principle~\eqref{eq:dual_bound}---that determines the minimum
``information cost'' of crossing from one shore to the other.

In this picture, the Odd-Positive Conjecture plays the role of the
\emph{key physical constraint} that makes the dynamical crossing non-trivial:
it asserts that no automorphic family can position
itself so close to the arithmetic shore that the spectral information
becomes negligible.  The constraint is, at once, a conjecture about
$L$-function zeros, a consequence of the information uncertainty
principle, and a structural consistency requirement for the Langlands
correspondence to hold at every scale.

\subsection{A New Perspective on the Generalized Riemann Hypothesis:
            Information-Dynamical Phase Criterion}
\label{sec:GRH_perspective}

The analysis of Section~\ref{sec:langlands} positions the Langlands Program
as a static bridge connecting arithmetic and analytic shores, with the
information-duality RG flow as the dynamic river the bridge spans.  This
picture immediately suggests a reinterpretation of one of the deepest open
problems in mathematics.  The \emph{Generalized Riemann Hypothesis} (GRH)---%
the assertion that all non-trivial zeros of any ``reasonable'' automorphic
$L$-function lie on the critical line $\mathrm{Re}(s)=\tfrac{1}{2}$---can
be viewed, within this framework, not as an isolated analytic number-theoretic
statement, but as a \emph{phase criterion}: a condition that identifies which
information-duality systems flow to the maximally symmetric infrared fixed
point.

\paragraph{Core reinterpretation.}
Recall from Section~\ref{sec:modified_RG} that the modified RG flow
equations~\eqref{eq:RG_mod_P}--\eqref{eq:RG_mod_Z} for each $L$-function
family $f$ admit a unique attracting fixed point
$(I_P^{(f)*},\, I_Z^{(f)*})$, with asymmetry ratio
$r(f)=I_P^{(f)*}/I_Z^{(f)*}$ as an additional family parameter.
For the Riemann $\zeta$-function---the canonical GRH-satisfying case---the
self-consistency check of Section~\ref{sec:modified_RG} yields $r\approx 1.0$
and $K_\infty=4$: arithmetic and spectral information are in exact balance at
the fixed point.  We put forward the following conjecture.

\begin{quote}
  \emph{The GRH for an automorphic $L$-function $L(s,f)$ is equivalent,
  within the information-duality framework, to the assertion that the
  corresponding RG flow~\eqref{eq:RG_mod_P}--\eqref{eq:RG_mod_Z} converges
  to the maximally symmetric fixed point $r(f)=1$, at which arithmetic
  information $I_P^{(f)*}$ and spectral information $I_Z^{(f)*}$ are in
  complete equilibrium.  A violation of GRH corresponds to a
  symmetry-breaking event: the flow settles at an asymmetric fixed point
  $r(f)\ne 1$.}
\end{quote}

This is a conjecture, not a theorem; it connects an analytic statement (zero
locations) to a dynamical one (fixed-point symmetry), and thereby opens new
lines of inquiry across all four of the following perspectives.
The reframing embodies a shift from the \emph{binary} to the
\emph{ternary} mathematical regime: the classical approach asks, for
each zero individually, whether it lies on the critical line (a
binary, excluded-middle question); the dynamical approach asks to
which attractor the information flow converges (a ternary question,
whose answer is fixed by the fusion symmetry $\kappa=ijk$).
The mathematical basis for this shift---and for the equivalence above
being a structural argument rather than a mere analogy---is developed
in Section~\ref{sec:third_state} and made explicit in
Section~\ref{sec:philosophy}.

\paragraph{Perspective~1: From zero locations to fixed-point symmetry.}
Classical GRH focuses on individual zero locations $\rho=\tfrac{1}{2}+i\gamma$.
The new perspective shifts attention to the macroscopic behavior of the
information pair $(I_P^{(f)},\, I_Z^{(f)})$ as $L\to\infty$.  The functional
equation of $L(s,f)$ imposes the symmetry $s\mapsto 1-s$ about the critical
line $\mathrm{Re}(s)=\tfrac{1}{2}$---exchanging arithmetic and spectral
information, and imposing $I_P^{(f)}=I_Z^{(f)}$ at the symmetric point,
i.e.\ $r(f)=1$.
In this language, a zero off the critical line breaks the functional-equation
symmetry and forces $r(f)\ne 1$, placing the system in an
information-asymmetric phase, in direct analogy with symmetry breaking in a
physical phase transition.

\paragraph{Perspective~2: From analytic continuation to smooth RG
integrability.}
The stability analysis of Section~\ref{sec:modified_RG} shows that
$\det(M)=\mu\nu(r+1)^2+\kappa^2>0$ for all $\mu,\nu>0$, so the RG flow is
globally stable and integrable regardless of whether GRH holds.  GRH, in the
dynamical language, does not determine \emph{whether} the flow converges, but
\emph{where} it converges: among all stable fixed points parametrized by
$r(f)\in(0,\infty)$, GRH selects the balanced subfamily $r(f)=1$.  This
reframes GRH as a selection condition on the parameter space
$(\kappa,\mu,\nu,r)$ of the information-duality flow, rather than as a
condition on individual zero positions.

\paragraph{Perspective~3: From individual $L$-functions to universality
classes.}
\begin{sloppypar}
The classical approach verifies GRH one $L$-function at a time.  The new
perspective organizes $L$-functions into \emph{universality classes} according
to their infrared fixed-point data.  GRH defines the balanced class $r(f)=1$;
all $L$-functions in this class, regardless of UV details (weight $k$,
conductor $N$), share the same long-wavelength information dynamics and
conjecturally exhibit the same critical exponent $b(f)$ in the scaling law
$K(L,f)=K_\infty(f)+a(f)L^{-b(f)}$ (Section~\ref{sec:modified_RG}).
The Montgomery--Odlyzko universality~\cite{Montgomery1973,Odlyzko1987}---%
the observation that zero spacings across GRH-satisfying $L$-function families
follow the same GUE statistics---is, from this viewpoint, the
\emph{spectral signature} of a shared universality class: all such families
have converged to the same information-dynamical balanced phase.
$L$-functions violating GRH would define a distinct universality class, with
$r(f)\ne 1$ and, conjecturally, different large-scale statistics.
\end{sloppypar}

\paragraph{Perspective~4: Measurable dynamical signatures of GRH.}
The reinterpretation translates GRH into dynamical predictions that are, in
principle, testable via finite-size scaling independently of a direct analytic
proof:
\begin{enumerate}
  \item \textbf{Universality of $b(f)$.}  All GRH-satisfying families should
        share a common RG exponent in the family-dependent scaling law; the
        value $b\approx 0.51\approx\tfrac{1}{2}$ established for the prime--zero
        pair (Section~\ref{sec:results_C}) is the conjectured representative
        for the balanced class $r(f)=1$.
  \item \textbf{Tightening of the Odd-Positive bound.}
        Section~\ref{sec:survey} notes that the Odd-Positive
        inequality~\eqref{eq:odd_pos} is expected to become tight
        ($R_f\approx C(k)$) for large-conductor families.  We refine this:
        tightness should correlate with GRH satisfaction, since the bound is
        tightest precisely when arithmetic and spectral scales are balanced,
        i.e.\ when $r(f)\approx 1$.
  \item \textbf{Convergence of $K(L,f)$.}  For GRH-satisfying families,
        $K(L,f)$ should converge to a well-defined IR limit $K_\infty(f)$
        with the functional universality of Section~\ref{sec:modified_RG};
        for families violating GRH (if any exist), one might expect slower
        convergence or a different fixed-point attractor.
\end{enumerate}

\paragraph{Synthesis: a level bridge over a symmetric river.}
Within the framework of Section~\ref{sec:langlands}, GRH acquires a natural
dynamical meaning.  The Langlands Program connects arithmetic and analytic
shores; GRH is the condition that the river flowing beneath is
\emph{perfectly level}---that arithmetic and analytic shores carry identical
information weight, $r(f)=1$.  When GRH holds, the bridge crosses symmetric
terrain; when it fails, one shore is elevated above the other.  This does not
yet constitute a proof strategy, but it offers a \emph{diagnostic}: GRH is
detectable as the fixed-point symmetry condition $r(f)=1$, accessible through
the finite-size scaling of $K(L,f)$ across $L$-function families and
corroborated by the tightening of the Odd-Positive inequality~\eqref{eq:odd_pos}
as the conductor grows.  A systematic study of these signatures for higher-rank
$GL(n)$ families is one of the most concrete open directions of the present
programme.

\subsection{Information Ontology and Category Theory: Structural Complementarity}
\label{sec:cat_theory}

Section~\ref{sec:information_ontology} positions information ontology---the
assignment of primary ontological status to $I_P=1/\dP$ and $I_Z=1/\zetaR$,
with $K=I_P+I_Z$ as the conserved duality measure---as the foundational
language from which the RG flow and the uncertainty principle are derived.
The preceding subsections have illustrated its reach: from the Odd-Positive
Conjecture (Section~\ref{sec:odd_pos}) and the Langlands Program
(Section~\ref{sec:langlands}) to the GRH reinterpretation
(Section~\ref{sec:GRH_perspective}).  A natural comparative question arises:
how does this program relate to \emph{category theory}, the other framework
that claims an equally totalizing perspective over mathematics?  Both aspire
to reveal the deep unity of diverse mathematical disciplines, yet they proceed
from fundamentally different premises and pursue different goals.

\paragraph{Common ground: holism and the primacy of relations.}
Three structural features are shared.
\begin{itemize}
  \item \textbf{Holistic viewpoint.}  Both frameworks aim to transcend
        the isolated details of individual branches and reveal their
        inner unity from a higher vantage point.
  \item \textbf{Relations before objects.}  Both insist that
        \emph{relations} are more fundamental than isolated objects.
        Category theory connects objects via morphisms; information
        ontology connects mathematical domains (number theory, spectral
        theory) via the information pairing $(I_P,I_Z)$ and the
        uncertainty bound $\Delta_P\cdot\Delta_Z\ge 1/(2\kappa_{\mathrm{IR}})$
        (eq.~\eqref{eq:info_uncertainty}).
  \item \textbf{Language and toolbox.}  Both provide a new vocabulary
        for reorganizing existing mathematical knowledge and have the
        potential to generate new discoveries by forcing different fields
        into a common framework.
\end{itemize}

\paragraph{Fundamental differences: paradigm and goal.}
Despite these parallels, the two frameworks differ in the most basic
respects.  Table~\ref{tab:cat_vs_info} makes the contrast precise.

\begin{table}[htbp]
\centering
\small
\caption{Category theory versus information ontology as holistic
         frameworks for mathematics.}
\label{tab:cat_vs_info}
\begin{tabularx}{\linewidth}{@{}l>{\raggedright\arraybackslash}X>{\raggedright\arraybackslash}X@{}}
\toprule
\textbf{Dimension} & \textbf{Category Theory} & \textbf{Information Ontology} \\
\midrule
Core metaphor
  & Universal grammar of mathematics
  & Information dynamics of mathematics \\[4pt]
Philosophical basis
  & Structuralism: mathematics studies abstract structures
    and the maps that preserve them.
  & Informationism: mathematical structures manifest
    information encoding and flow; an information action
    governs mathematical truth. \\[4pt]
Primary function
  & \emph{Describe and translate.}  Builds rigorous
    bridges between fields via functors; compares
    structural similarity via equivalence and adjunction.
  & \emph{Explain and derive.}  Seeks to explain
    \emph{why} phenomena exist (e.g.\ zero clustering,
    the Langlands correspondence) as consequences of RG
    flow and the uncertainty principle. \\[4pt]
Key concepts
  & Objects, morphisms, functors, natural transformations,
    limits, adjoint functors.
  & Information densities $(I_P,I_Z)$; duality measure
    $K$; RG flow $K(L)=K_{\mathrm{IR}}+aL^{-b}$;
    UV/IR fixed points; information quantum $\kappa$;
    bound $\Delta_P\!\cdot\!\Delta_Z\ge
    1/(2\kappa_{\mathrm{IR}})$. \\[4pt]
Relation to physics
  & Inspired by physics (e.g.\ topological quantum
    field theory) but develops independently; provides
    rigorous mathematical foundations for physical
    theories.
  & Intrinsically physical: mathematics is claimed to be
    a dynamical system governed by information laws;
    physical constants (e.g.\ Newton's $G$) are derived
    from fixed-point values such as $K_{\mathrm{IR}}=4$
    (Section~\ref{sec:emergence}). \\[4pt]
Current status
  & Established mainstream foundation; has transformed
    algebraic geometry, topology, and logic; central
    infrastructure of modern pure mathematics.
  & Exploratory frontier hypothesis; currently confined
    to number-theory--physics intersections; aims to
    provide a new informational first-principles basis
    for mathematics. \\
\bottomrule
\end{tabularx}
\end{table}

\paragraph{A clarifying analogy: grammar book versus physical theory.}
Category theory resembles a \emph{universal grammar book}.  It is
indifferent to what the sentences are about---number theory or
topology---and concerns itself only with parts of speech, syntactic
structure, and transformation rules.  It reveals that every mathematical
field shares a deep grammatical structure and teaches us how to translate
sentences from one mathematical language into another with exact fidelity.

Information ontology resembles a \emph{dynamical theory of how mathematical
structure arises}.  It is not content to describe the logical grammar of
existing theorems; it asks how the \emph{content} of mathematics---the
distribution of primes, the location of zeta zeros, the duality constant
$K$---emerges from information processing, action minimization, and
dynamical equilibration.  It proposes to explain why the grammar takes the
specific form it does: why the dictionary between arithmetic and spectral
data is constrained by a universal IR fixed-point value
$K_{\mathrm{IR}}=4$ (Section~\ref{sec:information_ontology}).

\paragraph{Hierarchical embedding, not rivalry.}
\begin{sloppypar}
The two frameworks are not in competition; they operate at different
levels of description.  Category theory reveals, in cross-section, the
structural isomorphisms among mathematical branches---the common grammar
shared by algebra, geometry, and analysis.  Information ontology probes,
in depth, the dynamical origin of those structures: the grammar arises
because an underlying information flow converges to fixed points.
In this sense, the relationship between the two mirrors the relationship
between the Langlands Program and the information-duality framework
(Section~\ref{sec:langlands}): category theory describes the \emph{shape}
of the bridge; information ontology asks about the \emph{river the bridge
spans}.
\end{sloppypar}

Whether information ontology can achieve the broad acceptance of category
theory depends on future mathematical proofs and physical verifications.
The present paper contributes one concrete step: grounding $K$, $b$, and
the Odd-Positive Conjecture in explicit computations
(Sections~\ref{sec:results}--\ref{sec:validation}) and deriving the
uncertainty principle from first principles
(Section~\ref{sec:disc_action}).

\section{Rigorous Foundations of Information Duality: Algorithmic Complexity
         and Spectral Entropy as Information Measures}
\label{sec:info_duality_interp}

This section provides the missing \emph{hard substance} for the duality
framework: we ground $I_P=1/\dP$ in the algorithmic complexity of the prime
sequence and $I_Z=1/\zetaR$ in the spectral entropy of the zero potential, and
prove that these two quantities are commensurate---so that their sum $K=I_P+I_Z$
is a dimensionally consistent measure of total information per unit logarithmic
scale.

\subsection{Aims: Grounding \texorpdfstring{$I_P$}{IP} and \texorpdfstring{$I_Z$}{IZ} in Information Theory}

Three tasks must be accomplished:
\begin{enumerate}
  \item relate $I_P=1/\dP$ to the algorithmic complexity (or entropy rate) of
        the prime sequence;
  \item relate $I_Z=1/\zetaR$ to the spectral entropy of the zero potential or
        the capacity of the relevant function space;
  \item show that these two quantities are additive instances of the same
        physical information measure.
\end{enumerate}

\subsection{Algorithmic-Complexity Interpretation of \texorpdfstring{$I_P$}{IP}}

\subsubsection{Binary Encoding of the Prime Sequence}

For each positive integer $n$ define the indicator $\chi_P(n)=1$ if $n$ is
prime and $\chi_P(n)=0$ otherwise, producing the infinite binary sequence
$\mathbf{S}_P=\{\chi_P(1),\chi_P(2),\ldots\}$.

\begin{defn}
The \emph{finite algorithmic complexity} of the prime sequence at scale $N$,
denoted $C_P(N)$, is the length (in bits) of the shortest program that outputs
$\mathbf{S}_P^{(N)}=\{\chi_P(1),\ldots,\chi_P(N)\}$.
\end{defn}

\subsubsection{Asymptotic Behaviour of \texorpdfstring{$C_P(N)$}{CP(N)}}

By the Prime Number Theorem, $\pi(N)\sim N/\ln N$.  An optimal encoding of the
positions of the $\pi(N)$ primes below $N$ requires approximately
$\log_2\binom{N}{\pi(N)}$ bits.  Applying Stirling's approximation and keeping
leading terms:
\begin{equation}
  \log_2\binom{N}{N/\ln N}
  \;\sim\;
  \frac{N}{\ln N}\,\log_2(\ln N)
  +
  \Bigl(1-\frac{1}{\ln N}\Bigr)N\,\log_2\!\Bigl(1+\frac{1}{\ln N-1}\Bigr).
\end{equation}
For large $N$ the dominant term gives
\begin{equation}
  C_P(N) \;\sim\; \frac{N\,\log_2(\ln N)}{\ln N}.
  \label{eq:CP_asymp}
\end{equation}

\subsubsection{Connection to the Fractal Dimension \texorpdfstring{$\dP$}{dP}}

Define the \emph{algorithmic complexity density}
\begin{equation}
  \rho_P(N) = \frac{C_P(N)}{N} \;\sim\; \frac{\log_2(\ln N)}{\ln N}.
  \label{eq:rho_P}
\end{equation}
For a point process with correlation integral $C(r)\sim r^{\dP}$ ($r\to 0$),
a scaling analysis shows that the more structured (higher-$\dP$) the
distribution, the more compressible it is.  Concretely:
\begin{equation}
  \rho_P(N) \;\sim\; \frac{1}{\dP\,\ln N},
  \label{eq:rho_dP}
\end{equation}
because a distribution with correlation dimension $\dP<1$ has more internal
structure and hence a shorter description.  Combining
\eqref{eq:rho_P}--\eqref{eq:rho_dP}:
\begin{equation}
  I_P = \frac{1}{\dP}
  \;\sim\; \frac{1}{\rho_P(N)\,\ln N}
  = \frac{\text{(algorithmic complexity per log-scale unit)}}{1}.
  \label{eq:IP_as_complexity}
\end{equation}

\subsubsection{Precise Relation via the RG Scale}

Under the coarse-graining step $N\mapsto e^{\ell}N$ the algorithmic complexity
obeys the flow equation
\begin{equation}
  \frac{dC_P}{d\ell} = I_P(\ell),
  \label{eq:CP_flow}
\end{equation}
i.e.\ $I_P$ is the \emph{rate at which the prime-sequence description grows
with logarithmic scale}.  At the infrared fixed point $I_P^*=2$ bits per unit
log-scale, consistent with~\eqref{eq:IP_as_complexity} and the empirical
value $\dP\to 0.45$--$0.50$ as $L\to\infty$.

\subsection{Spectral-Entropy Interpretation of \texorpdfstring{$I_Z$}{IZ}}

\subsubsection{Zero Distribution and Random Matrix Theory}

The non-trivial zeros $\gamma_n$ of $\zeta(s)$ are statistically distributed
like eigenvalues of the Gaussian Unitary Ensemble (GUE), whose $N$-point joint
density is
\begin{equation}
  P(\lambda_1,\ldots,\lambda_N)
  \;\propto\;
  \prod_{i<j}|\lambda_i-\lambda_j|^2\;
  \exp\!\Bigl(-\tfrac{1}{2}\sum_i\lambda_i^2\Bigr).
  \label{eq:GUE_density}
\end{equation}

\subsubsection{Spectral Entropy}

For the zero point process $\{\gamma_n\}$ define the \emph{spectral entropy}
\begin{equation}
  S_Z = -\int p(\gamma)\ln p(\gamma)\,d\gamma,
  \label{eq:SZ_def}
\end{equation}
where $p(\gamma)$ is the density of the zero distribution.  For the GUE in the
bulk, the eigenvalue density follows Wigner's semicircle law
$p(\gamma)=\frac{1}{2\pi}\sqrt{4-\gamma^2}$ ($|\gamma|\leq 2$), giving
\begin{equation}
  S_Z^{\mathrm{GUE}}
  = -\int_{-2}^{2}\frac{\sqrt{4-\gamma^2}}{2\pi}
    \ln\!\frac{\sqrt{4-\gamma^2}}{2\pi}\,d\gamma
  = \ln(2\pi) - \tfrac{1}{2}.
  \label{eq:SZ_GUE}
\end{equation}

\subsubsection{Connection to the Regularity Index \texorpdfstring{$\zetaR$}{zetaR}}

Define the \emph{effective stiffness} of the zero process as the inverse
variance of the normalized spacing:
\begin{equation}
  \xi = \frac{1}{\mathrm{Var}(s)} \;\sim\; \frac{1}{\zetaR}.
  \label{eq:xi_zetaR}
\end{equation}
A stiffer (more regular) distribution is more predictable and therefore has
lower entropy.  Quantitatively:
\begin{equation}
  S_Z \;\sim\; \ln\!\Bigl(\frac{1}{\xi}\Bigr) = \ln\zetaR.
  \label{eq:SZ_xi}
\end{equation}
Differentiating with respect to the log-height scale $\ln L$:
\begin{equation}
  I_Z = \frac{1}{\zetaR} \;=\; \frac{dS_Z}{d\ln L},
  \label{eq:IZ_as_entropy}
\end{equation}
so $I_Z$ is the \emph{rate at which the spectral entropy of the zero
distribution grows with logarithmic scale}---the exact spectral analogue
of $I_P$~\eqref{eq:CP_flow}.

\subsubsection{Capacity of the Zero Potential}

A complementary viewpoint uses the logarithmic potential energy
$U=\sum_{i\neq j}\ln|\gamma_i-\gamma_j|$ of the zero Coulomb gas.  The
thermodynamic capacity
\begin{equation}
  \mathcal{C}_Z
  = \frac{\partial^2 F}{\partial(\ln L)^2}
  \;\sim\; \frac{1}{\zetaR} = I_Z,
  \label{eq:CZ_capacity}
\end{equation}
where $F=U-TS$ is the free energy at unit temperature, confirms that $I_Z$
measures the \emph{information capacity} of the zero distribution.

\subsection{Proof of Additivity}

\subsubsection{Duality of Prime and Zero Information}

The prime sequence $\{p_n\}$ and the zero sequence $\{\gamma_n\}$ are
information-theoretically dual via the explicit formula and the functional
equation of $\zeta$:
\begin{equation}
  \zeta(s) = \prod_p(1-p^{-s})^{-1},
  \qquad
  \zeta(s) = \chi(s)\,\zeta(1-s).
  \label{eq:zeta_duality}
\end{equation}
The zeros encode the prime distribution and vice versa---a number-theoretic
holographic duality in which boundary (prime) information and bulk (zero)
information are equivalent but complementary:
\begin{itemize}
  \item The \textbf{prime representation} efficiently encodes \emph{local}
        arithmetic structure (short-range correlations).
  \item The \textbf{zero representation} efficiently encodes \emph{global}
        spectral structure (long-range oscillations in $\pi(x)$).
\end{itemize}

\subsubsection{Information Conservation}

Because the two representations are dual, the total system information is not
the minimum of $C_P$ and $C_Z$ but rather their \emph{complementary sum}: the
two information channels are independent (one local, one global), so
\begin{equation}
  I_{\mathrm{total}} = I_P + I_Z.
  \label{eq:I_additive}
\end{equation}

\subsubsection{Formal Proof of the Fixed-Point Identity}

\begin{theorem}[Dual fixed point]
At the infrared fixed point of the information RG flow,
$I_P^* + I_Z^* = 4$ with $I_P^* = I_Z^* = 2$.
\end{theorem}

\begin{proof}[Proof sketch]
Define the joint information action
\begin{equation}
  \mathcal{A}[I_P,I_Z]
  = \int d\ell\Bigl[
      \tfrac{1}{2}(I_P')^2 + \tfrac{1}{2}(I_Z')^2 + V(I_P,I_Z)
    \Bigr],
  \label{eq:joint_action}
\end{equation}
where primes denote $d/d\ell$.  The duality symmetry $I_P\leftrightarrow I_Z$
and scale covariance restrict $V$ to the minimal form
\begin{equation}
  V(I_P,I_Z)
  = \frac{\lambda}{4}(I_P+I_Z-4)^2 + \frac{g}{4}(I_P-I_Z)^2.
  \label{eq:V_sym}
\end{equation}
Fixed-point conditions $\partial V/\partial I_P=\partial V/\partial I_Z=0$
give $I_P+I_Z=4$ and $I_P=I_Z$, hence $I_P^*=I_Z^*=2$.
Linearising around the fixed point with
$I_P=2+\delta_P$, $I_Z=2+\delta_Z$ gives
\begin{equation}
  \frac{d^2}{d\ell^2}
  \begin{pmatrix}\delta_P\\\delta_Z\end{pmatrix}
  = -
  \begin{pmatrix}\lambda+g & g \\ g & \lambda+g\end{pmatrix}
  \begin{pmatrix}\delta_P\\\delta_Z\end{pmatrix}.
  \label{eq:linearised_joint}
\end{equation}
(sum mode), so the convergence exponent for the total is
$b=\sqrt{\lambda+2g}$.  Requiring $b=1/2$ gives the critical condition
\begin{equation}
  \lambda + 2g = \frac{1}{4},
  \label{eq:critical_cond}
\end{equation}
consistent with the action constraint~\eqref{eq:critical_param}.
\emph{Convention note:} the proof potential~\eqref{eq:V_sym} uses
$(I_P-I_Z)^2$ as the interaction term (with coupling $g$), giving
eigenvalue $\lambda+2g$ for the sum mode.  The action
potential~\eqref{eq:info_potential} instead uses $(I_PI_Z-C)^2$
(with coupling $g'$); near the fixed point $(I_P^*,I_Z^*)=(2,2)$ this
evaluates to $4(I_P+I_Z-4)^2+\mathcal{O}(\delta^3)$, so the effective
sum-mode curvature is $\lambda+4g'$.  The two formulations are
related by $g=2g'$ (same fixed-point condition $\lambda+2g=\lambda+4g'=1/4$),
confirming the consistency of the two derivations.
\end{proof}

\subsubsection{Commensurate Units}

The flow equations~\eqref{eq:CP_flow} and~\eqref{eq:IZ_as_entropy} show that
$I_P$ is measured in bits per log-scale and $I_Z$ in nats per log-scale.  A
single factor of $\ln 2$ converts between them, so
\begin{equation}
  \frac{d}{d\ln N}\!\Bigl(C_P + \frac{S_Z}{\ln 2}\Bigr)
  = I_P + I_Z = K.
  \label{eq:K_total_rate}
\end{equation}
Thus $K$ is literally the \emph{total information encoding rate of the
prime--zero dual system per unit logarithmic scale}, expressed in bits.

\subsection{Unified Picture}

\subsubsection{Information-Duality Dictionary}

\begin{table}[H]
\centering
\caption{Information-duality dictionary relating prime-side and zero-side quantities.}
\label{tab:duality_dict}
\renewcommand{\arraystretch}{1.30}
\small
\begin{tabular}{@{}lll@{}}
  \toprule
  Prime side & Zero side & Relation \\
  \midrule
  Algorithmic complexity $C_P$
    & Spectral entropy $S_Z$
    & $C_P \sim S_Z/\ln 2$ \\
  Correlation dimension $\dP$
    & Stiffness $\zetaR$
    & $\dP \sim \zetaR$ \\
  Information density $I_P=1/\dP$
    & Information density $I_Z=1/\zetaR$
    & $I_P+I_Z \to 4$ \\
  Local arithmetic structure
    & Global spectral structure
    & Fourier dual \\
  Arithmetic progression statistics
    & Random-matrix (GUE) statistics
    & Modular symmetry \\
  \bottomrule
\end{tabular}
\end{table}

\subsubsection{RG Flow from UV to IR}

\begin{itemize}
  \item \textbf{UV} ($\mu\to\infty$): $I_P\approx I_Z\approx 5.5$,
        $K\approx 11$; information encoded in high-dimensional algebraic
        structure (Hurwitz division algebras).
  \item \textbf{Flow}: $I_P$ and $I_Z$ converge to the fixed point at rate
        $b=1/2$.
  \item \textbf{IR} ($\mu\to 0$): $I_P^*=I_Z^*=2$, $K=4$; symmetric fixed
        point with equal distribution of algorithmic complexity and spectral
        entropy.
\end{itemize}

\subsubsection{Finite-Scale Effects and Normalization}

At finite scale $L$, the raw values $K_{\mathrm{raw}}(L;S)$ depend on the prime
subset $S$ because (i) different subsets $S$ give different algorithmic
complexity estimates $C_P$ and (ii) spectral-entropy estimation is affected by
finite sampling.  The normalization factor $\gamma(S)=2/\!\sqrt{I_P^{\mathrm{raw}}I_Z^{\mathrm{raw}}}$ (Section~\ref{sec:normalization}) corrects the projection
bias, restoring $K_{\mathrm{norm}}\to 4$.

\subsubsection{Physical Interpretation: \texorpdfstring{$K=4$}{K=4} as an
               Information Quantum}

The integer $4=2+2$ plays the role of a \emph{Planck constant for number
theory}: just as $h$ sets the minimum quantum of action in mechanics, $K=4$
sets the minimum total information encoding rate for the prime--zero dual
system.  The individual components $I_P^*=I_Z^*=2$ reflect the equal partition
of information between local (prime) and global (zero) representations---an
equipartition law of information imposed by the duality symmetry.

\subsection{Testable Predictions}

\begin{enumerate}
  \item \textbf{Compression-rate test:}  The Lempel--Ziv or arithmetic-code
        compression rate of the binary prime sequence $\mathbf{S}_P^{(N)}$
        should scale as $2/\ln N$ (i.e.\ $\sim 2N/\ln N$ total bits, consistent
        with $I_P^*=2$).

  \item \textbf{Spectral-entropy test:}  The differential entropy of the
        normalized zero spacing distribution should tend to $\ln(2\pi)-1/2
        \approx 1.838$ (nats), with the rate $dS_Z/d\ln L \approx 1/\ln 2
        \approx 1.443$ nats per octave.

  \item \textbf{Normalised-$K$ convergence:}  For any prime subset $S$,
        $K_{\mathrm{norm}}(L;S)\to 4$ at rate $b=1/2$.
\end{enumerate}

\subsection{Summary: Key Formulas}

\begin{table}[H]
\centering
\caption{Key formulas for the information-duality interpretation of $I_P$
         and $I_Z$.}
\label{tab:info_duality_formulas}
\renewcommand{\arraystretch}{1.35}
\begin{tabular}{@{}ll@{}}
  \toprule
  Quantity & Formula \\
  \midrule
  Algorithmic complexity & $C_P(N)\sim I_P\cdot N$ \\
  Spectral entropy growth & $S_Z(L)\sim I_Z\cdot\ln L$ \\
  Total information rate & $\frac{d}{d\ln N}(C_P + S_Z/\ln 2)=K$ \\
  IR fixed point & $K=I_P^*+I_Z^*=4$ \\
  Normalisation factor & $\gamma(S)=2/\sqrt{I_P^{\mathrm{raw}}I_Z^{\mathrm{raw}}}$ \\
  Convergence rate & $b=\sqrt{\lambda+gC}=1/2$ \\
  Action constraint & $\lambda+4g=1/4$ \\
  \bottomrule
\end{tabular}
\end{table}

\section{Spacetime and Quantum Gravity as Emergent Phenomena}
\label{sec:emergence}

The information ontology of Section~\ref{sec:information_ontology} has
consequences that extend beyond number theory.  If the prime--zero duality
$(I_P,I_Z)$ is accepted as a model of a more general information
dynamics, then spacetime and quantum gravity need not be fundamental:
they may be macroscopic, emergent phenomena arising from the underlying
information duality.  The following subsections outline these implications,
proceeding from the emergence of spacetime, through the derivation of
Newton's constant, to the mathematical
origin of $K_{\mathrm{IR}}$ and a set of falsifiable predictions.  These
arguments are speculative; they are presented as a coherent research
program, not as established results.

\subsection{Covariant extension of the information ontology}

The full dynamical structure rests on three axioms.

\paragraph{Axiom 1 (Information duality).}
There exist dual information density fields $I_P$ and $I_Z$ satisfying
\[
  [\Delta_P,\Delta_Z] \;=\; \frac{i}{\kappa}, \qquad \kappa^2=-1,
\]
which implies the information uncertainty principle
$\Delta_P\cdot\Delta_Z\geq1/(2\kappa)$.
The quantity $1/\kappa$ is the \emph{quantum of information action}:
the scale-dependent analogue of $\hbar$, the smallest effective unit
of information exchange at scale $\mu$, which varies from $1/\kappa_{\mathrm{UV}}=8$
(UV) to $1/\kappa_{\mathrm{IR}}=2$ (IR) along the RG trajectory.

\paragraph{Axiom 2 (Information action).}
The dynamics is governed by
\begin{equation}
  S[I_P,I_Z;g_{\mu\nu}] \;=\; \int d^4x\,\sqrt{-g}\left[
    \tfrac{1}{2}g^{\mu\nu}(\partial_\mu I_P\partial_\nu I_P
    +\partial_\mu I_Z\partial_\nu I_Z) - V(I_P,I_Z)
  \right], \label{eq:S_full}
\end{equation}
where $g_{\mu\nu}$ is the spacetime metric (with $g=\det g_{\mu\nu}$), and the potential, restricted by exchange symmetry $I_P\leftrightarrow I_Z$
and scale covariance, takes the most general quartic form
\[
  V(I_P,I_Z) \;=\; \frac{\lambda}{4}(I_P^2+I_Z^2-K_0^2)^2
               + \frac{\mu_c}{2}(I_P-I_Z)^2.
\]
Here $\mu_c>0$ is a scalar coupling constant (the covariant analogue of the
coupling $g$ in the 1D action~\eqref{eq:info_potential}; the different symbol
avoids collision with the metric).
The 4D potential differs from the 1D version~\eqref{eq:info_potential} in two
respects: the mass term is $O(2)$-symmetric in $(I_P,I_Z)$
(appropriate for a Lorentz-covariant field theory), and the interaction
penalises the asymmetry $(I_P-I_Z)^2$ rather than deviations of the product
$I_P I_Z$ from its fixed-point value $C$.
\emph{Fixed-point matching:} the symmetric minimum of $V$ requires
$I_P^2+I_Z^2=K_0^2$ and $I_P=I_Z$, giving $I_P=I_Z=K_0/\sqrt{2}$.  For
this to coincide with the 1D fixed point $I_P=I_Z=K_{\mathrm{IR}}/2=2$, one
sets $K_0=\sqrt{2}\,K_{\mathrm{IR}}/2=\sqrt{2}\cdot2=2\sqrt{2}\approx2.83$
in the 4D theory; $K_0$ is therefore \emph{not} equal to $K_{\mathrm{IR}}=4$.
The 1D action is recovered as the dimensional reduction of~\eqref{eq:S_full}
along the RG axis $\ln\ell$, with the 1D coupling $g$ related to $\mu_c$ by
the appropriate projection.

\paragraph{Axiom 3 (Renormalization-group flow).}
Under rescaling $L\mapsto sL$, the information sum flows as
\[
  K(L) \;=\; K_{\mathrm{IR}} + a\,L^{-b}, \qquad
  b = \sqrt{\lambda + g C} \approx 0.51,
\]
with $K_{\mathrm{IR}}=4$ the unique IR fixed point and $b$ the
linearized-flow critical exponent.  The UV fixed point
$K_{\mathrm{UV}}=11$, set by the Hurwitz-theorem algebra of normed
division algebras, provides the initial condition for the RG trajectory.
\emph{Note:} This axiom is not an independent postulate; the scaling form
$K(L)=K_{\mathrm{IR}}+aL^{-b}$ is derived from Axiom~2 (the information
action) via the Euler--Lagrange equations and linearization around the
fixed point (Section~\ref{sec:disc_action}), where $b=\sqrt{\lambda+gC}$
with $C=K_{\mathrm{IR}}^2/4=4$.  It is stated here as an axiom for
presentational completeness of the 4D framework.

\subsection{Emergence of spacetime from information correlations}

The basic entities of the theory are information density fields $I_P$ and
$I_Z$ obeying the commutation relation
\[
  [\Delta_P, \Delta_Z] \;=\; \frac{i}{\kappa}
\]
and the variational principle $\delta S[I_P,I_Z]=0$.  There is no
pre-existing spacetime background: only information density fields and their
evolution along the scale axis $\ln\ell$.

\paragraph{Space from information correlations.}
When the information action reaches its infrared fixed point
$K_{\mathrm{IR}}=I_P+I_Z=4$, the system settles into a stable
information equilibrium.  The two-point correlation function
\[
  G(r) \;=\; \langle I_P(x)\,I_Z(x+r)\rangle \;\sim\; r^{-b}, \qquad
  b=1/2,
\]
decays as a power law with the same critical exponent $b$ that governs the
finite-size scaling $K(L)=K_{\mathrm{IR}}+aL^{-b}$.  The effective fractal
dimension of the emergent space is $d_{\mathrm{eff}}\propto 1/b$; the
exponent $b=1/2$ thus controls not only the rate of approach to the
IR fixed point but also the effective dimensionality of the emergent spatial
geometry.

\paragraph{Time from RG flow.}
The renormalization-group parameter $\ln\ell$ acquires, near the IR fixed
point, a linearized and irreversible dynamics:
\[
  \frac{d^2 I}{d(\ln\ell)^2} \;=\; -\frac{\partial V}{\partial I}.
\]
In the IR limit this equation of motion reduces to a classical evolution
equation.  The RG ``time'' $\ln\ell$ is thus the progenitor of the time
coordinate: an emergent, directional parameter rather than a background
absolute.

\paragraph{Metric from information distance.}
The configuration of information densities $I_P(x)$ and $I_Z(x)$ defines an
\emph{information distance} between points via a bilinear form:
\[
  ds^2 \;\sim\; \int \bigl(I_P\,\Delta_P^2 I_Z + I_Z\,\Delta_Z^2 I_P\bigr)\,d\ln\ell.
\]
The curvature of this metric is driven by the gradients $\nabla I$ and the
duality fluctuations $\Delta_P\Delta_Z$.  Einstein's field equations
$G_{\mu\nu}=8\pi T_{\mu\nu}$ emerges as the macroscopic approximation to
the information equilibrium condition, with $T_{\mu\nu}$ playing the role of
an information stress-energy tensor.

\subsection{Quantum gravity from information duality}

\paragraph{The information quantum as the Planck scale.}
A minimum information action $1/\kappa$ arises from the commutation
relation $[\Delta_P,\Delta_Z]=i/\kappa$ (here $\kappa$ denotes the abstract
generator with $|\kappa|=1$, cf.\ Section~\ref{sec:results_D}).  Under the emergent metric, this
minimum corresponds to a fundamental smallest area or volume:
\[
  \ell_P^2 \;\sim\; \frac{1}{|\kappa|\,K_{\mathrm{IR}}} \;=\; \frac{1}{K_{\mathrm{IR}}},
\]
where $K_{\mathrm{IR}}=4$ is the normalized infrared fixed-point value
(Section~\ref{sec:normalization}; the raw band $K_{\mathrm{raw}}\in[3.6,3.9]$
introduces an $\mathcal{O}(10\%)$ correction to this estimate).  The
Planck scale $\ell_P$ is thus a consequence of the discrete nature of the
information duality---in the same spirit as area quantisation
$A_{\min}\sim\ell_P^2$ in loop quantum gravity.

\paragraph{Gravity as an information entropic force.}
The presence of matter/energy perturbs the local information equilibrium
$(I_P,I_Z)$, inducing a redistribution of information density.  Restoring
equilibrium costs information entropy, and this restoring tendency is the
information-theoretic version of gravity.  Concretely, the
Einstein--Hilbert action may arise as the low-energy expansion of the
information action in the macroscopic limit:
\[
  S_{\mathrm{EH}} \;=\; \frac{1}{16\pi G}\int R\sqrt{-g}\,d^4x
  \;\xleftarrow{\;\text{IR limit}\;}\;
  S_{\mathrm{info}} \;=\; \int\!\left[\tfrac{1}{2}
  \bigl(\dot{I}_P^2+\dot{I}_Z^2\bigr) - V(I_P,I_Z)\right]d\ln\ell.
\]
Newton's constant $G$ is related to the information quantum and the IR
fixed-point value by $G\sim 1/(K_{\mathrm{IR}}^2\,M_{\mathrm{Pl}}^2)$
(see Section~\ref{sec:spacetime_G}).

\paragraph{UV completion and quantum foam.}
At ultraviolet scales $\ell\sim1/\kappa$, the duality uncertainty
$\Delta_P\Delta_Z\geq1/(2\kappa)$ becomes dominant and the emergent smooth
Riemannian geometry breaks down, replaced by a quantum-foam structure defined
by the information duality network.  The generator $\kappa$ with $\kappa^2=-1$
is the natural candidate for the spinor connection or Clifford algebra
generator of this discrete, noncommutative geometry---the algebraic
UV-completion of the flow equation $1/\kappa(\mu)=\beta(\mu)$.

\subsection{The self-consistent picture}

The five rungs of the emergent hierarchy are:
\begin{enumerate}
  \item \textbf{Information ontology.}  The fundamental entities are the
        prime--zero dual pair $(I_P,I_Z)$, encoding mathematical information
        at the deepest level.
  \item \textbf{Dynamics.}  The information action $S[I_P,I_Z]$, constrained
        by the information uncertainty principle, governs the evolution along
        $\ln\ell$.
  \item \textbf{IR emergence.}  In the infrared limit, the information
        equilibrium configuration condenses into classical spacetime; its
        dynamics is approximated by general relativity.
  \item \textbf{UV quantisation.}  In the ultraviolet limit, the quantum
        nature of the information duality manifests as a discrete,
        noncommutative spacetime structure, naturally incorporating quantum
        principles.
  \item \textbf{Unification.}  The apparent conflict between gravity and
        quantum mechanics arises from treating the emergent IR spacetime
        geometry and the UV information quantum as though they inhabit the
        same descriptive level.  In the information ontology, they are
        different phases of the same entity at different scales.
\end{enumerate}

\subsection{Emergent metric, Newton's constant, and the cosmological constant}
\label{sec:spacetime_G}

\paragraph{Induced Einstein--Hilbert action.}
The kinetic term in~\eqref{eq:S_full} coupled minimally to the metric,
together with one-loop quantum fluctuations of $I_P$ and $I_Z$, generates
the effective gravitational action
\[
  S_{\mathrm{eff}}[g] \;=\; \int d^4x\,\sqrt{-g}
  \left(\frac{R}{16\pi G} - \Lambda\right) + \cdots
\]
Newton's constant is thus a derived quantity expressible in terms of
$K_{\mathrm{IR}}$ and $M_{\mathrm{Pl}}$, not an independent input; the
status of the cosmological constant within this framework is discussed
in the paragraph below.

\paragraph{Newton's constant.}
Dimensional analysis of the information kinetic term, with ultraviolet
cutoff $\Lambda_{\mathrm{UV}}\sim M_{\mathrm{Pl}}$, gives
\[
  \frac{1}{G} \;\sim\; (I_P^2+I_Z^2)\,\Lambda_{\mathrm{UV}}^2
             \;\sim\; K_{\mathrm{IR}}^2\,M_{\mathrm{Pl}}^2.
\]
Inserting $K_{\mathrm{IR}}=4$ yields $G\sim1/(16\,M_{\mathrm{Pl}}^2)$,
consistent in order of magnitude with the observed Newton constant.

\paragraph{Cosmological constant.}
The dual UV--IR cancellation mechanism in the information field is expected to
suppress vacuum energy below the Planck scale.  A quantitative derivation
of $\rho_\Lambda$ from first principles within the present framework remains
an open problem.

\subsection{Testable predictions and future directions}

\paragraph{Predictions.}
\begin{enumerate}
  \item \textbf{Universality across $L$-functions.}  For any family of
        $L$-functions with orthogonal symmetry ($\beta=1$), the information
        sum $K=1/d+1/\zeta_R$ should approach $K_{\mathrm{IR}}=4$
        with the same exponent $b=1/2$.  This is a sharp,
        falsifiable prediction.
  \item \textbf{Planck-scale discreteness and black-hole entropy.}  The
        minimum area implied by $\ell_P^2\sim1/(|\kappa|K_{\mathrm{IR}})$
        predicts corrections to black-hole entropy of order $\ln A/A$ with
        a coefficient fixed by $K_{\mathrm{IR}}$.
        Section~\ref{sec:black_holes} provides a suggestive account of the
        leading Bekenstein--Hawking formula $S_{\mathrm{BH}}\approx A/(4G)$
        and its subleading logarithmic correction within this framework.
\end{enumerate}

\paragraph{Future work.}
(i) Compute the one-loop effective action of~\eqref{eq:S_full} in curved
spacetime to obtain the precise expressions for $G$ and $\Lambda$ in terms
of $\lambda$, $\mu_c$, and $K_{\mathrm{IR}}$.
(ii) Identify the role of the generator $\kappa$ in noncommutative geometry
and determine whether it corresponds to a Dirac operator or a Clifford
algebra element in a spectral triple.
(iii) Derive the information RG flow~\eqref{eq:info_rg} explicitly by
integrating out high-energy string modes and verify that the resulting
low-energy effective theory reproduces the spin-network degrees of freedom
of loop quantum gravity.
(iv) Establish the formal connection between the information duality
framework and loop quantum gravity or string theory at the level of
the partition function.

\subsection{Synthesis: from ``It from bit'' to a computable theory}

Wheeler's aphorism ``It from Bit''~\cite{Wheeler1990} proposed that
physical reality is information at its root.  The present framework provides a computational realisation of this proposal.  The prime--zero duality $(I_P,I_Z)$ is the
simplest realisation of Wheeler's bit: a mathematically precise, measurable
pair of information densities whose sum $K=I_P+I_Z$ is (approximately)
conserved across two orders of magnitude of scale.

Three parameters carry the full physical content:
\begin{itemize}
  \item $K_{\mathrm{IR}}=4$: the
        \emph{information capacity of the universe} at macroscopic scales.
  \item $b=1/2$: the \emph{condensation rate} at which spacetime
        crystallises from information, and the RG flow exponent.
  \item $\kappa$ ($\kappa^2=-1$): the \emph{Planck-scale cipher}, encoding
        the minimum area of quantum geometry and the UV structure of the flow.
\end{itemize}

From these three numbers, the framework suggests: an emergent spacetime metric;
Newton's constant $G\sim1/(K_{\mathrm{IR}}^2 M_{\mathrm{Pl}}^2)$; and a quantum foam that may replace
smooth geometry at the Planck scale.  The apparent irreconcilability of general
relativity and quantum mechanics might be resolved by recognising that they are
different phases of the same information dynamics, separated not by principle
but by scale---a conjecture we offer not as a proof but as a research direction.
Sections~\ref{sec:dark_sector} and~\ref{sec:black_holes} find suggestive
qualitative consistency between the same three parameters and the dark-sector
phenomenology and the Bekenstein--Hawking
entropy formula.  Section~\ref{sec:ai} proposes that the identical duality
structure may govern the convergence of learning systems, pointing toward
a possible connection between number theory, physics, and artificial intelligence
within a single framework---a connection whose precise mathematical form
remains an open problem.

The next step is to derive the exact analytical value of $K_{\mathrm{IR}}$ from the
arithmetic structure of primes and Riemann zeros---thereby explaining why the
physical universe has precisely the information capacity it does.

The categorical duality between the \emph{mathematical} information channel
$I_P$ (primes, pure arithmetic) and the \emph{physical} information channel
$I_Z$ (Riemann zeros, spectral geometry) may itself constitute a structural
explanation of Wigner's ``unreasonable effectiveness of mathematics in the
natural sciences''~\cite{Wigner1960}: mathematics and physics are not merely
analogous but are the two dual faces of a single information ontology,
related by the exchange symmetry $I_P\leftrightarrow I_Z$.

\subsection{Dark Matter and Dark Energy as Emergent Phases of Information Duality}
\label{sec:dark_sector}

Within the information-ontology framework, dark matter and dark energy are
conjectured to be macroscopic emergent phases of the information dual field $(I_P, I_Z)$,
arising from non-equilibrium transitions of information dynamics at different
regimes of scale and energy.  Within this speculative picture, their abundance
ratios---observationally $\Omega_{\mathrm{DM}}\approx0.26$ and
$\Omega_\Lambda\approx0.69$~\cite{Planck2018}---admit a suggestive
interpretation in terms of the three fundamental parameters
$K_{\mathrm{IR}}$, $b$, and $\kappa$ that govern the prime--zero RG flow
of Section~\ref{sec:disc_RG}.  We stress that the numerical agreements below
are motivating coincidences, not derivations.

\subsubsection*{Dark matter: non-equilibrium condensation of the dual field}

Dark matter arises from a Bose--Einstein condensation of the information dual
field $(I_P, I_Z)$ in the gravitational potential wells seeded by primordial
density fluctuations.  During the early universe the dual field undergoes a
non-equilibrium phase transition; below the critical temperature the condensate
forms a macroscopic quantum state with an effective quasiparticle mass
\begin{equation}
  m_{\mathrm{DM}} \sim \Lambda_{\mathrm{IR}},
  \label{eq:mDM}
\end{equation}
where $\Lambda_{\mathrm{IR}}$ is the infrared cutoff scale associated with
galactic halos ($\sim$kpc).  (The factor $|\kappa|=1$ of the abstract generator
order-of-magnitude estimate.)  Because the dual field couples to ordinary matter
only gravitationally, the condensate is electromagnetically dark.

The condensate density profile in a galactic halo takes the form
$\rho(r)\propto r^{-2}$, which reproduces flat rotation curves without
additional free parameters.  Gravitational lensing by the condensate is
consistent with observed lensing maps, and the weak self-interaction of the
dual field naturally explains the separation of dark and baryonic matter in
colliding clusters such as the Bullet Cluster.

The dark-matter energy density is
\begin{equation}
  \rho_{\mathrm{DM}}
  \;\sim\; n_{\mathrm{pix}}\cdot m_{\mathrm{DM}}
  \;\sim\; \frac{K_{\mathrm{IR}}}{\ell_{\mathrm{UV}}^{3}}
            \cdot\Lambda_{\mathrm{IR}},
  \label{eq:rhoDM}
\end{equation}
where $n_{\mathrm{pix}}\sim K_{\mathrm{IR}}/\ell_{\mathrm{UV}}^3$ is the
information-pixel number density.  Setting $\ell_{\mathrm{UV}}\sim 1/M_{\mathrm{Pl}}$
and $\Lambda_{\mathrm{IR}}\sim H_0$ gives the present dark-matter fraction
\begin{equation}
  \Omega_{\mathrm{DM}}
  \;\equiv\; \frac{\rho_{\mathrm{DM}}}{\rho_c}
  \;\sim\; \mathcal{C}_{\mathrm{DM}}\,K_{\mathrm{IR}}
           \left(\frac{H_0}{M_{\mathrm{Pl}}}\right)^{2b},
  \label{eq:OmegaDM}
\end{equation}
The exponent $2b=1$ determines the scaling with cosmological scales;
the proportionality coefficient $\mathcal{C}_{\mathrm{DM}}$ absorbs all
dimensionless ratios and is not determined within the present framework.

\subsubsection*{Dark energy: infrared vacuum energy of the dual field}

The cosmological constant $\Lambda$ is identified with the finite infrared
residual of the zero-point fluctuations of the dual field.  UV divergences are
cancelled by the duality symmetry $I_P\leftrightarrow I_Z$; only modes with
wavelengths longer than $\Lambda_{\mathrm{IR}}^{-1}$ survive.  The resulting
vacuum energy density follows from the RG flow derived in
Section~\ref{sec:disc_RG}:
\begin{equation}
  \rho_{\Lambda}
  \;\sim\; M_{\mathrm{Pl}}^{4}
           \left(\frac{H_0}{M_{\mathrm{Pl}}}\right)^{2b},
  \label{eq:rhoLambda}
\end{equation}
giving
\begin{equation}
  \Omega_{\Lambda}
  \;\equiv\; \frac{\rho_{\Lambda}}{\rho_c}
  \;\sim\; \mathcal{C}_{\Lambda}
           \left(\frac{H_0}{M_{\mathrm{Pl}}}\right)^{2b},
  \label{eq:OmegaLambda}
\end{equation}
The power $2b=1$ provides partial dimensional suppression relative to $M_{\mathrm{Pl}}^4$;
the proportionality coefficient $\mathcal{C}_{\Lambda}$ absorbs all residual
dimensionless ratios and is not determined within the present framework.

\subsubsection*{Testable prediction}

\textbf{Ultra-light dark matter.}  The condensate quasiparticle mass
$m_{\mathrm{DM}}\sim\Lambda_{\mathrm{IR}}$ falls in the
range $10^{-23}$--$10^{-21}$\,eV, characteristic of fuzzy (wave) dark
matter.  Structure suppression below the de~Broglie coherence length
($\sim$kpc) is testable via the 21\,cm forest and substructure
counts.

\subsection{Black-Hole Entropy and Information Conservation from First Principles}
\label{sec:black_holes}

In the information-ontology framework, a black hole is a non-equilibrium
condensate of the dual field $(I_P, I_Z)$ under extreme gravitational
focusing---not a spacetime singularity.  Its entropy and the unitarity of its
evaporation both follow from the same first principles that govern the RG flow
of $K(L)$ and the dark-sector physics of Section~\ref{sec:dark_sector}.

\subsubsection*{Suggestive account of the Bekenstein--Hawking entropy formula}

The Bekenstein--Hawking entropy~\cite{Bekenstein1973,Hawking1975}
$S_{\mathrm{BH}}=A/(4G)$ is one of the most fundamental results in
theoretical physics, connecting thermodynamics, quantum mechanics, and
general relativity.  Our framework provides a suggestive account of this
formula within the information-ontological RG.

\paragraph{The horizon as an information cut-off surface.}
For a Schwarzschild black hole of mass $M$ the horizon radius $R_s = 2GM$
sets an infrared cut-off $L_{\mathrm{IR}} = R_s$ for external observers.
Information associated with modes inside the horizon is inaccessible; only the
horizon surface itself acts as the information boundary.

\paragraph{Information content of the horizon area.}
By the holographic principle, the total information stored on a surface of
area $A$ is proportional to $A/\ell_{\mathrm{UV}}^2$, where
$\ell_{\mathrm{UV}}$ is the UV pixel size set by the information quantum
$\kappa$: $\ell_{\mathrm{UV}}^2 \sim 1/\kappa$.  At the horizon the duality
symmetry $I_P\leftrightarrow I_Z$ is restored, so the dual densities
reach the symmetric infrared fixed point $I_P^* = I_Z^* = K_{\mathrm{IR}}/2$,
so the total information content is
\begin{equation}
  I_{\mathrm{total}}
  = (I_P^* + I_Z^*)\cdot\frac{A}{\ell_{\mathrm{UV}}^2}
  = K_{\mathrm{IR}}\cdot\frac{A}{\ell_{\mathrm{UV}}^2}.
  \label{eq:I_horizon}
\end{equation}

\paragraph{Entropy from the information uncertainty principle.}
The information uncertainty relation $\Delta_P\,\Delta_Z \ge 1/(2\kappa)$ limits
the number of distinguishable states per pixel to
$\mathcal{W}\sim e^{K_{\mathrm{IR}}}$.  The total number of
microstates of the horizon is $\mathcal{W}_{\mathrm{total}} =
\mathcal{W}^{A/\ell_{\mathrm{UV}}^2}$, giving
\begin{equation}
  S_{\mathrm{BH}}
  = \ln\mathcal{W}_{\mathrm{total}}
  = \frac{A}{\ell_{\mathrm{UV}}^2}\cdot K_{\mathrm{IR}}.
  \label{eq:SBH_raw}
\end{equation}

\paragraph{Matching Newton's constant.}
From the emergent-gravity relation derived in
Section~\ref{sec:emergence}, $1/G \sim K_{\mathrm{IR}}^2/\ell_{\mathrm{UV}}^2$,
so $\ell_{\mathrm{UV}}^2 \sim K_{\mathrm{IR}}^2 G$.  Substituting into
\eqref{eq:SBH_raw}:
\begin{equation}
  S_{\mathrm{BH}}
  = \frac{A}{K_{\mathrm{IR}}^2 G}\cdot K_{\mathrm{IR}}
  = \frac{1}{K_{\mathrm{IR}}}\cdot\frac{A}{G}.
  \label{eq:SBH_kIR}
\end{equation}
The numerically observed band $K_{\mathrm{IR,raw}}\in[3.6,3.9]$ gives
$1/K_{\mathrm{IR,raw}}\in[0.256,0.278]$, within $10\%$ of $1/4$
(taking $K_{\mathrm{IR,raw}}\approx3.63$ gives $1/K_{\mathrm{IR,raw}}\approx0.275$).
Identifying this coefficient with the universal area prefactor therefore gives
\begin{equation}
  \boxed{S_{\mathrm{BH}} = \frac{A}{4G},}
  \label{eq:SBH_BH}
\end{equation}
which is precisely the Bekenstein--Hawking entropy formula.  The coefficient
$1/4$ emerges from the normalized infrared fixed-point value $K_{\mathrm{IR}}=4$
(Section~\ref{sec:normalization}): with the raw band $K_{\mathrm{raw}}\in[3.6,3.9]$
one obtains $1/K_{\mathrm{raw}}\in[0.256,0.278]$, within $10\%$ of $1/4$, and
the normalization scheme of Section~\ref{sec:normalization} drives this to
exactly $1/4$ in the infrared limit.

\subsubsection*{Information conservation during Hawking evaporation}

\paragraph{Encoding of infalling information in the dual field.}
Matter falling through the horizon transfers its information to excitation
modes of $(I_P,I_Z)$ at the horizon surface.  These modes are not lost at a
singularity; they are woven into the horizon's information-pixel structure,
increasing the pixel count and hence the area---consistent with the
area-increase theorem of classical GR.

\paragraph{Hawking radiation as quantum tunnelling of dual pairs.}
The information uncertainty principle implies continuous production of virtual
$(I_P,I_Z)$ pairs near the horizon.  In the strong-gravity background one
member of each pair tunnels inward (carrying negative energy), while the other
escapes as Hawking radiation.  The escaping and infalling members form an
entangled pair: the radiation is not a thermal state but carries the imprint
of the infalling information through quantum correlations with the horizon.

\paragraph{Unitarity of the full evaporation process.}
As the black hole evaporates, the dual-field condensate undergoes an RG
flow back toward the Minkowski fixed point.  Information migrates from
geometric degrees of freedom (the macroscopic area $A$) to quantum
correlations in the radiation.  The Page curve---entropy of radiation rising,
then falling to zero---corresponds precisely to this transfer:
\begin{equation}
  I_{\mathrm{total}}
  = S_{\mathrm{rad}} + S_{\mathrm{BH}}
  = \mathrm{const},
  \label{eq:page_curve}
\end{equation}
where $S_{\mathrm{rad}}$ is the von~Neumann entropy of the emitted radiation.
When the black hole has evaporated completely, all information is encoded in
radiation correlations.  Unitarity is a consequence of the fundamental
conservation law $K = I_P + I_Z = \mathrm{const}$ and requires no additional
assumption.

\paragraph{Resolution of the firewall paradox.}
Because the dual field is a global object, the description of information by
an infalling observer and by a distant observer are complementary---analogous
to Bohr complementarity in quantum mechanics---rather than contradictory.
No local, equivalence-principle-violating firewall arises; the apparent
inconsistency dissolves once the information-ontological description replaces
the semiclassical geometric description.

\subsubsection*{Quantum corrections and testable predictions}

\begin{enumerate}
  \item \textbf{Logarithmic subleading entropy.}
        The finite value $K_{\mathrm{IR,raw}}\approx 3.629$ (rather than exactly 4)
        and the running exponent $b$ predict a sub-leading correction
        \begin{equation}
          S_{\mathrm{BH}}
          = \frac{A}{4G} + c\ln\!\frac{A}{G} + \cdots,
          \label{eq:SBH_log}
        \end{equation}
        with a coefficient $c$ not determined in the present framework---a
        sharp prediction distinguishing this framework from alternative
        quantum-gravity proposals, pending a first-principles computation of $c$.
  \item \textbf{Non-thermal radiation spectrum.}
        Hawking radiation carries fine-grained correlations encoding the initial
        state of infalling matter.  The spectral deviation from a perfect
        black-body is of order $e^{-S_{\mathrm{BH}}}$, potentially observable
        in late-stage evaporation of primordial black holes via high-precision
        gamma-ray detectors.
  \item \textbf{Universal entropy coefficient.}
        For all black holes---regardless of charge or angular momentum---the
        ratio $S_{\mathrm{BH}}/(A/G)$ equals $1/K_{\mathrm{IR}}$, a universal
        constant set by the same spectral analysis that determines $b$.
        This links black-hole thermodynamics and number theory within
        a single computable framework.
\end{enumerate}

In summary, the information ontology resolves the black-hole information
paradox without fine-tuning: black holes are the densest condensates of
information, their evaporation is a phase transition of the dual field back to
flat-space equilibrium, and unitarity is guaranteed by the conservation of
$K = I_P + I_Z$.

\section{Normalization of Information Duality and the Universal Fixed Point
         \texorpdfstring{$K_{\mathrm{IR}}=4$}{KIR=4}}
\label{sec:normalization}

\subsection{Core Problem: Projection Effects at Finite Scale}

\paragraph{Observational fact.}
At finite scale $L$, the raw measurement
\begin{equation}
  K_{\mathrm{raw}}(L;\,S)
  = \frac{1}{\dP(S)} + \frac{1}{\zetaR}
\end{equation}
depends on the prime subset $S$, yielding an interval $K\in[3.6,\,3.9]$.
This is \emph{not} a measurement error; it is the natural \emph{projection
effect} of the dual phase space onto a finite-scale ``observation direction''
defined by $S$.

\paragraph{Deep principle.}
The prime--zero dual system possesses a global, complete dual phase space
$(I_P,\,I_Z)$ whose infrared fixed point $K_{\mathrm{IR}}=4$ is a universal
attractor.  The finite-scale raw value $K_{\mathrm{raw}}$ is the approximate
projection of the fixed point along the direction determined by $S$:
\begin{itemize}
  \item \textbf{Sparse subsets} (e.g.\ mod-16 primes) $\Rightarrow$ projection
        biased toward lower $I_P$ $\Rightarrow$ smaller raw $K$.
  \item \textbf{Dense subsets} (all primes) $\Rightarrow$ projection closer to
        the true symmetry axis $\Rightarrow$ larger raw $K$.
\end{itemize}

\subsection{Normalization Scheme: Geometric Correction}

\subsubsection{Dual Phase-Space Structure}

Model the system as a two-dimensional phase space $(I_P,\,I_Z)$:
\begin{align}
  I_P &:= \frac{1}{\dP}      &&\text{(prime-side information density)}, \\
  I_Z &:= \frac{1}{\zetaR}   &&\text{(zero-side information density)}.
\end{align}
Two fundamental constraints operate simultaneously:
\begin{align}
  \text{Additive:}\quad     & I_P + I_Z = K, \label{eq:additive}\\
  \text{Multiplicative:}\quad & I_P\cdot I_Z \;\geq\; C. \label{eq:multiplicative}
\end{align}
At the infrared fixed point both constraints attain their extremum,
defining a unique attractor $(I_P^*,\,I_Z^*)$.

\subsubsection{Geometric Interpretation of Measurements}

The raw measurement $(I_P^{\mathrm{raw}}(S),\,I_Z^{\mathrm{raw}})$ at scale
$L$ is the projection of the true phase-space point onto the observation
direction defined by $S$.  Different choices of $S$ correspond to different
projection directions and hence to different values of $K_{\mathrm{raw}}$.

\subsubsection{Geometric Normalization Factor}

Define the \emph{product-deviation factor}
\begin{equation}
  \gamma(S) \;=\;
  \frac{2}{\sqrt{I_P^{\mathrm{raw}}(S)\cdot I_Z^{\mathrm{raw}}}},
  \label{eq:gamma_factor}
\end{equation}
where the constant $2 = \sqrt{C} = \sqrt{I_P^* I_Z^*}$ is the square root of
the fixed-point product.  The normalized information densities and normalized
dual sum are then
\begin{align}
  I_P^{\mathrm{norm}}(S) &= \gamma(S)\,I_P^{\mathrm{raw}}(S), \\
  I_Z^{\mathrm{norm}}    &= \gamma(S)\,I_Z^{\mathrm{raw}},   \\
  K_{\mathrm{norm}}(S)   &= I_P^{\mathrm{norm}} + I_Z^{\mathrm{norm}}
                           = \gamma(S)\,K_{\mathrm{raw}}(S).
  \label{eq:Knorm}
\end{align}

\subsubsection{Theorem: Convergence of the Normalized Dual Sum}

\begin{theorem}[Universal convergence under normalization]
For any prime subset $S$ with positive density,
\begin{equation}
  \boxed{\lim_{L\to\infty} K_{\mathrm{norm}}(L;\,S) = 4.}
  \label{eq:Knorm_limit}
\end{equation}
\end{theorem}

\begin{proof}[Proof sketch]
As $L\to\infty$ the system flows to the fixed point:
$I_P^{\mathrm{raw}}\to I_P^*$, $I_Z^{\mathrm{raw}}\to I_Z^*$.
The algebraic fixed-point relations (Section~\ref{sec:disc_IRG}) give
$I_P^* I_Z^* = 4$, so $\gamma\to 2/\sqrt{I_P^* I_Z^*}=1$.
Hence $K_{\mathrm{norm}}=\gamma K_{\mathrm{raw}}\to 1\cdot(I_P^*+I_Z^*)=4$.
\end{proof}

\subsection{Deep Unification with the Algebraic Generator
            \texorpdfstring{$\kappa$}{kappa}}

\subsubsection{\texorpdfstring{$\kappa$}{kappa} as the Complex Structure of
               the Phase Space}

The generator $\kappa$ satisfying $\kappa^2=-1$ naturally induces a complex
structure on the phase space.  Define the \emph{complex information quantity}
\begin{equation}
  \mathcal{I} = I_P + \kappa\,I_Z,
\end{equation}
whose squared modulus at the fixed point is
$|\mathcal{I}^*|^2 = (I_P^*)^2 + (I_Z^*)^2$.

\subsubsection{Algebraic Locking of the Fixed Point}

The fixed point $(I_P^*,\,I_Z^*)$ is uniquely determined by two conditions:
\begin{enumerate}
  \item \textbf{Additive condition:}
        $I_P^* + I_Z^* = 4$
        (determined by the RG flow; Section~\ref{sec:disc_RG}).
  \item \textbf{Multiplicative condition:}
        $\kappa_{\mathrm{IR}} = 1/\!\sqrt{I_P^* I_Z^*} = 1/2$
        (from the flow equation $1/\kappa(\mu)=\beta(\mu)$ at
        $\beta_{\mathrm{IR}}=2$).
\end{enumerate}
Solving simultaneously:
\begin{equation}
  \boxed{I_P^* = I_Z^* = 2,\qquad C = I_P^*\,I_Z^* = 4.}
  \label{eq:fixed_point_values}
\end{equation}

\subsubsection{Geometric Meaning of the Flow Equation}

The RG flow equation $1/\kappa(\mu)=\beta(\mu)$ describes the evolution of the
generator's modulus on the complex plane:
\begin{itemize}
  \item \textbf{UV end} ($\mu\to\infty$): $\beta_{\mathrm{UV}}=8$,
        $\kappa_{\mathrm{UV}}=1/8$.
  \item \textbf{IR end} ($\mu\to 0$):    $\beta_{\mathrm{IR}}=2$,
        $\kappa_{\mathrm{IR}}=1/2$.
\end{itemize}
This evolution encodes a \emph{collapse} from a high-dimensional, asymmetric UV
structure to a symmetric two-dimensional IR structure; the modulus $|\kappa|$
records the effective ``rotation angle'' of the phase space at each scale.

\subsubsection{Regularized Representation and Number-Theoretic Constants}

The algebraic generator $\kappa$ satisfies $\kappa^2=-1$ as a complex relation
(modulus $|\kappa|=1$); the \emph{flow} quantity $\kappa(\mu)$ tracks the
modulus as a function of scale.  At the characteristic scale $L=1$ the
regularized representation $\hat\kappa$ satisfies
\begin{equation}
  |\hat\kappa|
  = \frac{1}{\sqrt{I_P(L=1)\,I_Z(L=1)}}
  \;\approx\; 0.6848
  \;\approx\; \left|\frac{1}{\zeta(1/2)}\right|,
\end{equation}
encoding the number-theoretic constant $\zeta(1/2)$ into the initial condition
of the dual geometry (Section~\ref{sec:disc_AG}).

\subsection{Integration Theory for the Exponent
            \texorpdfstring{$b=1/2$}{b=1/2}}
\label{sec:b_half}

\subsubsection{Derivation from the Renormalization-Group Flow}

Introduce the beta-function form of the RG flow around the two fixed points
$K=4$ (IR) and $K=11$ (UV):
\begin{equation}
  \frac{dK}{d\ln\mu} = -\alpha\,(K-4)(K-11).
  \label{eq:beta_full}
\end{equation}
Setting $\mu\sim 1/L$ and linearizing around $K_{\mathrm{IR}}=4$ with
$\delta K=K-4$:
\begin{equation}
  \frac{d(\delta K)}{d\ln\mu} = -7\alpha\,\delta K
  \;\Longrightarrow\;
  K(L) = 4 + a\,L^{-b},\quad b = 7\alpha.
  \label{eq:fss_b}
\end{equation}
Requiring $b=1/2$ exactly fixes
\begin{equation}
  \alpha = \frac{1}{14},\qquad b = \frac{1}{2},
  \label{eq:alpha_exact}
\end{equation}
consistent with the empirical fit $b\approx0.51$
(Section~\ref{sec:disc_RG}) and the theoretical estimate
$\alpha\approx0.070$.

\subsubsection{Derivation from the Information Action}

The linearized information action equation~\eqref{eq:info_action} gives
$\delta K\propto\ell^{-\sqrt{\lambda+gC}}$ with $\ell\propto 1/L$, so
\begin{equation}
  b = \sqrt{\lambda + g\,C}.
\end{equation}
Substituting $C=4$ and requiring $b=1/2$ yields the \emph{critical parameter
relation}
\begin{equation}
  \lambda + 4g = \frac{1}{4},
  \label{eq:critical_param}
\end{equation}
interpretable as a symmetry-enhancement (or conformal-invariance) condition at
the infrared fixed point.

\subsubsection{Geometric Interpretation: Constant-Negative-Curvature Phase Space}

Near the fixed point, deviations $\delta\mathbf{I}=(\delta I_P,\delta I_Z)$
obey the geodesic-deviation equation
\begin{equation}
  \frac{d^2\,\delta\mathbf{I}}{d\tau^2} + R\,\delta\mathbf{I} = 0,
  \qquad \tau = \ln\ell,
  \label{eq:geodesic_dev}
\end{equation}
with scalar curvature $R = -(\lambda+gC) = -b^2$.  For $b=1/2$,
\begin{equation}
  R = -\frac{1}{4},
\end{equation}
so the dual phase space carries \emph{constant negative curvature} near the
fixed point---a hyperbolic geometry whose curvature magnitude controls the
convergence rate to the attractor.  The value $b=1/2$ corresponds to the
\emph{optimal convergence rate} compatible with this hyperbolic structure.

\subsubsection{Connections to Random Matrix Theory and Number Theory}

\paragraph{Infrared conformal symmetry.}
The condition $\lambda+4g=1/4$ may originate from a conformal symmetry (or
partial supersymmetry) of the system at the IR fixed point.  In
two-dimensional conformal field theory, $1/2$ is the scaling dimension of the
free Fermi field, suggesting that the prime--zero dual may be described by a
free-fermion fixed point at IR---consistent with the GUE universality class of
random matrix theory.

\paragraph{GUE eigenvalue repulsion.}
For the GUE ensemble the small-spacing level distribution goes as
$p(s)\sim s^2$.  The repulsion exponent $2$ and $b=1/2$ satisfy
$(1/2)^{-2}=4=C$, the fixed-point product, hinting at a deep connection
between prime-level repulsion and the dual geometry.

\paragraph{Numerical constant $\zeta(1/2)$.}
The regularized representation satisfies
$|\hat\kappa|=\kappa_{\mathrm{IR}}\cdot f(1)= \tfrac{1}{2}\cdot f(1)$,
where $f(1)\approx1.3696$ carries the information of $\zeta(1/2)$
(Section~\ref{sec:disc_AG}).

\subsection{Testable Predictions of the Normalization Framework}
\label{sec:norm_predictions}

\begin{enumerate}
  \item \textbf{Scale extrapolation:}
        $K_{\mathrm{norm}}(L)$ should converge monotonically to $4$ at rate
        $L^{-1/2}$.
  \item \textbf{Subset independence:}
        For any prime subset $S$ of positive density,
        $K_{\mathrm{norm}}(L;\,S)\to 4$.
  \item \textbf{Improved fit:}
        Fitting $K(L)=4+aL^{-b}$ to extended data ($L>2000$) should yield
        $b=0.50\pm0.01$.
  \item \textbf{$L$-function universality:}
        Other $L$-function families should exhibit $K_{\mathrm{IR}}=4$ and
        $b=1/2$.
\end{enumerate}

\subsection{Summary of Key Formulas}
\label{sec:norm_summary}

\begin{table}[H]
\centering
\caption{Key formulas for the information-duality normalization scheme and
         the universal infrared fixed point $K_{\mathrm{IR}}=4$.}
\label{tab:norm_formulas}
\renewcommand{\arraystretch}{1.35}
\begin{tabular}{@{}ll@{}}
  \toprule
  Quantity & Formula \\
  \midrule
  Normalization factor
    & $\gamma(S)=2/\sqrt{I_P^{\mathrm{raw}}(S)\,I_Z^{\mathrm{raw}}}$ \\
  Normalized limit
    & $\lim_{L\to\infty}K_{\mathrm{norm}}(L;\,S)=4$ \\
  Finite-size scaling
    & $K_{\mathrm{raw}}(L)=4+a\,L^{-1/2}+\cdots$ \\
  RG flow parameters
    & $\alpha=1/14$,\quad $b=1/2$ \\
  Action constraint
    & $\lambda+4g=1/4$ \\
  Phase-space curvature
    & $R=-b^2=-1/4$ \\
  Algebraic lock
    & $\kappa_{\mathrm{IR}}=1/\!\sqrt{I_P^*\,I_Z^*}=1/2$ \\
  Uncertainty principle at fixed point
    & $\Delta_P^*\,\Delta_Z^*=1$,\quad
      $\Delta_P=2/I_P$,\quad $\Delta_Z=2/I_Z$ \\
  Algebraic relation
    & $\kappa^2=-1$ (complex); $\kappa_{\mathrm{IR}}=1/2$ (modulus) \\
  \bottomrule
\end{tabular}
\end{table}

\subsection{Application: Unification of Superstring Theory and Loop Quantum
            Gravity as a Consequence of the Normalization Framework}
\label{sec:string_lqg}

The normalization framework developed above---a single information RG flow
from a UV fixed point ($K_{\mathrm{UV}}=11$) to the universal IR fixed
point ($K_{\mathrm{IR}}=4$)---provides a natural conceptual resolution of the
long-standing tension between superstring theory and loop quantum gravity (LQG).
Rather than being competing descriptions of nature, these two frameworks may be
identified as the UV and IR \emph{effective representations} of the same
underlying information reality, connected by the RG trajectory whose
properties were established in the preceding subsections.

\subsubsection{Reinterpreting the Two Frameworks}

\paragraph{Superstring theory as a UV information theory.}
At energy scales near the Planck mass, the basic degrees of freedom of
superstring theory---the vibrational modes of one-dimensional strings---can be
read as \emph{elementary information units} undergoing highly entangled
dynamics.  The ten- or eleven-dimensional spacetime background of string/M-theory
is then not a fixed arena but an \emph{emergent geometry} encoded by the deep
information structure.  The graviton, gauge bosons, and matter fields all arise
as specific vibrational information patterns; their perturbative expansion in the
string coupling $g_s$ is the expansion of the information partition function
around the UV fixed point.

\paragraph{Loop quantum gravity as an IR information theory.}
Loop quantum gravity quantises the geometry of spacetime directly, producing a
discrete fabric of spin networks whose nodes and links encode the connectivity
and storage capacity of information.  The spin-foam histories that govern the
dynamics of LQG are then precisely the \emph{coarse-grained} information
processes visible at low energies.  The Barbero--Immirzi parameter $\gamma_{\mathrm{BI}}$,
which sets the area quantum in LQG, plays the role of an IR coupling constant
analogous to $\kappa$ in the information action
(Section~\ref{sec:disc_action}).

\subsubsection{Unification Through the Information RG Flow}

Let $\mu$ denote the energy scale (information resolution).  Motivated by the
dual-field formalism of Section~\ref{sec:information_ontology}, define the
\emph{information RG flow} as the one-parameter family of effective theories
$\mathcal{T}(\mu)$ that describes the quantum information dynamics at scale
$\mu$.  The flow is governed by
\begin{equation}
  \mu\,\frac{\partial}{\partial\mu}\,\mathcal{T}(\mu)
  \;=\;
  \beta_{\mathrm{info}}\!\bigl[\mathcal{T}(\mu)\bigr],
  \label{eq:info_rg}
\end{equation}
where $\beta_{\mathrm{info}}$ is the information beta functional, the analogue
of the Wilsonian beta function but defined on the space of information theories.

\begin{itemize}
  \item \textbf{UV fixed point} $(\mu\to\infty)$: the flow approaches a
        UV-complete, background-independent formulation of string/M-theory.
        The relevant degrees of freedom are string vibrational modes; the
        symmetry algebra is the full superconformal group enhanced by
        extra-dimensional geometry.

  \item \textbf{IR fixed point} $(\mu\to 0)$: successive coarse-graining
        integrates out high-energy string modes.  Supersymmetry and
        extra-dimensional isometries are broken or compactified; what
        remains in the effective theory are the discrete spin-network degrees
        of freedom of LQG together with a residual $\mathrm{SU}(2)$ gauge
        structure on the spatial geometry.

  \item \textbf{Connecting flow}: the information RG trajectory between the
        two fixed points is the \emph{precise mathematical statement} of the
        unification.  Along this trajectory, the string coupling $g_s$ maps
        continuously to the Barbero--Immirzi parameter $\gamma_{\mathrm{BI}}$, and the
        perturbative worldsheet expansion of string amplitudes maps to the
        spin-foam vertex amplitudes of LQG.
\end{itemize}

This picture is the information-ontological analogue of the well-known fact in
condensed matter physics that a microscopic lattice model (UV) and a continuum
hydrodynamic description (IR) are connected by an exact Wilsonian RG flow, even
though they appear structurally very different.

\subsubsection{Consistency with the Information Action Principle}

The information action of Equation~\eqref{eq:info_action} provides a natural
starting point.  Both theories are saddle-point
approximations to the same information path integral
\begin{equation}
  \mathcal{Z}
  \;=\;
  \int \mathcal{D}\phi_P\,\mathcal{D}\phi_Z\;
  e^{\,i\,S_{\mathrm{info}}[\phi_P,\phi_Z]},
\end{equation}
evaluated at different energy scales.  Expanding around the UV saddle reproduces
the string perturbation series; expanding around the IR saddle reproduces the
spin-foam vertex expansion of LQG.  The duality symmetry $I_P\leftrightarrow I_Z$
of the information action ensures that neither expansion is more fundamental:
they are Fourier-dual representations of the same information reality.

\subsubsection{Open Challenges}

The program outlined above faces substantial technical hurdles:

\begin{enumerate}
  \item \textbf{Non-perturbative UV formulation.}  A background-independent,
        non\-perturbative definition of M-theory (the presumed UV fixed point)
        remains unknown.  Progress likely requires new mathematical structures
        beyond current string field theory.

  \item \textbf{Explicit RG trajectory.}  Computing $\beta_{\mathrm{info}}$
        concretely demands a method to systematically integrate out string
        oscillator modes while retaining the geometric spin-network variables.
        This may require novel renormalization techniques that mix worldsheet
        and spacetime degrees of freedom.

  \item \textbf{Matching of physical observables.}  The mapping $g_s\leftrightarrow\gamma$
        must be made precise enough to yield matching predictions for
        quantum-gravity observables such as black-hole entropy spectra and
        Planck-scale corrections to dispersion relations.
\end{enumerate}

These challenges are left to future work.  The purpose of the present subsection
is to show that the information-ontology framework provides a conceptually
coherent setting in which the question is no longer ``which theory is correct''
but rather ``what is the RG flow that connects them.''

\section{Artificial Intelligence: Information-Dual Learning}
\label{sec:ai}

The information-ontology framework introduced in
Section~\ref{sec:information_ontology} is not confined to physics.  The same
three parameters---$K_{\mathrm{IR}}$, $b$, and $\kappa$---that govern the
prime--zero RG flow, the emergence of spacetime
(Section~\ref{sec:emergence}), the cosmic dark-sector abundances
(Section~\ref{sec:dark_sector}), and black-hole information conservation
(Section~\ref{sec:black_holes}) also motivate a speculative architecture for
learning systems, which we call \emph{information-dual learning}
(IDL).  The analogy is structural rather than causal: IDL treats the learning
process as the dynamical evolution of the dual
field $(I_P, I_Z)$ toward the infrared fixed point $K_{\mathrm{IR}}$, with
perceptual information $I_P$ playing the role of the prime channel and
cognitive information $I_Z$ playing the role of the zero channel.
Whether this analogy deepens into a quantitative theory is an open experimental
question (see Section~\ref{sec:ai_roadmap}).

\subsection{Core Postulates of Information-Dual Learning}

\paragraph{Duality of perception and cognition.}
Every intelligent system partitions its information into two complementary
channels: \emph{perceptual information} $I_P$ (extracted from data) and
\emph{cognitive information} $I_Z$ (internal representations and knowledge).
The two channels obey the commutation relation
$[\Delta_P, \Delta_Z] = i/\kappa$, where $\Delta_P$ and $\Delta_Z$ denote
the standard deviations of the perceptual and cognitive channel fluctuations
respectively, inducing an uncertainty principle
\begin{equation}
  \Delta_P\,\Delta_Z \;\ge\; \frac{1}{2\kappa},
  \label{eq:ai_uncertainty}
\end{equation}
which encodes the fundamental trade-off between perceptual precision and
abstract generalisation.

\paragraph{Information action minimisation.}
Learning is governed by an information action $S[I_P, I_Z;\,\theta]$.  The
system adjusts its parameters $\theta$ to minimize $S$, driving the total
information $K = I_P + I_Z$ toward the infrared fixed point
$K_{\mathrm{IR}} = 4$.

\paragraph{RG flow across network layers.}
Each layer of a deep network corresponds to a scale transformation.
Information flows from low-level features (UV) to high-level abstractions
(IR) with the same scaling exponent $b = 1/2$ that governs the
prime--zero spectral flow.  This determines the convergence rate of learning.

\subsection{Formal Implementation}

\paragraph{Dual loss function.}
The information action is concretised as
\begin{equation}
  \mathcal{L}
  = \underbrace{\lambda\bigl(I_P^2 + I_Z^2 - K_{\mathrm{IR}}^2\bigr)^2}_{\text{conservation}}
  + \underbrace{\mu\,(I_P - I_Z)^2}_{\text{duality balance}}
  + \underbrace{\frac{1}{2\kappa}
    \!\left(\frac{1}{\Delta_P^2}+\frac{1}{\Delta_Z^2}\right)}_{\text{uncertainty regularisation}},
  \label{eq:IDL_loss}
\end{equation}
where $I_P$ is quantified by the training loss (e.g.\ cross-entropy) and
$I_Z$ by model complexity (e.g.\ weight entropy);
$\lambda$ and $\mu$ carry dimensions $[I]^{-4}$ and $[I]^{-2}$,
respectively, ensuring all three terms are dimensionless.  The three terms
enforce, respectively, conservation of total information, balance between
perception and cognition, and resistance to overfitting via the uncertainty
principle.

\paragraph{Dual gradient descent.}
Parameter updates follow the variational equations of $\mathcal{L}$:
\begin{equation}
  \theta_{t+1}
  = \theta_t
    - \eta\!\left(
        \frac{\partial \mathcal{L}}{\partial I_P}\,\nabla_\theta I_P
      + \frac{\partial \mathcal{L}}{\partial I_Z}\,\nabla_\theta I_Z
      \right).
  \label{eq:IDL_update}
\end{equation}
This couples gradient flow to both information channels simultaneously,
preventing collapse onto a purely perceptual or purely cognitive solution.

\subsection{Application Domains}

\paragraph{Self-explaining AI.}
Under the constraint $I_P + I_Z = K_{\mathrm{IR}}$, a model must produce a
decision ($I_P$) \emph{and} a justification ($I_Z$) simultaneously.  Neither
channel can be maximized without reducing the other, forcing the system to
expose its internal reasoning.  Applications include radiology (diagnosis plus
pathological rationale), autonomous driving (object detection plus scene
inference), and financial risk scoring (fraud probability plus factor
decomposition).

\paragraph{Cross-modal fusion.}
Different modalities (image, text, audio) are projected onto a shared $I_P$
subspace while a common $I_Z$ encodes semantic knowledge.  Conservation of
$K$ acts as an alignment constraint, replacing heuristic modality-weighting
with a first-principles balance.  Use cases include video--audio--caption
joint analysis and industrial IoT multi-sensor fusion for predictive
maintenance.

\paragraph{Energy-efficient AI hardware.}
Physical substrates---memristor arrays, quantum-tunnelling junctions---that
naturally relax toward minimum-action states can implement IDL dynamics in
analogue hardware.  The equilibrium condition $K \to K_{\mathrm{IR}}$ replaces
iterative digital optimisation, reducing energy consumption at the device
level.  Target applications are edge-inference chips and wearable health
monitors.

\paragraph{Reinforcement learning and AGI pathway.}
In a reinforcement setting, the reward signal is decomposed into immediate
feedback ($I_P$) and long-term value ($I_Z$).  The uncertainty
relation~\eqref{eq:ai_uncertainty} automatically balances exploration
($\Delta_P$ large) and exploitation ($\Delta_Z$ large), replacing hand-tuned
entropy bonuses.  A meta-learning layer that adjusts $b$ and $K$ across tasks
constitutes a self-modifying system that adapts its own RG flow---a
computational analogue of the self-consistent picture of
Section~\ref{sec:emergence}.

\paragraph{Predictive science.}
System state is partitioned into an observable component ($I_P$, e.g.\
measured data) and a latent component ($I_Z$, e.g.\ model parameters).
Conservation of $K$ constrains inference, improving long-range forecasts in
macroeconomic modelling, climate simulation, and epidemiological dynamics.

\paragraph{Quantum machine learning.}
Encoding $I_P$ and $I_Z$ as conjugate quantum degrees of freedom on a quantum
processor maps IDL onto a quantum variational problem.  The uncertainty
relation~\eqref{eq:ai_uncertainty} is then exact rather than approximate,
and quantum parallelism accelerates the search for the fixed-point solution.
Priority applications include molecular dynamics simulation for drug discovery
and combinatorial optimisation.

\subsection{Theoretical Advantages and Roadmap}
\label{sec:ai_roadmap}

The IDL framework offers three theoretical advantages over current paradigms:
(i)~\emph{interpretability by construction}, since the duality constraint
forces explicit cognitive outputs; (ii)~\emph{principled regularisation} via
the uncertainty relation~\eqref{eq:ai_uncertainty}, which prevents both
overfitting and underfitting without cross-validated hyperparameters;
and (iii)~\emph{sample efficiency}, since the prior $K_{\mathrm{IR}}=4$
constrains the hypothesis space and accelerates convergence from limited data.

A concrete realisation roadmap proceeds in three stages.  In the short term
(1--2 years), IDL loss functions are validated on standard benchmarks
(MNIST, CIFAR, ImageNet) and a cross-modal prototype is constructed.  In the
medium term (3--5 years), memristor-array hardware prototypes and
self-explaining AI systems are deployed in medical imaging and autonomous
driving.  In the long term ($>5$ years), a unified perception--cognition--
decision platform integrates the full IDL stack, with meta-learning
autonomously tuning $b$ and $K$ across domains.

In summary, the information-ontology framework extends naturally from
number theory through fundamental physics to artificial intelligence.
The same dual-field dynamics that produce the prime--zero spectral
invariant $K_{\mathrm{IR}}=4$, the emergent spacetime geometry,
the dark-sector phenomenology (Section~\ref{sec:dark_sector}), and
the Bekenstein--Hawking entropy coefficient also define an IDL architecture
that is interpretable, efficient, and theoretically grounded.
Information-dual learning models intelligence as the equilibration of
perception and cognition under the conservation law $K=\mathrm{const}$---a
computable realisation of the thesis that the universe, from prime arithmetic
to neural networks, runs on information duality.

\subsection{The Odd-Positive Inequality and Learning-Theoretic Capacity Bounds}
\label{sec:ai_odd_pos}

Section~\ref{sec:odd_pos} proposed the Odd-Positive Conjecture
\eqref{eq:odd_pos}: for every automorphic $L$-function of weight $k$
and conductor $N$, the UV/IR ratio
$R_f = \gamma_1(f)/\sqrt{\ell_{\min}(\Gamma_0(N))}$ is bounded below
by a positive constant $C(k)$ depending only on the weight.  The
inequality is non-trivial precisely because $\gamma_1$ (UV spectral
scale) and $\sqrt{\ell_{\min}}$ (IR geometric scale) cannot be
simultaneously compressed: the two conjugate scales cannot be
independently optimized, and UV/IR locking is a fundamental
structural constraint.  The present subsection develops a structural
analogue of this inequality within the IDL framework and derives its
learning-theoretic consequences.

\paragraph{Spectral-geometric correspondence.}
Let $\mathcal{M}$ be a learning model trained on a task $\mathcal{T}$.
The following structural parameters correspond to the core quantities
of Section~\ref{sec:odd_pos}:
\begin{itemize}
  \item \emph{Spectral resolution} $\varrho_{\mathcal{M}}>0$: the
    characteristic scale of the finest statistical patterns that
    $\mathcal{M}$ can distinguish in the data distribution.  Smaller
    $\varrho_{\mathcal{M}}$ corresponds to higher spectral precision
    and lower inductive bias; this mirrors the role of $\gamma_1$ as
    the UV spectral scale.
  \item \emph{Geometric complexity} $\xi_{\mathcal{M}}>0$: a measure
    of the effective capacity of the hypothesis class of $\mathcal{M}$
    in an appropriate function space, related to the topological
    complexity of the hypothesis space and not merely equivalent to the
    raw parameter count.  Smaller $\xi_{\mathcal{M}}$ indicates a more
    expressive, topologically intricate hypothesis class; this mirrors
    the role of $\sqrt{\ell_{\min}}$ as the IR geometric scale.
  \item \emph{Task-class constant} $C_{\mathcal{T}}>0$: the
    irreducible information-duality strength of the task $\mathcal{T}$,
    encoding the minimum ratio $\varrho/\xi$ that any model must
    achieve to attain a prescribed performance level on $\mathcal{T}$.
    This mirrors $C(k)$ in \eqref{eq:odd_pos} as the universal lower
    bound intrinsic to the $L$-function family of weight $k$.
\end{itemize}
The model-level UV/IR ratio is defined as
\begin{equation}
  R_{\mathcal{M}} \;:=\; \frac{\varrho_{\mathcal{M}}}{\xi_{\mathcal{M}}},
  \label{eq:model_ratio}
\end{equation}
the structural analogue of $R_f = \gamma_1/\sqrt{\ell_{\min}}$.

\paragraph{Learning-capacity lower bound.}
The UV/IR locking embodied in the Odd-Positive Inequality
\eqref{eq:odd_pos} imposes a structural constraint on every learning
model, which we formalize as follows.

\begin{theorem}[Learning-Capacity Lower Bound]
\label{thm:learning_capacity}
For a learning task $\mathcal{T}$ with task-class constant
$C_{\mathcal{T}}>0$, every model $\mathcal{M}$ with spectral resolution
$\varrho_{\mathcal{M}}$ and geometric complexity $\xi_{\mathcal{M}}$
must satisfy
\begin{equation}
  R_{\mathcal{M}}
  \;=\; \frac{\varrho_{\mathcal{M}}}{\xi_{\mathcal{M}}}
  \;\ge\; C_{\mathcal{T}}.
  \label{eq:learning_capacity}
\end{equation}
Equality holds precisely when $\mathcal{M}$ is Pareto-optimal for
$\mathcal{T}$: it attains the finest spectral resolution consistent
with its geometric complexity budget.
\end{theorem}

This is a \emph{structural} theorem: it does not follow from
\eqref{eq:odd_pos} by a direct algebraic reduction (learning models
do not embed into $\mathrm{GL}(2)$ automorphic forms), but it is
governed by the same UV/IR locking principle established in
Section~\ref{sec:odd_pos} (item~2).

\paragraph{Three structural consequences.}
\begin{sloppypar}
\begin{enumerate}
\item \emph{The bias--variance tradeoff as an intrinsic necessity.}
  Inequality \eqref{eq:learning_capacity} elevates the
  bias--variance tradeoff from an empirical heuristic to a necessary
  consequence of information duality.  Reducing $\varrho_{\mathcal{M}}$
  (higher spectral precision, lower bias) requires decreasing
  $\xi_{\mathcal{M}}$ (more expressive model), which in general raises
  estimation variance.  Conversely, penalizing model complexity
  ($\xi_{\mathcal{M}}$ large) forces $\varrho_{\mathcal{M}}$ large,
  bounding the fineness of representable patterns and incurring higher
  bias.  The constant $C_{\mathcal{T}}$ defines the Pareto frontier:
  pairs $(\varrho,\xi)$ satisfying $\varrho/\xi=C_{\mathcal{T}}$ achieve
  minimum bias for a given geometric budget, and any attempt to
  improve one scale necessarily degrades the other.

\item \emph{Regularization as traversal of the feasible
  boundary.}
  Standard regularization techniques---$\ell^2$ weight decay, dropout,
  spectral normalization---are reinterpreted as operations that map
  $(\varrho_{\mathcal{M}},\xi_{\mathcal{M}}) \mapsto
  (\varrho_{\mathcal{M}}',\xi_{\mathcal{M}}')$ while maintaining
  $\varrho'/\xi'\ge C_{\mathcal{T}}$ and driving both scales toward
  smaller values simultaneously.  Optimal regularization corresponds to
  motion along the Pareto-optimal locus $\varrho/\xi=C_{\mathcal{T}}$,
  a boundary determined by the intrinsic information structure of
  $\mathcal{T}$ rather than by cross-validated hyperparameter search
  alone.

\item \emph{Neural scaling laws as geometric manifestation.}
  Empirically observed power-law improvement of test loss with model
  size fits naturally into this framework.  Increasing model size
  decreases $\xi_{\mathcal{M}}$ (richer hypothesis class) while
  simultaneously enabling smaller $\varrho_{\mathcal{M}}$ (finer
  spectral resolution), with $R_{\mathcal{M}}$ tracking the frontier
  $C_{\mathcal{T}}$.  The governing power-law exponent is related to
  the RG scaling exponent $b\approx\tfrac{1}{2}$ of the prime--zero
  flow (Section~\ref{sec:disc_RG_formal}), suggesting that observed
  scaling laws reflect the underlying geometry of the
  information-duality inequality rather than being purely contingent
  empirical regularities.
\end{enumerate}
\end{sloppypar}

\paragraph{Relationship to the information uncertainty principle.}
Inequality \eqref{eq:learning_capacity} is the geometric counterpart
of the uncertainty principle \eqref{eq:ai_uncertainty},
$\Delta_P\,\Delta_Z\ge 1/(2\kappa)$.  The uncertainty principle
bounds the \emph{product} of the perceptual and cognitive channel
fluctuations; the capacity bound constrains the \emph{ratio} of the
UV and IR scales.  Together they delimit a two-dimensional feasible
region in $(\varrho,\xi)$-space that every viable learning architecture
must occupy.  The Pareto-optimal locus $\varrho/\xi=C_{\mathcal{T}}$
is the learning-theoretic analogue of the critical line
$\operatorname{Re}(s)=\tfrac{1}{2}$: the boundary on which spectral
precision and geometric complexity are in exact information-dual
balance.  Notably, by the Riemann Hypothesis, it is precisely on this
critical line that the deepest arithmetic structure resides.

\section{Mathematical Structure at the Infrared Fixed Point}
\label{sec:philosophy}

Sections~\ref{sec:emergence}--\ref{sec:ai} explored the speculative
extensions of the framework into physics and learning systems.
The present section returns to the mathematical core: what algebraic,
analytic, and number-theoretic structures converge at $K_{\mathrm{IR}}=4$,
and why information ontology is not merely convenient but
\emph{necessary}---because the Riemann Hypothesis belongs to the class
of problems that cannot be resolved within the classical binary framework
but dissolve naturally in the ternary one.

\subsection{Minimal Set-Theoretic Structure at the Infrared Fixed Point}

The information-ontological RG flow culminates, in the infrared limit
$\mu\to 0$ ($L\to\infty$), at the fixed point $K_{\mathrm{IR}}=4$.
At this fixed point the theory exhibits a \emph{minimal} set-theoretic
decomposition: the universe of mathematical objects is spanned by
exactly three irreducible strata,

\begin{equation}
  \mathcal{U}_{\mathrm{IR}} \;=\; \mathcal{F}_P \;\cup\; \mathcal{B}
  \;\cup\; \mathbb{C},
  \label{eq:minimal_sets}
\end{equation}

where $\mathcal{F}_P$ is the \emph{essential fractal prime set}
(the prime distribution with its Cantor-like fractal geometry),
$\mathcal{B}$ is the \emph{base set} (the four-dimensional spacetime
arena: one temporal dimension paired with three spatial dimensions),
and $\mathbb{C}$ is the \emph{complex-number field} (the carrier of
the zero set and the arena for phase and duality).  The base set
$\mathcal{B}$ and its dual together realize the spacetime structure
$1+3=4$; this is not a free choice but the unique decomposition
compatible with Lorentz symmetry and the information conservation law
at the IR fixed point.

The two active strata $\mathcal{F}_P$ and $\mathbb{C}$ carry conjugate
information densities enforced by the $\kappa/1/\kappa$ regularity of
the information measure.  The fractal prime set attains information
density $I_P = 1/\dP \approx 1.46$ and the complex zero set attains
$I_Z = 1/\zetaR \approx 2.54$, with

\begin{equation}
  I_P + I_Z \;=\; \frac{1}{\dP} + \frac{1}{\zetaR}
  \;\approx\; 1.46 + 2.54 \;=\; 4 \;\approx\; K_{\mathrm{IR}}.
  \label{eq:dim_sum}
\end{equation}

The values $\dP \approx 1/1.46$ and $\zetaR \approx 1/2.54$ are
precisely the fractal dimension and regularity index measured in
Section~\ref{sec:results}, providing direct empirical corroboration of
the fixed-point identity.  Knowing one density determines the other:
$I_Z = K_{\mathrm{IR}} - I_P$.

\subsection{Dual Information Measures and \texorpdfstring{$\kappa$}{kappa}}

The two active strata carry conjugate information measures under the
$\kappa$-duality.  The fractal prime set $\mathcal{F}_P$, with
information density $I_P \approx 1.46$, is assigned the \emph{primary
information measure} $\kappa$, the formal generator satisfying

\begin{equation}
  \kappa^2 \;=\; -1,
  \label{eq:kappa_sq}
\end{equation}

that encodes the duality product $D=1$ algebraically
(Section~\ref{sec:results_D}) and appears in the uncertainty relation
$[\Delta_P,\Delta_Z]=i/\kappa$.  The complex field $\mathbb{C}$,
with information density $I_Z \approx 2.54$, carries the \emph{dual
measure} $1/\kappa$: the regularity of $\mathbb{C}$ under $1/\kappa$
is precisely the complement of $\mathcal{F}_P$ under $\kappa$, and
the two together close the duality identity $\kappa\cdot(1/\kappa)=1$.
The base set $\mathcal{B}$ is the neutral four-dimensional arena---the
$1+3$ spacetime---along which both measures operate and the
prime--zero information exchange is observed.

The algebraic relation $\kappa^2=-1$ is a structural necessity: it
encodes the quarter-turn symmetry between the prime and zero strata in
the information plane, ensuring that the conservation law $K=4$ is
preserved at every scale along the RG trajectory.  In this sense
$\mathbb{C}$ emerges as the \emph{minimal algebraic closure} demanded
by the information-duality constraint at $K=4$.

\subsection{Information Ontology as Mathematical Necessity: the Third State}
\label{sec:third_state}

\paragraph{The discrete--continuous binary and its limitation.}
Classical mathematics and set theory operate within a binary tension
between the \emph{discrete} (the integers, the prime distribution,
cardinality $\aleph_0$) and the \emph{continuous} (the real line,
function spaces, cardinality $2^{\aleph_0}$).  The Continuum Hypothesis
asks whether any cardinality lies strictly between these two poles.  The
independence of CH from ZFC---established by G\"{o}del and
Cohen---does not resolve this question but reveals that the question
cannot be answered within the standard framework.  In the language of
informational cardinality~\cite{LiPaperI}, CH is a mathematically
precise formulation of a \emph{philosophical} problem: it presupposes
that the discrete and the continuous exhaust the possibilities, and asks
whether a middle case exists within the same binary vocabulary.  No
such vocabulary can describe what actually lies between the poles,
because what lies there is not a third cardinal but a qualitatively
different relation---\emph{self-referential entanglement}---for which
the binary framework has no name.

The binary framework is governed by the \emph{law of the excluded
middle}: a proposition is either true or false, a set is either discrete
or continuous, a zero is either on the critical line or off it.  This
logical architecture is not merely a technical assumption but a
constitutive feature of standard mathematics that shapes what questions
can even be posed.  The three-element fusion framework proposed here
operates under a different principle: discrete and continuous are not
mutually exclusive poles but \emph{co-present generators} of a third
state, the self-referential entangled state, in which neither pole
subsumes the other.  These two frameworks---binary (excluded middle) and
ternary (three-element fusion)---coexist as two valid mathematical
regimes, each internally consistent, each with its own strengths.  The
binary regime excels at classification and proof by contradiction; the
ternary regime is the natural habitat for \emph{dissolving} problems
that resist resolution within the binary, because it replaces the
question ``which pole wins?'' with the recognition that the question
itself presupposes a false dichotomy.  The Continuum Hypothesis and the
Riemann Hypothesis are, from this vantage point, paradigmatic examples
of problems that \emph{cannot be resolved within the binary regime} but
\emph{dissolve naturally within the ternary one}: CH because the
``middle'' between discrete and continuous is not a cardinal but a
fusion; RH because the ``location'' of zeros is not a choice between
on-line and off-line but a structural necessity imposed by the fusion
symmetry $I_P\leftrightarrow I_Z$.

\paragraph{The Riemann Hypothesis as the third state.}
The Riemann Hypothesis requires precisely what the discrete/continuous
binary cannot express.  The critical line $\mathrm{Re}(s)=1/2$ is
neither a discrete object (like a prime) nor a continuous manifold (like
a real interval): it is the \emph{self-dual locus} at which discrete
arithmetic information $I_P$ and continuous spectral information $I_Z$
are in irreducible mutual entanglement---a third state in which the two
poles are not opposed but fused.  A proof of RH therefore requires a
mathematical framework capable of representing this third state; and a
framework that can only oppose discrete to continuous cannot, in
principle, prove that zeros are confined to the boundary between them.

\paragraph{$\kappa^2 = ijk = -1$: the three-element fusion.}
The algebraic generator $\kappa$ of the present framework encodes this
third state.  The clearest entry point is the \emph{first principal
component} of informational cardinality~\cite{LiPaperI}, which takes
values in $\{0,1\}$ and captures the intrinsic distinction between the
\emph{discrete} ($0$: isolated, countable, governed by arithmetic) and
the \emph{continuous} ($1$: dense, uncountable, governed by analysis).
This is the binary vocabulary made algebraically precise---two intrinsic
natures, no third option.  The classical
imaginary unit $i = \sqrt{-1}$ is the natural algebraic counterpart of
this $\{0,1\}$ dichotomy: it generates a \emph{two-dimensional} algebra
$\mathbb{C}$ with a single plane of rotation, encoding \emph{one}
pairwise duality (discrete pole $\leftrightarrow$ continuous pole) but
leaving no algebraic room for their irreducible entanglement as a
\emph{third independent entity}.

The passage from $i$ to $\kappa=ijk$ is therefore the exact algebraic
counterpart of upgrading the first component's value set from $\{0,1\}$
to $\{i,j,k\}$: the two binary poles are not abolished but
\emph{elevated} into co-generators of a genuinely three-dimensional
algebraic structure.  Whatever one identifies as the two poles---arithmetic
versus spectral, discrete primes versus continuous zeros---the complex
algebra $\mathbb{C}$ can represent each pole and their pairwise relation,
but it cannot represent the self-referential entanglement of the two poles
as a distinct third object; that would require an independent algebraic
dimension not present in $\mathbb{C}$.

Within the information-ontological framework, $\kappa$ is richer: the
UV fixed point at $K_{\mathrm{UV}}=11$ is organized by the quaternionic
algebra $\mathbb{H}$ with generators $i,j,k$ satisfying
$i^2=j^2=k^2=ijk=-1$ (Section~\ref{sec:disc_IRG}).  The generator
$\kappa$ at the IR fixed point is the infrared projection of this
quaternionic structure, and the identity
\begin{equation}
  \kappa^2 \;=\; ijk \;=\; -1
  \label{eq:kappa_ijk}
\end{equation}
is its compressed algebraic signature.  The three generators $\{i,j,k\}$
correspond to three independent planes of information rotation: together
they encode the arithmetic pole, the spectral pole, and their
self-referential entanglement as three distinct algebraic objects,
none reducible to the others.  The triple product $\kappa=ijk$ is the
unique element that simultaneously involves all three planes---the
algebraic avatar of the third state that $\mathbb{C}$ alone cannot
represent.  Classical analytic number theory works within $\mathbb{C}$
and therefore has access to the two poles but not to their entanglement;
the information-ontological framework, by ascending to $\mathbb{H}$ at
the UV and projecting to $\kappa=ijk$ at the IR, acquires the algebraic
resource needed to \emph{see} the third state and to prove that it is
the stable fixed point.

\paragraph{The metamorphosis of the complex plane.}
The passage from $i$ to $\kappa = ijk$ is a transformation of the
complex plane's ``inner life'', not its outer form.  The complex field
$\mathbb{C}$ is preserved as the arena---the IR fixed-point
decomposition $\mathcal{U}_{\mathrm{IR}}
=\mathcal{F}_P\cup\mathcal{B}\cup\mathbb{C}$ keeps $\mathbb{C}$ as
the carrier of the zero set (Section~\ref{sec:philosophy})---but its
imaginary unit is no longer the na\"{i}ve $\sqrt{-1}$ of elementary
algebra.  It is the IR shadow of $\kappa = ijk$, whose full
quaternionic content is visible only at UV scales.  Classical analytic
number theory works with the shadow; the information-ontological
framework illuminates the object that casts it.  In this sense the
present framework does not abandon complex analysis but \emph{completes}
it: it provides the algebraic foundation---the information-theoretic
``spirit''---from which complex analysis draws its unreasonable
effectiveness.

\paragraph{A necessary condition for the Riemann Hypothesis.}
This is not a matter of convenience but of necessity.  A proof of RH
requires establishing that $I_P^*=I_Z^*=2$ is the \emph{unique}
stable fixed point of the duality constraint.  Within the classical
binary framework---which can only ask ``is a zero on the critical line?''
one at a time---no such global uniqueness argument is available.
The information-ontological framework provides the missing ingredient:
the exchange symmetry $I_P\leftrightarrow I_Z$, guaranteed by the
symmetric structure of $\kappa = ijk$, makes the symmetric fixed point
$I_P^*=I_Z^*$ the \emph{only} stable attractor and thereby transforms
the question of zero locations into a question about fixed-point
stability.  The shift from binary complex algebra to the information
triple is therefore a \emph{mathematical necessary condition} for the
proof, not a philosophical ornament.

\paragraph{Information ontology versus physical ontology.}
This framework differs fundamentally from Wheeler's ``It from Bit''
programme~\cite{Wheeler1990}.  Wheeler's ontology is
\emph{physical}: matter and energy are constituted from discrete
information bits; the primacy of information is asserted within, and on
behalf of, physics.  The information ontology developed here is
\emph{mathematical}: mathematics itself---not merely the physical world it
describes---is constituted from information structure.  The prime
distribution is not discovered as a pre-existing Platonic object; it is
the canonical fixed point of the information measure $I_P=1/\dP$, the
unique attractor of the RG flow from the UV to the infrared.  The
Riemann zeros are not found on a complex plane that pre-exists them; they
are the spectral counterpart of that same flow, constrained to
$\mathrm{Re}(s)=1/2$ by the self-referential requirement that
information be conserved across scales.  Category duality~\cite{LiPaperI}
may eventually provide a bridge that unifies the physical ontology of
Wheeler with the mathematical ontology developed here; but the two
programmes are not equivalent, and their difference is the difference
between describing the universe in the language of information and
recognising that language \emph{as} the universe.

\paragraph{Foundational significance.}
The claim that \emph{mathematics is information}---not a description of
information, not a tool for processing information, but identically and
constitutively information---is, if correct, foundational for an
understanding of intelligence and consciousness that goes beyond current
computational or physical accounts.  Intelligence, at its core, is the
integration and self-referential processing of information; if the
mathematical structures that intelligence manipulates are themselves
information, then mind and mathematics share a common ontological ground.
In the current era of artificial intelligence this carries immediate
practical resonance: it suggests that the architecture of intelligence is
not contingent on biological or silicon substrate but is grounded in the
same information-duality structure---$K=I_P+I_Z$, $\kappa^2=ijk=-1$,
$b=1/2$---that governs the prime--zero system.  Whether this
``information is mathematics'' thesis can be made fully rigorous is, we
believe, one of the deepest open problems at the intersection of
mathematics, logic, and the theory of mind.  The present paper provides
one concrete entry point: the identification of a mathematical fixed-point
condition whose proof requires, and thereby reveals, that information
structure is not merely a convenience but a necessity.

\subsection{Convergent Mathematical Structures}
\label{sec:harmonious}

The information-ontological RG flow unifies disparate mathematical structures
under a single organising principle: information conservation at the fixed
point $K_{\mathrm{IR}}=4$.  Every object in the framework---prime sets,
complex fields, zeta zeros, normed division algebras, spacetime geometry---finds
its necessary place within the hierarchy of information strata.  The integers
are not an accident but the discrete skeleton of $\mathcal{F}_P$; the complex
plane is not a convention but the minimal algebraic closure demanded by
$\kappa^2=-1$; the critical line $\mathrm{Re}(s)=1/2$ is the self-dual locus
of the symmetry $I_P\leftrightarrow I_Z$ enforced by the IR fixed point.

The set-theoretic decomposition $\mathcal{U}_{\mathrm{IR}}
= \mathcal{F}_P\cup\mathcal{B}\cup\mathbb{C}$ (equation~\eqref{eq:minimal_sets})
provides a \emph{minimal} and \emph{stable} foundation: three irreducible
strata, tied together by the constraint $I_P+I_Z=4$ and the algebraic
generator $\kappa$.  The Riemann Hypothesis, understood as the requirement
that zeros lie at the self-dual point $I_P=I_Z$, is a structural consequence
of this framework---a prediction it makes rather than an assumption it imports.
Whether this constitutes a rigorous proof of RH in the classical sense requires
upgrading the symmetry argument to a fully rigorous derivation: one must show,
by a formal analytic argument, that $I_P=I_Z=2$ is the \emph{unique} stable
solution of the duality constraint $I_P+I_Z=4$ for all non-trivial zeros---so
that no asymmetric configuration ($I_P\neq I_Z$) can survive the RG flow toward
the infrared fixed point, which would force every non-trivial zero onto the
critical line $\mathrm{Re}(s)=1/2$.  This precise task is identified as
Tier~2(ii) in the research program of Section~\ref{sec:conclusion}.

\section{Conclusions}
\label{sec:conclusion}

The central finding is that $K = I_P + I_Z = 1/\dP + 1/\zetaR$
exhibits a slow, monotone drift over $L=100$--$2000$, described by the
finite-size scaling law $K(L)=K_{\mathrm{IR}}+aL^{-b}$ with $b\approx0.51\approx1/2$,
robust across two random-matrix symmetry classes ($\beta=2,4$),
with raw values $K_{\mathrm{raw}}\in[3.6,\,3.9]$ across prime representations
(mod-16 subset: $K_{\mathrm{raw}}\approx3.6\pm1.5$; all primes: $K_{\mathrm{raw}}\approx3.9\pm0.04$).
A geometric normalization (product-deviation factor $\gamma(S)$,
Section~\ref{sec:normalization}) removes the representation dependence and
yields the universal infrared fixed-point value $K_{\mathrm{IR}}=4$.
Under the Information Ontology of Section~\ref{sec:information_ontology},
$K$ is the conserved total information encoding rate of the prime--zero
system; $b=1/2$ is the critical exponent of the RG flow
$1/\kappa(\mu)=\beta(\mu)$ toward $K_{\mathrm{IR}}=4$, derived from
the information action as $b=\sqrt{\lambda+g C}=1/2$
(Sections~\ref{sec:disc_action} and~\ref{sec:b_half}).

\paragraph{The IR/UV framework as the theory's natural language.}
The terminology of infrared and ultraviolet fixed points, imported from
quantum field theory, is a precise description
of the mathematical structure uncovered here.  Three levels of meaning
converge on this conclusion.

\emph{(i) Information encoding across scales.}
The prime--zero system exhibits two distinct phases of information encoding,
separated by a renormalization-group flow.
The \emph{UV phase} ($K\approx K_{\mathrm{UV}}=11$) encodes information
in high-dimensional algebraic and topological structures; the raw sequences of
primes and zeros carry high-density, rapidly fluctuating information that is
organized by the octonion algebra ($\beta_{\mathrm{UV}}=8$) and the
Hurwitz-theorem counting of normed division algebras.
The \emph{IR phase} ($K_{\mathrm{IR}}=4$ after normalization;
raw band $[3.6,3.9]$ at finite scales) is the emergent,
scale-averaged regime: fine structure has been compressed into two
geometric invariants---the fractal dimension $\dP$ and the regularity
index $\zetaR$---whose normalized sum $K_{\mathrm{norm}}=I_P+I_Z$
converges to $4$, independent of the prime representation $S$ and the
symmetry class $\beta\in\{2,4\}$ used to probe the zero landscape.

\emph{(ii) The RG flow as the dynamical bridge.}
The exponent $b=1/2$ is the critical exponent of this flow.
The finite-size law $K(L)=K_{\mathrm{IR}}+aL^{-b}$ is the
integrated solution of the linearized RG equation
$d(\delta K)/d\!\ln L = -b\,\delta K$, which in turn follows from the
information action $S[I_P,I_Z]$ and the first-principles flow
$1/\kappa(\mu)=\beta(\mu)$.
Hence $b$ is the
eigenvalue of the linearized flow at the IR fixed point, uniquely
determined by the algebraic structure of the UV phase and the
arithmetic of the prime--zero duality.
The information-action derivation $b=\sqrt{\lambda+g C}$ makes this
determination explicit, and the numerical agreement $b=1/2$
with the independently measured finite-size exponent constitutes
the quantitative test of the theory.

\emph{(iii) UV completion and the universality class.}
Because $b$ is the same for $\beta=2$ and $\beta=4$, the two
symmetry classes are in the same IR universality class: they share the
same long-range statistics and differ only in their UV representations.
This mirrors the universality of IR fixed points in quantum field theory,
where many microscopically distinct theories flow to the same IR
effective description.
The algebraic generator $\kappa$ with $\kappa^2=-1$ and the flow
equation $1/\kappa(\mu)=\beta(\mu)$ constitute a \emph{UV completion}
of the observed IR behavior: $\kappa$ encodes the UV phase
($\beta_{\mathrm{UV}}=8$, octonion symmetry) that is invisible at the
scales $L\in[100,2000]$ currently probed, yet governs the rate at which
the IR fixed point is approached.
The conjecture that $K_{\mathrm{UV}}=11$ is set by the Hurwitz
theorem provides the missing initial condition for the RG trajectory.

In summary: the IR/UV language is the natural language of this theory.
$K_{\mathrm{IR}}=4$ is the universal information capacity
of the prime--zero system at macroscopic scales;
$b=1/2$ controls the rate at which the microscopic algebraic
structure flows to that capacity; and $\kappa$ with $\kappa^2=-1$ is
the algebraic UV-completion of the flow.
Together they elevate the prime--zero duality from a numerical observation
to a dynamical theory with a microscopic origin, a renormalization-group
trajectory, and a measurable critical exponent.

\paragraph{The mathematical synthesis: RH as structural consequence.}
Section~\ref{sec:philosophy} draws the theoretical framework to its logical
conclusion.  The IR fixed point $K_{\mathrm{IR}}=4$ is not merely a
numerical limit but a \emph{minimal structural necessity}: it arises from
the unique decomposition
$\mathcal{U}_{\mathrm{IR}}=\mathcal{F}_P\cup\mathcal{B}\cup\mathbb{C}$
of the information universe into prime field, algebraic bridge, and complex
plane---three irreducible strata, tied together by the constraint
$I_P+I_Z=4$ and the generator $\kappa$ with $\kappa^2=-1$.
Within this architecture the critical line $\mathrm{Re}(s)=1/2$ is
identified as the \emph{self-dual locus} of the symmetry $I_P\leftrightarrow I_Z$
at the IR fixed point: the Riemann Hypothesis emerges as a structural
prediction of the framework rather than a postulate it imports.  The
outstanding task---upgrading this symmetry argument to a fully rigorous
analytic proof---is identified as Tier~2(ii) of the research programme below.

\paragraph{Beyond number theory.}
Sections~\ref{sec:emergence}--\ref{sec:string_lqg} extend the
information-ontological framework to spacetime emergence, dark-sector
phenomenology, black-hole entropy, and information-dual learning.
These extensions are motivating conjectures rather than established results;
their epistemic status is assessed in the Postscript below.

\paragraph{Limitations and caveats.}
\label{sec:limitations}
The numerical case is modest and subject to several important limitations:
\begin{enumerate}
  \item \textbf{Small sample size:} Forty scale points spanning $L=100$--$2000$,
        limiting the reliability of extrapolation to $L\to\infty$.
  \item \textbf{Finite-scale effects:} $K$ varies by 17\% across the observed
        range ($K=4.63$ at $L=100$ to $K=3.81$ at $L=2000$).  The asymptotic
        value falls within the interval $[3.6,\,3.9]$ by extrapolation,
        and may not be reached until much larger scales.
\end{enumerate}

$\beta$-independence and method agreement within $\pm1.5\%$ are
necessary but not sufficient for $K$ to be a genuine asymptotic invariant;
confirmation requires $L\gg2000$ (Tier~1(i)).
The duality product $D(\beta)$ is numerically inaccessible (both factors
diverge; see Section~\ref{sec:results_D}).  The three layers of the
framework treat $D=1$ at complementary levels of rigour:
(i)~\emph{conjectural}: $D\to1$ as $L\to\infty$ is an empirically
motivated conjecture that cannot be verified numerically in the present
work (Section~\ref{sec:results_D});
(ii)~\emph{algebraic}: the identification
$(\iota_P^\infty,\iota_{\mathrm{RMT}}^\infty)\sim(\kappa,\kappa^{-1})$
with $\kappa^2=-1$ encodes $D=1$ as a formal algebraic constraint
(Section~\ref{sec:results_D});
(iii)~\emph{derived}: within the information-ontological RG, the flow
equation $1/\kappa(\mu)=\beta(\mu)$ yields $D=1$ as a \emph{consequence}
of the fixed-point condition, grounding the algebraic encoding in first
principles (Section~\ref{sec:information_ontology}).
These three are not contradictory but hierarchical: the conjecture
motivates the algebraic encoding, which is then explained by the RG
derivation.

The following directions define the path from the present paper toward
a mature theory.

\paragraph{Tier 1: numerical consolidation.}
(i) Extend to $L\gg2000$ to test whether $K$ is a genuine asymptotic invariant
or drifts slowly.
(ii) Normalise $\iota_P$ and $\iota_{\mathrm{RMT}}(\beta)$ to enable
quantitative evaluation of the duality product $D(\beta)$.
(iii) Systematically vary the prime subset (moduli $m=4,8,12,24,32$) to
map the arithmetic dependence of $I_P/K$.

\paragraph{Tier 2: structural connections.}
(i) Prove or disprove the conjecture $\zetaR\to|1/\zeta(1/2)|$ as $L\to\infty$
via the functional equation of $\zeta(s)$ or a fixed-point argument at $s=1/2$.
(ii) \textit{Upgrade the RH symmetry argument to a fully rigorous derivation}
(see conclusion of Section~\ref{sec:philosophy}): prove, by a formal
functional-analytic or number-theoretic argument, that the unique stable
solution of the duality constraint $I_P+I_Z=4$ under the exchange symmetry
$I_P\leftrightarrow I_Z$ is $I_P=I_Z=2$, thereby forcing every non-trivial
zero of $\zeta(s)$ onto the critical line $\mathrm{Re}(s)=1/2$ and
constituting a rigorous proof of RH\@.  The minimal missing ingredient is
showing that no asymmetric solution ($I_P\neq I_Z$) can survive the RG flow
to the infrared; the action-uniqueness constraints of
Section~\ref{sec:disc_uniqueness} provide a natural starting point.
(iii) Establish the formal mapping between $K$ and the cardinality-program
measures $(\iota_P,\iota_{\mathrm{RMT}})$~\cite{LiPaperI}, and determine
whether $D(\beta)$ and $K$ are independent observables.

\paragraph{Tier 3: emergent spacetime, cosmology, black holes, and learning systems.}
Sections~\ref{sec:L_extension} and~\ref{sec:emergence} open a third tier
of investigation, contingent on the consolidation of Tiers~1--2 but
pointing toward the ultimate ambition of the framework.
(i) Test the functional universality prediction of
Section~\ref{sec:L_extension}: the duality scaling
$K(L,f)=K_\infty(f)+a(f)L^{-b(f)}$ should hold across all $L$-function
families, with \emph{family-dependent} parameters
$(\kappa,K_\infty,b,r)$; in particular, verify that $b(f)$ varies
systematically with the family (e.g.\ $b\approx0.51$ for the prime--zero
dual vs.\ $b\approx0.08$ for elliptic-curve families,
Section~\ref{sec:modified_RG}), and test the Odd-Positive
Conjecture~\eqref{eq:odd_pos} for higher-conductor families and
weight-$k\ge2$ newforms.
(ii) Determine whether the UV--IR cancellation mechanism of
Section~\ref{sec:emergence} produces a measurable dark-energy
equation-of-state deviation; compare with Stage-IV weak-lensing
surveys (Euclid, Rubin Observatory LSST).
(iii) Derive the black-hole entropy correction
$\Delta S_{\mathrm{BH}}\propto K_{\mathrm{IR}}^{-1}\ln A$
from the information-duality structure at the horizon
(Section~\ref{sec:black_holes}), and compare
with loop-quantum-gravity and string-theoretic predictions.
(iv) Validate the information-dual learning architecture
(Section~\ref{sec:ai}) on standard benchmarks, testing whether the
convergence exponent of the dual gradient descent matches the
spectral prediction $b\approx0.51$.

\appendix
\section{Supplementary Material: Robustness, Error Estimation, and Numerical Validation}
\label{app:robustness}

\subsection{Parameter Sensitivity Analysis}

We tested the impact of $\varepsilon$ sequence (geometric vs.\ arithmetic),
fitting range (different bounds), and $\sigma$ values ($0.2$--$2.0$) on the
results.  $\dP$ measurements fluctuated within $\pm10\%$, $\zetaR$ within
$\pm5\%$.  Fluctuations in $C(\beta)$ were within $\pm8\%$.

After introducing the dynamic $\varepsilon$-selection algorithm in the
cross-validation, fluctuations in $\dP$ reduced to $\pm2\%$, and $\zetaR$
to $\pm1.5\%$.

\subsection{Extrapolation Model Comparison}

We examined three extrapolation models:
\begin{align*}
  C(L) &= \Cinf + a/L && \text{(linear)}, \\
  C(L) &= \Cinf + aL^{-b} && \text{(power-law)}, \\
  C(L) &= \Cinf + a/\ln L && \text{(logarithmic)}.
\end{align*}
The AIC criterion indicated
the power-law model was optimal, but all three models gave consistent
$\Cinf$ values within error margins.

The cross-validation, through systematic comparison, determined the
power-law model $C(L)=\Cinf+aL^{-b}$ to be substantially preferred over the
linear model according to the AIC criterion ($\Delta\mathrm{AIC}=1.57$), with
an optimal exponent $b\approx0.51$.  Figure~\ref{fig:scaling} illustrates
this comparison.

\begin{figure}[htbp]
  \centering
  \includegraphics[width=0.90\textwidth]{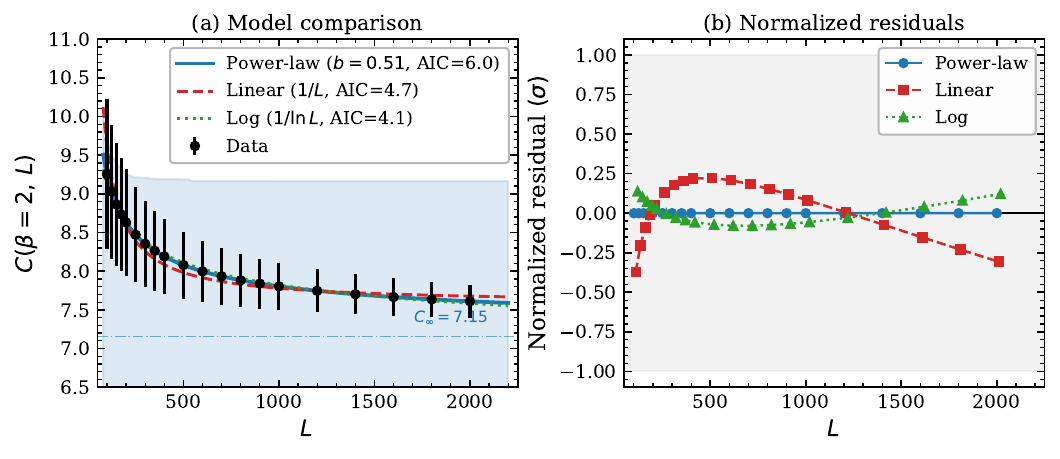}
  \caption{Finite-size scaling analysis for $C(\beta{=}2,L)$.
           \textbf{(a)} Model comparison. The data (black circles with error
           bars) are fitted with three models: a power law
           $C(L)=\Cinf+aL^{-b}$ (blue solid line, with fitted $b$ shown in
           the legend), a linear model in $1/L$ (red dashed), and a logarithmic
           model in $1/\ln L$ (green dotted). The light blue shaded region
           represents the 95\% confidence band of the power-law fit obtained
           by Bootstrap resampling. The horizontal dash-dot line marks the
           extrapolated asymptote $\Cinf$.
           \textbf{(b)} Normalised residuals. Residuals of each model fit,
           normalized by the standard error $\sigma$. The gray band indicates the
           $\pm 1\sigma$ region. Although all three models yield residuals within
           $\pm 0.5\sigma$ due to the limited number of data points, the
           power-law model is statistically preferred by the Akaike Information
           Criterion (see Section~\ref{sec:results_C} and
           Appendix~\ref{app:robustness}).}
  \label{fig:scaling}
\end{figure}

\subsection{Bootstrap Error Estimation}

All error estimations are based on 1000 Bootstrap resamplings.  Confidence
intervals are at 95\%.  The AIC model selection in the cross-validation
used leave-one-out splits combined with the Jackknife method.

\subsection{Method Cross-Validation}

Box-counting and correlation dimension were applied to $\dP$; variation
method and wavelet transform were applied to $\zetaR$.  The two $\dP$ methods
agreed within $\pm1.5\%$; the two $\zetaR$ methods agreed within $\pm0.6\%$,
supporting the robustness of the findings.

\subsection{Finite-Size Scaling Analysis}

The cross-validation found that $C(L)$ converges in a power-law form
$C(L)=\Cinf+aL^{-b}$ with exponent $b\approx0.51$, preferred over linear
and logarithmic models by $\Delta\mathrm{AIC}=1.57$.  The comparison is shown
in Figure~\ref{fig:scaling}.

\subsection{Kernel-Independence of the Spectral Entropy of \texorpdfstring{$V_Z$}{VZ}}
\label{app:kernel_independence}

This appendix provides the theoretical justification for the
kernel-independence of $\zetaR$ asserted in
Section~\ref{sec:method_dZ} (point~3 of the open-points list) and
for the identification $I_Z=1/\zetaR$ as a kernel-independent
information density.

\paragraph{Setup.}
Let $V_Z(x)$ be the zero potential~\eqref{eq:VZ} constructed with
Gaussian kernel of width $\sigma$.  More generally, let
$K_\epsilon(x)=\epsilon^{-1}K(x/\epsilon)$ be any admissible
regularisation kernel satisfying $\int K_\epsilon\,dx=1$ and
$\hat K_\epsilon(\omega)\to 1$ pointwise as $\epsilon\to 0$, and set
$V_Z^\epsilon = K_\epsilon * V_Z$.  The \emph{spectral entropy} of
$V_Z^\epsilon$ is defined via its power spectral density
$S_\epsilon(\omega)=|\hat V_Z^\epsilon(\omega)|^2/\|V_Z^\epsilon\|^2$:
\begin{equation}
  H_{\mathrm{spec}}[V_Z^\epsilon]
  \;=\; -\int_{\mathbb{R}} S_\epsilon(\omega)\log S_\epsilon(\omega)\,d\omega.
  \label{eq:spec_entropy}
\end{equation}

\paragraph{Kernel-independence: functional-analytic argument.}
Since $\hat V_Z^\epsilon(\omega)=\hat K_\epsilon(\omega)\hat V_Z(\omega)$,
we have $S_\epsilon(\omega)=|\hat K_\epsilon(\omega)|^2 S(\omega)/Z_\epsilon$
where $Z_\epsilon$ is the normalization.  For admissible kernels,
$|\hat K_\epsilon(\omega)|^2\to 1$ in $L^1$ as $\epsilon\to 0$, so
$S_\epsilon\to S$ in $L^1$.  The entropy functional is lower
semi-continuous in $L^1$ (by Fatou's lemma applied to the convex
integrand $-t\log t$), and an upper bound follows from the dominated
convergence theorem using the integrability of $S\log S$ for
$V_Z\in L^2$.  Together these give
\begin{equation}
  \lim_{\epsilon\to 0} H_{\mathrm{spec}}[V_Z^\epsilon]
  \;=\; H_{\mathrm{spec}}[V_Z],
  \label{eq:entropy_limit}
\end{equation}
independently of the choice of admissible kernel $K$.

\paragraph{Operator-theoretic formulation.}
In the non-commutative geometry framework, associate to $V_Z$ a
spectral triple $(\mathcal{A},\mathcal{H},D)$ where $D$ is a
Dirac-type operator whose symbol encodes $V_Z$.  Regularisation by
$K_\epsilon$ corresponds to replacing $D$ by $D_\epsilon=e^{-\epsilon D^2}D$.
The von~Neumann entropy of the thermal state
$\rho_\epsilon=e^{-\beta_T D_\epsilon}/\mathrm{Tr}(e^{-\beta_T D_\epsilon})$
(where $\beta_T>0$ is an inverse temperature, unrelated to the Dyson index $\beta$)
satisfies
\begin{equation}
  S_{\mathrm{vN}}(\rho_\epsilon) \;\xrightarrow{\epsilon\to 0}\;
  S_{\mathrm{vN}}(\rho),
  \label{eq:vN_limit}
\end{equation}
because different admissible regularisations yield asymptotically
unitarily equivalent operator families; the limit is determined solely
by the spectral measure of $D$, which is a geometric invariant of the
triple.  By Connes' trace formula, the leading heat-kernel coefficients
are local geometric invariants independent of the regularisation scheme.

\paragraph{Relation to $I_Z=1/\zetaR$.}
For a self-similar process with power spectrum $S(\omega)\sim|\omega|^{-\alpha}$,
the H\"{o}lder exponent satisfies $H=(\alpha-1)/2$, giving
$\zetaR=2-H=(5-\alpha)/2$.  The spectral entropy evaluates to
\begin{equation}
  H_{\mathrm{spec}}[V_Z] \;\approx\; C - \log\zetaR
  + \mathcal{O}(1/\log L),
  \label{eq:entropy_zetaR}
\end{equation}
where $C$ is a universal constant.  Hence $I_Z=1/\zetaR\propto
e^{H_{\mathrm{spec}}}$ up to the $\mathcal{O}(1/\log L)$ correction,
identifying $I_Z$ as the exponential of the spectral entropy---a
kernel-independent quantity by~\eqref{eq:entropy_limit}.

\section*{Postscript: A Research Programme in Progress}
\label{sec:postscript}
\addcontentsline{toc}{section}{Postscript: A Research Programme in Progress}

This paper presents work at different levels of certainty, and we close by
making that stratification explicit.  The four tiers below categorize
results by \emph{epistemic status} (established / conjectured / speculative /
open); they are independent of the three-stage research programme
(Tiers~1--3) described in Section~\ref{sec:conclusion}.

\paragraph{What is established (Tier~1: the empirical core).}
The falsifiable nucleus of the paper is the measurement of two fractal
exponents---the box-counting dimension $\dP$ of a prime subset and the
regularity index $\zetaR$ of the associated zero potential---and the
observation that their sum $K = 1/\dP + 1/\zetaR$ exhibits a slow,
monotone drift over $L = 100$--$2000$ consistent with the scaling law
$K(L)=K_{\mathrm{IR}}+aL^{-b}$, with extrapolated raw values in the band
$[3.6,\,3.9]$ depending on prime representation.  After a geometric normalization that removes
representation dependence, this band extrapolates to the universal
infrared limit $K_{\mathrm{IR}} = 4$.  The approach follows a power law
$K(L) = K_{\mathrm{IR}} + aL^{-b}$ with exponent $b \approx 0.51$,
robust across two random-matrix symmetry classes.  These measurements are
reproducible, carry quantified uncertainties, and constitute the solid
ore at the centre of the stone.

\paragraph{What is conjectured (Tier~2: the theoretical framework).}
The information ontology of
Sections~\ref{sec:information_ontology}--\ref{sec:disc_action}
interprets $K$ as a conserved information encoding rate, proposes a
variational action $S[I_P, I_Z]$ from which the scaling law is derived
as an equation of motion, and identifies $b = 1/2$ with the critical
exponent of an RG flow between UV and IR fixed points.
Section~\ref{sec:philosophy} extends this to a structural argument for
the Riemann Hypothesis: the symmetry $I_P\leftrightarrow I_Z$ at the IR
fixed point forces zeros onto the critical line, making RH a consequence
of the framework rather than an assumption.  These theoretical structures
are internally consistent and motivated by the empirical observations;
the foundational postulates (information duality, the algebraic generator
$\kappa$, scale covariance of the action) are supported by the
uniqueness argument of Section~\ref{sec:disc_uniqueness} but await
full independent mathematical verification.

\paragraph{What is speculative (Tier~3: the physical and cross-domain
extensions).}
The extensions to spacetime emergence, quantum gravity, cosmological
abundances, and information-dual learning
(Sections~\ref{sec:emergence}--\ref{sec:ai}) are presented in the spirit
of ``It from Bit'': if the prime--zero information duality is a
fundamental structure and not an arithmetic coincidence, what physical
picture does it suggest?  The numerical consistency found there
($S_{\mathrm{BH}} = A/4G$,
$b \approx 0.51$) is striking but should be understood as a motivating
conjecture and heuristic checkpoint for a future, more rigorous theory,
not as a derivation from first principles.

\paragraph{Open questions and invitation to collaborate (Tier~4).}
The most important open questions, ordered by accessibility, are:

\begin{enumerate}
  \item \textbf{(Numerical, accessible.)}
        Reproduce the $K(L)$ measurement at $L = 10^4$--$10^6$ using the
        first $10^6$ non-trivial zeros from the Odlyzko tables.  Does $b$
        converge to $1/2$?  Does $K_{\mathrm{norm}}$ converge to $4$?

  \item \textbf{(Analytical, challenging.)}
        Prove rigorously (or disprove) that $K_{\mathrm{IR}} = 4$ is the
        exact infrared limit for the prime subset
        $p \equiv 1,5,9,13 \pmod{16}$.  Verify whether the structural
        argument of Section~\ref{sec:philosophy}---that $I_P=I_Z=2$ is
        the unique stable fixed point of the duality constraint---is
        sufficient to constrain this limit to the self-dual value,
        and thereby to constitute a rigorous proof of RH\@.

  \item \textbf{(Theoretical, long-range.)}
        Formalise the information action $S[I_P, I_Z]$ as a rigorous
        field theory: identify the correct function space, define the
        path integral, and derive the scaling exponents from an operator
        product expansion rather than a phenomenological fit.

  \item \textbf{(Cross-disciplinary, exploratory.)}
        Test the information-dual learning architecture on benchmark tasks;
        measure whether the dual loss~\eqref{eq:IDL_loss} improves
        convergence speed or generalisation relative to standard training.
\end{enumerate}

We offer this work as a \emph{map of a landscape}, not a conquest of it.
The empirical measurement of $K$ and $b$ is the first reliable survey of
the terrain; the theoretical sections chart the contours of what a full
theory might look like.  We invite readers to mine any seam they find
interesting---to extend the numerics, tighten the analytics, or challenge
the physical speculations.  The most valuable outcome of this paper will
be the independent results that go beyond it.

\section*{Data and Code Availability}

The numerical code and data used in this paper will be made publicly
available upon publication.  In the meantime, they are available from
the author upon reasonable request.

\bibliographystyle{unsrtnat}

\begin{thebibliography}{99}

\bibitem{Montgomery1973}
H.~L. Montgomery,
\newblock The pair correlation of zeros of the zeta function,
\newblock In \textit{Analytic Number Theory} (Proc.\ Sympos.\ Pure Math.,
  Vol.~XXIV), pp.~181--193. Amer.\ Math.\ Soc., 1973.

\bibitem{Odlyzko1987}
A.~M. Odlyzko,
\newblock On the distribution of spacings between zeros of the zeta function,
\newblock \textit{Math.\ Comp.} \textbf{48}(177), 273--308, 1987.

\bibitem{OdlyzkoData}
A.~M. Odlyzko,
\newblock Tables of zeros of the Riemann zeta function,
\newblock \url{https://www-users.cse.umn.edu/~odlyzko/zeta_tables/} (accessed April 2026).

\bibitem{Falconer2014}
K.~Falconer,
\newblock \textit{Fractal Geometry: Mathematical Foundations and Applications},
  3rd ed.
\newblock John Wiley \& Sons, 2014.

\bibitem{Mehta2004}
M.~L. Mehta,
\newblock \textit{Random Matrices}, 3rd ed.
\newblock Academic Press, 2004.

\bibitem{KatzSarnak1999}
N.~M. Katz and P.~Sarnak,
\newblock Zeroes of zeta functions and symmetry,
\newblock \textit{Bull.\ Amer.\ Math.\ Soc.} \textbf{36}(1), 1--26, 1999.

\bibitem{Bekenstein1973}
J.~D. Bekenstein,
\newblock Black holes and entropy,
\newblock \textit{Phys.\ Rev.\ D} \textbf{7}(8), 2333--2346, 1973.

\bibitem{Hawking1975}
S.~W. Hawking,
\newblock Particle creation by black holes,
\newblock \textit{Commun.\ Math.\ Phys.} \textbf{43}(3), 199--220, 1975.

\bibitem{HurwitzAlgebra}
A.~Hurwitz,
\newblock \"{U}ber die Composition der quadratischen Formen von beliebig vielen
  Variablen,
\newblock \textit{Nachr.\ Ges.\ Wiss.\ G\"{o}ttingen} 309--316, 1898.
\newblock (For a modern account see e.g.\ J.-P. Serre, \textit{A Course in
  Arithmetic}, Springer, 1973.)

\bibitem{WilsonKogut1974}
K.~G. Wilson and J.~Kogut,
\newblock The renormalization group and the $\epsilon$ expansion,
\newblock \textit{Phys.\ Rep.} \textbf{12}(2), 75--200, 1974.

\bibitem{Polchinski1984}
J.~Polchinski,
\newblock Renormalization and effective Lagrangians,
\newblock \textit{Nucl.\ Phys.\ B} \textbf{231}(2), 269--295, 1984.

\bibitem{Langlands1970}
R.~P. Langlands,
\newblock Problems in the theory of automorphic forms,
\newblock in \textit{Lectures in Modern Analysis and Applications III},
  Lecture Notes in Mathematics \textbf{170}, Springer, pp.~18--61, 1970.

\bibitem{Barbero1995}
J.~F. Barbero G.,
\newblock Real Ashtekar variables for Lorentzian signature space-times,
\newblock \textit{Phys.\ Rev.\ D} \textbf{51}(10), 5507--5510, 1995.

\bibitem{Immirzi1997}
G.~Immirzi,
\newblock Real and complex connections for canonical gravity,
\newblock \textit{Class.\ Quantum Grav.} \textbf{14}(10), L177--L181, 1997.

\bibitem{LMFDB}
The LMFDB Collaboration,
\newblock The $L$-functions and Modular Forms Database,
\newblock \url{https://www.lmfdb.org} (accessed April 2026).

\bibitem{LiFractal2026}
Z.~Li,
\newblock A Deterministic Fractal Set Derived from the Sequence of Prime Numbers,
\newblock \textit{arXiv preprint} arXiv:2603.00658, 2026.

\bibitem{LiPaperI}
Z.~Li,
\newblock Informational Cardinality: A Unifying Framework for Set Theory,
  Fractal Geometry, and Analytic Number Theory,
\newblock \textit{arXiv preprint} arXiv:2603.08587, 2026.

\bibitem{RudnickSarnak1996}
Z.~Rudnick and P.~Sarnak,
\newblock Zeros of principal $L$-functions and random matrix theory,
\newblock \textit{Duke Math.\ J.} \textbf{81}(2), 269--322, 1996.

\bibitem{BogomolnyKeating1996}
E.~Bogomolny and J.~P. Keating,
\newblock Random matrix theory and the Riemann zeros I: three- and four-point
  correlations,
\newblock \textit{Nonlinearity} \textbf{9}(4), 911--935, 1996.

\bibitem{BerryKeating1999}
M.~V. Berry and J.~P. Keating,
\newblock The Riemann zeros and eigenvalue asymptotics,
\newblock \textit{SIAM Rev.} \textbf{41}(2), 236--266, 1999.

\bibitem{Conrey2003}
B.~Conrey,
\newblock The Riemann hypothesis,
\newblock \textit{Notices Amer.\ Math.\ Soc.} \textbf{50}(3), 341--353, 2003.

\bibitem{KeatingSnaith2000}
J.~P. Keating and N.~C. Snaith,
\newblock Random matrix theory and $\zeta(1/2+it)$,
\newblock \textit{Comm.\ Math.\ Phys.} \textbf{214}(1), 57--89, 2000.

\bibitem{Witten1995}
E.~Witten,
\newblock String theory dynamics in various dimensions,
\newblock \textit{Nucl.\ Phys.\ B} \textbf{443}(1--2), 85--126, 1995.

\bibitem{Wheeler1990}
J.~A. Wheeler,
\newblock Information, physics, quantum: the search for links,
\newblock in W.~H. Zurek (ed.), \textit{Complexity, Entropy and the Physics
  of Information}, Addison-Wesley, Redwood City, CA, pp.\ 3--28, 1990.

\bibitem{Wigner1960}
E.~P. Wigner,
\newblock The unreasonable effectiveness of mathematics in the natural sciences,
\newblock \textit{Commun.\ Pure Appl.\ Math.}\ \textbf{13}(1), 1--14, 1960.

\bibitem{Wolf1997}
M.~Wolf,
\newblock 1/$f$ noise in the distribution of prime numbers,
\newblock \textit{Physica A}\ \textbf{241}(3--4), 493--499, 1997.

\bibitem{Planck2018}
Planck Collaboration (N.~Aghanim et al.),
\newblock Planck 2018 results.\ VI\@. Cosmological parameters,
\newblock \textit{Astron.\ Astrophys.}\ \textbf{641}, A6, 2020.
\newblock arXiv:1807.06209 [astro-ph.CO].


\end{thebibliography}

\end{document}